\documentclass[11pt,reqno]{amsart}

\usepackage[usenames,dvipsnames]{xcolor}
\usepackage{amsmath,amsthm,verbatim,amssymb,amsfonts,amscd, graphicx, enumerate, enumitem,dsfont,sidecap,wrapfig,tikz-cd,stmaryrd}
\usepackage[hidelinks]{hyperref}
\usepackage{graphics}
\usepackage{mathtools}
\usepackage[font=small,labelfont=bf,textfont=md,margin=12pt]{caption}
\usepackage{mathrsfs}
\usepackage{parskip,float}
\usepackage[numbers]{natbib}
\setcitestyle{open={[},close={]}}

\hypersetup{colorlinks=true,linkcolor=MidnightBlue,citecolor=MidnightBlue, urlcolor=black}

\usepackage{array,graphicx,enumitem,tikz-cd,hyperref,longfbox,multirow,comment,wrapfig,diagbox,hhline,colortbl}

\renewcommand{\paragraph}[1]{\subsection*{#1}}
\newcommand{\sidecaption}{}
\newcommand{\nocontentsline}[3]{}
\newcommand{\tocless}[2]{\bgroup\let\addcontentsline=\nocontentsline#1{#2}\egroup}

\usepackage[margin=1.4in,headsep=.4in, top=1.35in, bottom=1.35in]{geometry}

\usepackage{amssymb}
\usepackage{xcolor}
\usepackage{graphicx}

\definecolor{MutedBlue}{RGB}{40,90,160}

\definecolor{cellgray}{gray}{0.875}

\let\int\relax
\newcommand{\int}{\mathring}

\DeclareMathOperator{\id}{{id}}

\newcommand{\acal}{\mathcal{A}}
\newcommand{\dcal}{\mathcal{D}}
\newcommand{\fcal}{\mathcal{F}}
\newcommand{\cc}{\mathbb{C}}
\newcommand{\rr}{\mathbb{R}}
\newcommand{\zz}{\mathbb{Z}}
\newcommand{\ff}{\mathbb{F}}
\newcommand{\chh}{\mathcal{C}}
\newcommand{\hh}{\mathcal{H}}
\newcommand{\rcal}{\mathcal{R}}

\newcommand{\Mod}{\boldsymbol{\operatorname{Mod}}}

\newcommand{\Cobt}{\boldsymbol{\operatorname{Cob}^3}}
\newcommand{\Cobl}{\boldsymbol{\operatorname{Cob}^3_{/l}}}

\newcommand{\Cobtd}{\boldsymbol{\operatorname{Cob}^3_{\bullet}}}
\newcommand{\Cobld}{\boldsymbol{\operatorname{Cob}^3_{\bullet/l}}}

\newcommand{\Kobh}{\mathbf{Kob_{/h}}}
\newcommand{\Link}{\boldsymbol{\operatorname{Link}}}

\newcommand{\compfont}[1]{\textsf{\small#1}}

\newcommand{\lb}{\llbracket}
\newcommand{\rb}{\rrbracket}

\newcommand{\lmr}{\fontfamily{lmr}\selectfont}
\newcommand{\Tu}{\lmr \textit{4{\kern-0.75pt}Tu}}

\newcommand{\hfl}{\operatorname{HFL}}
\newcommand{\hflm}{\operatorname{HFL}^{\kern-1.5pt -}}
\newcommand{\hflc}{\operatorname{HFL}^{\kern-1pt \circ}}
\newcommand{\cfl}{\operatorname{CFL}}
\newcommand{\cflm}{\operatorname{CFL}^{\kern-1.5pt -}}

\newcommand{\hfk}{\operatorname{HFK}}

\newcommand{\bfo}{\mathbf{1}}
\newcommand{\bfx}{\mathbf{x}}
\newcommand{\bfy}{\mathbf{y}}
\newcommand{\kh}{\operatorname{Kh}}
\newcommand{\ckh}{\operatorname{CKh}}
\newcommand{\bn}{\operatorname{BN}}
\newcommand{\cbn}{\operatorname{CBN}}

\newcommand{\var}{U}

\newcommand{\wh}{\operatorname{Wh}}

\newcommand{\crossing}{\raisebox{-.2\height}{\includegraphics[scale=.2]{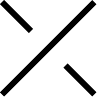}} \null}
\newcommand{\zsmooth}{\raisebox{-.2\height}{\includegraphics[scale=.2]{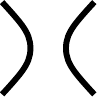}}\null}
\newcommand{\osmooth}{\raisebox{-.2\height}{\includegraphics[scale=.2]{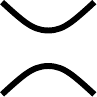}}\null}
\newcommand{\kcap}{\includegraphics[scale=.2,trim=0 .2cm  -0.2cm 0]{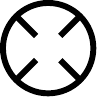}\null}

\newcommand{\kcup}{\includegraphics[scale=.15,trim=0 .2cm  -0.2cm 0]{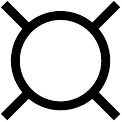}\null}

\newcommand{\ksmooth}{\includegraphics[scale=.2,trim=0 .2cm -0.2cm 0]{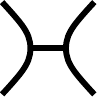}\null}

\newcommand{\ksquare}{\includegraphics[scale=.2,trim=-0.25cm .2cm 0.5cm 0]{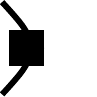}\null}

\newcommand{\sphere}{\hspace{0.5pt} \raisebox{-.25\height}{\includegraphics[scale=.315]{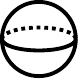}}\null}
\newcommand{\torus}{\hspace{0.25pt} \raisebox{-.3\height}{\includegraphics[scale=.315]{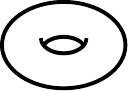}\hspace{0.25pt}}\null}

\newtheoremstyle{plain}
  {5pt plus3pt minus3pt}   
  {5pt plus3.3pt minus3.3pt}   
  {\itshape}  
  {0pt}       
  {\bfseries} 
  {.}         
  {5pt plus 1pt minus 1pt} 
  {}          

\newtheoremstyle{thm}{2.5 pt}{2.5 pt}{\itshape}{}{\bfseries}{.}{.5em}{}
\theoremstyle{thm}
\newtheorem{theorem}{Theorem}[section]
\newtheorem*{theorem*}{Theorem} 

\newtheorem{lemma}[theorem]{Lemma}
\newtheorem{proposition}[theorem]{Proposition}

\newtheorem*{question*}{Question} 
 
\newtheorem*{thom*}{Theorem \thetheorem~(Thom Conjecture)}

\newtheorem*{mingenusprob*}{The Minimal Genus Problem}

\newtheoremstyle{def}{2.5 pt}{2.5 pt}{}{}{\bfseries}{.}{.5em}{}
\theoremstyle{def}
\newtheorem{definition}[theorem]{Definition}
\newtheorem{remark}[theorem]{Remark}

\newtheoremstyle{example}
  {2pt plus3.3pt minus3.6pt}   
  {2pt plus3.3pt minus3.3pt}   
  {}  
  {0pt}       
  {\bfseries} 
  {.}         
  {5pt plus 1pt minus 1pt} 
  {}          
  \theoremstyle{example}

\newtheorem{example}[theorem]{Example}
\newtheorem{exercise}[theorem]{Exercise}

\graphicspath{{Figures/hayden/}}

\author{\vspace{-.1in}\small Kyle Hayden}

\title{\large Lecture notes on link homologies and knotted surfaces}

\usepackage[affil-it]{authblk} 
\usepackage{etoolbox}
\usepackage{lmodern}
\makeatletter
\patchcmd{\@maketitle}{\LARGE \@title}{\fontsize{16}{19.2}\selectfont\@title}{}{}
\makeatother

\usepackage{fancyhdr}
\begin{document}

%
%
%
%
%
%
%
%
%


\renewcommand{\thesection}{\arabic{section}}

%
%
%
%
%


\

\vspace{-.10in}

\begingroup
\def\uppercasenonmath#1{}
\let\MakeUppercase\relax
\maketitle
\endgroup
\thispagestyle{empty}

%

\vspace{-.15in}


\begin{center} \begin{minipage}{.85\linewidth}\footnotesize

\textsc{Abstract.}  Link homology theories (such as knot Floer homology and Khovanov homology) have become indispensable tools for studying knots and links, including powerful 4-dimensional obstructions. These notes, based on lectures given at the 2024 Georgia Topology Summer School,  discuss what these toolkits say about surfaces in 4-space themselves, via the homomorphisms assigned to link cobordisms. We begin with a brief overview of these theories (focusing on their shared formal properties) and survey some of their applications to knotted surfaces. Afterwards, we give an introduction to Khovanov homology (with an eye towards its cobordism maps), discuss hands-on computational techniques for Khovanov and Bar-Natan homology, and outline the role of the Bar-Natan category in this story. 

\end{minipage}
\end{center}




\bigskip
\bigskip
\bigskip

\section{Introduction and survey}\label{hsec:intro}

\smallskip

In these lecture notes, we will be interested in \emph{link homology theories}. For the purposes of this first lecture, we will avoid fussing with details in order to convey the bigger picture. Formally, we will view a link homology theory $\hh$ as a functor $\Link^+ \to \Mod_\rcal$, where $\Link^+$ is the category whose objects are oriented links $L \subset S^3$ and whose morphisms are oriented link cobordisms in $S^3 \times I$, and $\Mod_\rcal$ is the category of $\rcal$-modules.  (In these lectures, we most often let $\rcal$ be $\ff_2$, $\zz$, or a polynomial ring over one of these.) 
Here a \textit{link cobordism} $\Sigma: L_0 \to L_1$ is a smooth, compact, oriented, properly embedded surface $\Sigma$ in $S^3 \times I$  whose boundary is a pair of oriented links  $(-S^3,-L_0) \sqcup  (S^3,L_1)$. Note that these theories are assumed to be covariant functors, so that if $\Sigma = \Sigma_0 \circ \Sigma_1$, then
$$\hh(\Sigma)=\hh(\Sigma_0) \circ \hh(\Sigma_1).$$ 
As outlined below, these lectures are primarily concerned with Khovanov homology, but this introductory lecture also discusses link Floer homology. 

\smallskip

\paragraph{\textbf{Organization}}

These lectures are organized as follows: In the rest of \S\ref{hsec:intro}, we briefly introduce Khovanov homology and link Floer homology, summarize the basic functoriality properties of these link homology theories, and then survey a handful of applications of these to the study of knotted surfaces. In \S\ref{sec:khovanov}, we re-introduce Khovanov homology in greater detail and discuss techniques for computing the Khovanov cobordism maps, including many examples and exercises. In \S\ref{sec:tqft}, we introduce Bar-Natan homology and Bar-Natan's perspective on Khovanov homology, which enables us to give further applications and hint at the proofs of invariance (for both the homology theory itself as well as the cobordism maps).

\smallskip

\paragraph{\textbf{Conventions}}

 All surfaces will be smoothly embedded. Unless otherwise specified or made clear from context, submanifolds $\Sigma \subset X$ will be assumed to be neatly embedded, i.e., with $\partial \Sigma \subset \partial X$.

\smallskip

\paragraph{\textbf{Acknowledgements}}

The author thanks Yikai Teng and Isaac Sundberg for their important contributions to this lecture series. Some of the topics and exercises included here are influenced by various lectures, notes, and insights from Jesse Cohen, Robert Lipshitz, Kim Morrison, Taketo Sano, and Melissa Zhang, and the author is grateful to them for these valuable resources. Thanks also to the  anonymous referee for their helpful suggestions. Finally, thank you to the organizers of the 2024 Georgia Topology Summer School and Conference for their hospitality and for putting together  an excellent program. This work was supported in part by NSF grant DMS-2243128. 

\bigskip

\subsection{Link homology theories}

For the study of knotted surfaces, the most relevant link homology theories are Khovanov homology and link Floer homology. We write these as  $\hh=\kh$ and $\hh=\hfl$, respectively; these are umbrella terms, as each of these theories admit various specific flavors of invariants. 

We warn the reader that this overview will not attempt any sort of mathematically precise and consistent treatment of link homology theories in the abstract; instead, we will use $\hh$ as a placeholder when describing an unspecified link homology theory. This will be useful because different link homology theories (in all their various flavors) often share  structural similarities. For example, the theories discussed here all admits $\zz \oplus \zz$-bigradings
$$\hh(L) = \bigoplus_{(i,j) \in \zz \oplus \zz} \hh_{i,j}(L)$$
where one grading is \emph{homological} and the other grading is more closely tied to topological information.

Below we give a brief overview of these theories,  focusing on 
\begin{enumerate}[leftmargin=1.25cm,itemsep=-4pt]
\item describing the generators of the underlying chain complexes, with an aim of conveying the different flavor of these theories, and \smallskip

\item providing enough language to briefly survey the state of the art.
\end{enumerate} Throughout, we suggest references where the reader can find more detail.



\medskip

\subsection*{Khovanov homology}

 Khovanov homology was introduced in \cite{Khovanov}, with important variations of the theory introduced by Lee \cite{lee} and Bar-Natan \cite{barnatan}. Its construction can be understood as a lift of the Kauffman state sum model for the Jones polynomial \cite{Kauffman} to a chain complex. 
 
 This  construction takes as input a link diagram $\dcal$ with $n$ crossings and associates to it a ``cube of resolutions'' consisting of the $2^n$ diagrams obtained by resolving the crossings in all possible ways;  see Figure~\ref{fig:tref-complex-intro} for an example with the trefoil. In Khovanov's original theory, the associated chain complex $\ckh(\dcal)$ is built by replacing each resolved diagram (called a \emph{smoothing}) with a copy of $\acal^{\otimes k}$, where $k$ is the number of connected components  in the resolved diagram and  $\acal$ is a certain rank-2 Frobenius algebra $\acal=\mathcal{R}\langle \mathbf{1}, \mathbf{x} \rangle$. Therefore, one may view the generators of $\ckh(\dcal)$ as smoothings of $\dcal$ where each connected component  is labeled with a choice of $\bfo,\bfx \in \acal$.  The differential on $\ckh(\dcal)$ is defined so that the edges of the cube of resolution become component arrows of the differential in an appropriate way, pointing from a given smoothing $\sigma$ to a subset of those  smoothings that differ from $\sigma$ at exactly one crossing. 
 

The resulting homology theory splits along a \emph{homological} grading $h$ and a \emph{quantum} grading $q$, so that we may write
$$ \kh(L) = \bigoplus_{(h,q)}  \kh^{h,q}(L).$$
Various choices in the construction (especially in terms of the underlying Frobenius algebra) yield variants of Khovanov's original theory. As alluded to above, in \S\ref{sec:tqft} we will  consider one such variant introduced by Bar-Natan, denoted $\bn(L)$.

\begin{figure}\center
\includegraphics[width=.9\linewidth]{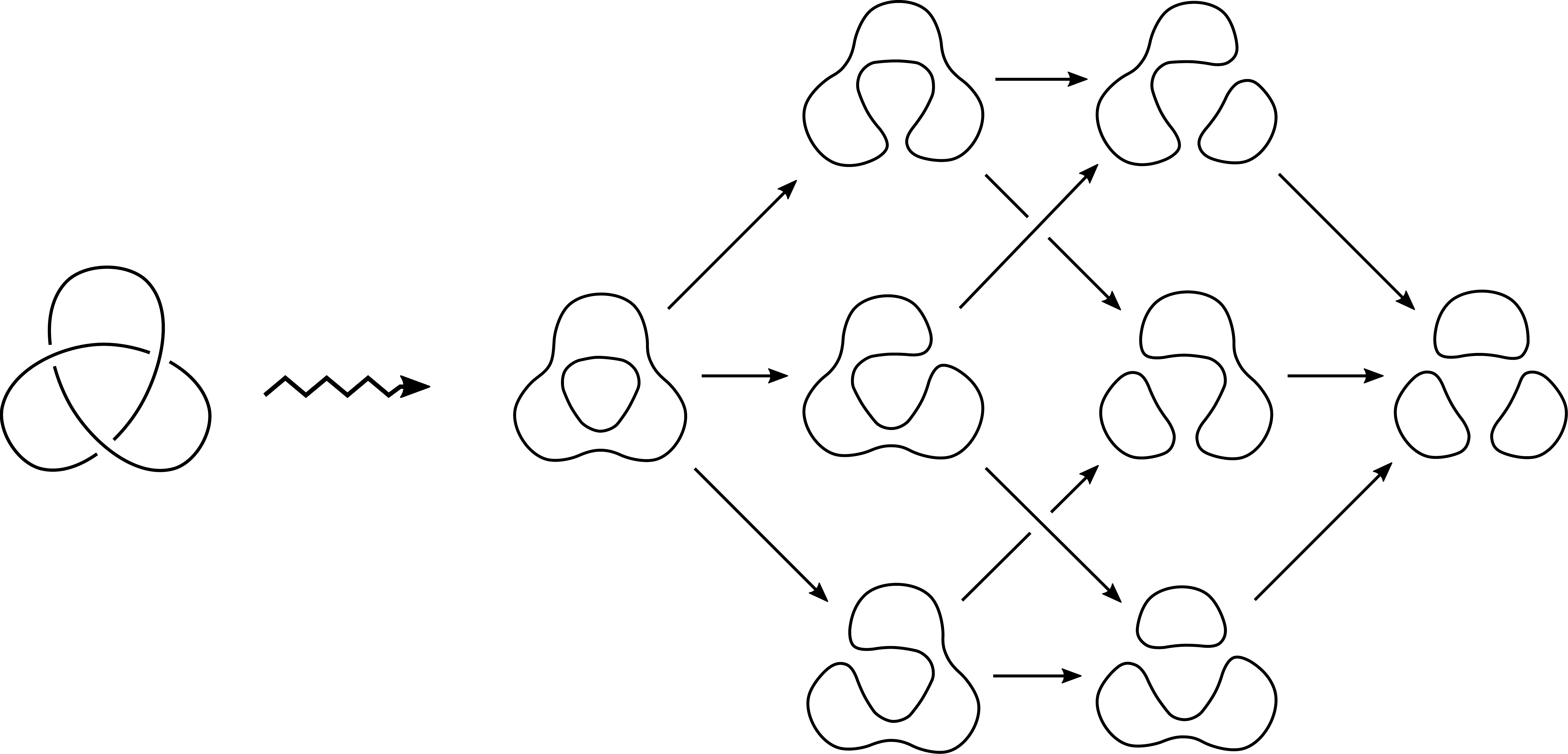}
\caption{The cube of resolutions for a diagram of the trefoil.}\label{fig:tref-complex-intro}
\end{figure}



%

\medskip \paragraph{\textbf{Link Floer homology}}

Heegaard Floer homology was introduced by Ozsv\'ath--Szab\'o in \cite{oz-sz:three}, with link Floer homology introduced shortly after by Ozsv\'ath--Szab\'o \cite{oz-sz:hfk} and Rasmussen \cite{rasmussen:thesis}. This theory comes in a variety of flavors. For our purposes, the relevant ones are the \emph{hat} flavor $\widehat \hfl(L)$ over $\ff_2$ and the richer \emph{minus} flavor $\hflm(L)$ over $\ff_2[U,V]$. When discussing properties common to both flavors, we will use $\hfl^\circ$ to refer to both theories simultaneously. For example,  both theories decompose along a \emph{Maslov} grading $m$ and an \emph{Alexander} grading $a$:
$$
\hflc(L) = \bigoplus_{(m,a)} \hflc_m(L,a)
$$



In the case of a knot $K\subset S^3$, the necessary input is a \emph{doubly-pointed Heegaard diagram},  which is a tuple $\dcal=(\Sigma,\alpha,\beta,w,z)$ where
\begin{enumerate}[label=(\roman*),itemsep=-2pt]
\item $(\Sigma, \alpha , \beta)$ is a Heegaard diagram for $S^3$ with collections of compressing curves $\alpha=(\alpha_1,\ldots,\alpha_g)$ and $\beta=(\beta_1,\ldots,\beta_g)$ for $g=g(\Sigma)$,
\item $w$ and $z$ are basepoints in $\Sigma \setminus (\alpha \cup \beta)$,
\item $K$ is obtained from a union of (oriented) arcs $a \subset \Sigma  \setminus \alpha$ and $b = \Sigma \setminus \beta$, where $a$ goes from $w$ to $z$ and $b$ goes from $z$ to $w$, by pushing these arcs' interiors slightly into the $\alpha$- and $\beta$-handlebodies, respectively.\footnote{Therefore, with $\Sigma$ oriented as the boundary of the $\alpha$-handlebody, the oriented knot $K$ intersects $\Sigma$ positively at $z$ and negatively at $w$.}
\end{enumerate}
An example of a doubly-pointed Heegaard diagram for the trefoil is given in Figure~\ref{fig:heeg-tref}. This admits a natural generalization to multi-pointed Heegaard diagrams for oriented links in $S^3$, where each component has at least one pair of $w$ and $z$ basepoints.

\begin{figure}[tb]
\center
\def\svgwidth{.8\linewidth}
\begingroup%
  \makeatletter%
  \providecommand\color[2][]{%
    \errmessage{(Inkscape) Color is used for the text in Inkscape, but the package 'color.sty' is not loaded}%
    \renewcommand\color[2][]{}%
  }%
  \providecommand\transparent[1]{%
    \errmessage{(Inkscape) Transparency is used (non-zero) for the text in Inkscape, but the package 'transparent.sty' is not loaded}%
    \renewcommand\transparent[1]{}%
  }%
  \providecommand\rotatebox[2]{#2}%
  \newcommand*\fsize{\dimexpr\f@size pt\relax}%
  \newcommand*\lineheight[1]{\fontsize{\fsize}{#1\fsize}\selectfont}%
  \ifx\svgwidth\undefined%
    \setlength{\unitlength}{567.69034799bp}%
    \ifx\svgscale\undefined%
      \relax%
    \else%
      \setlength{\unitlength}{\unitlength * \real{\svgscale}}%
    \fi%
  \else%
    \setlength{\unitlength}{\svgwidth}%
  \fi%
  \global\let\svgwidth\undefined%
  \global\let\svgscale\undefined%
  \makeatother%
  \begin{picture}(1,0.35102938)%
    \lineheight{1}%
    \setlength\tabcolsep{0pt}%
    \put(0,0){\includegraphics[width=\unitlength,page=1]{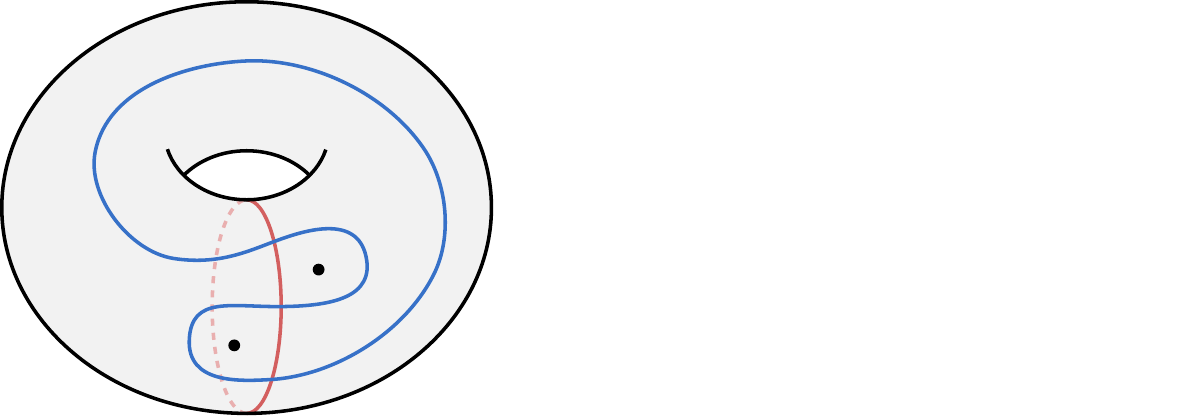}}%
    \put(0.2127958,0.06816674){\color[rgb]{0,0,0}\makebox(0,0)[t]{\smash{\begin{tabular}[t]{c}$z$\end{tabular}}}}%
    \put(0.2846849,0.13185808){\color[rgb]{0,0,0}\makebox(0,0)[t]{\smash{\begin{tabular}[t]{c}$w$\end{tabular}}}}%
    \put(0.21672683,0.10991397){\color[rgb]{0.75686275,0.20392157,0.20392157}\makebox(0,0)[t]{\smash{\begin{tabular}[t]{c}$\alpha_1$\end{tabular}}}}%
    \put(0.09005472,0.13681754){\color[rgb]{0.18823529,0.38039216,0.67058824}\makebox(0,0)[t]{\smash{\begin{tabular}[t]{c}$\beta_1$\end{tabular}}}}%
    \put(0,0){\includegraphics[width=\unitlength,page=2]{heeg-diagrams.pdf}}%
    \put(0.66856689,0.05541622){\color[rgb]{0,0,0}\makebox(0,0)[t]{\smash{\begin{tabular}[t]{c}$b$\end{tabular}}}}%
    \put(0.65040981,0.12066827){\color[rgb]{0.4,0.4,0.4}\makebox(0,0)[t]{\smash{\begin{tabular}[t]{c}$a$\end{tabular}}}}%
    \put(0,0){\includegraphics[width=\unitlength,page=3]{heeg-diagrams.pdf}}%
  \end{picture}%
\endgroup%

\caption{A doubly-pointed Heegaard diagram for the trefoil (left) and a depiction of the corresponding trefoil (right).}\label{fig:heeg-tref}
\end{figure}

To such a diagram one associates a chain complex $\widehat\cfl(\dcal)$ whose generators can be understood diagrammatically as  (unordered) $g$-tuples of intersection points  $x=[x_1,\ldots,x_g]$ where $x_i \in \alpha_i \cap \beta_{\sigma(i)}$ for some permutation $\sigma$ of $\{1,\ldots,g\}$; this ensures that each $\alpha$-curve and each $\beta$-curve contribute exactly one point to the generator $x$.  Geometrically, these generators correspond to intersection points between the $g$-dimensional tori $$\mathbb{T}_\alpha=\alpha_1 \times \cdots \times \alpha_g \qquad \text{ and } \qquad \mathbb{T}_\beta=\beta_1 \times \cdots \times \beta_g$$ in the symmetric product $\operatorname{Sym}^g(\Sigma)=\Sigma^{\times g} / S_g$, where $S_g$ is the symmetric group on $g$ elements. The differential $\hat \partial$ arises from this more geometric setting as a certain count of holomorphic ``Whitney'' disks between these intersection points\footnote{The ambient space $\operatorname{Sym}^g(\Sigma)$ can be equipped with a complex structure induced by a choice of  complex structure on $\Sigma$.}; see Figure~\ref{fig:heeg-disk} for a schematic. This construction is modeled on Floer's variant of Morse homology for pairs of Lagrangian submanifolds of a symplectic manifold  \cite{floer:lagr}, where the tori $\mathbb{T}_\alpha,\mathbb{T}_\beta \subset \operatorname{Sym}^g(\Sigma)$ play the roles of the Lagrangians. (It was subsequently shown that  $\operatorname{Sym}^g(\Sigma)$ admits symplectic forms for which $\mathbb{T}_\alpha$ and $\mathbb{T}_\beta$ are indeed Lagrangian and such that Heegaard/link Floer homology arises as a suitable version Lagrangian Floer homology \cite{perutz}.) For a beautiful historical account of this theory's development, see the survey by Greene  \cite{greene:floer}. 

\begin{figure}
\begin{minipage}{.35\linewidth}
\captionsetup{width=\linewidth}\caption{Schematic depiction of a holomorphic Whitney disk from $x$ to $y$, including boundary conditions.}\label{fig:heeg-disk}
\end{minipage}  \hfill  \begin{minipage}{.525\linewidth} \def\svgwidth{\linewidth}
\begingroup%
  \makeatletter%
  \providecommand\color[2][]{%
    \errmessage{(Inkscape) Color is used for the text in Inkscape, but the package 'color.sty' is not loaded}%
    \renewcommand\color[2][]{}%
  }%
  \providecommand\transparent[1]{%
    \errmessage{(Inkscape) Transparency is used (non-zero) for the text in Inkscape, but the package 'transparent.sty' is not loaded}%
    \renewcommand\transparent[1]{}%
  }%
  \providecommand\rotatebox[2]{#2}%
  \newcommand*\fsize{\dimexpr\f@size pt\relax}%
  \newcommand*\lineheight[1]{\fontsize{\fsize}{#1\fsize}\selectfont}%
  \ifx\svgwidth\undefined%
    \setlength{\unitlength}{414.71489854bp}%
    \ifx\svgscale\undefined%
      \relax%
    \else%
      \setlength{\unitlength}{\unitlength * \real{\svgscale}}%
    \fi%
  \else%
    \setlength{\unitlength}{\svgwidth}%
  \fi%
  \global\let\svgwidth\undefined%
  \global\let\svgscale\undefined%
  \makeatother%
  \begin{picture}(1,0.36736612)%
    \lineheight{1}%
    \setlength\tabcolsep{0pt}%
    \put(0,0){\includegraphics[width=\unitlength,page=1]{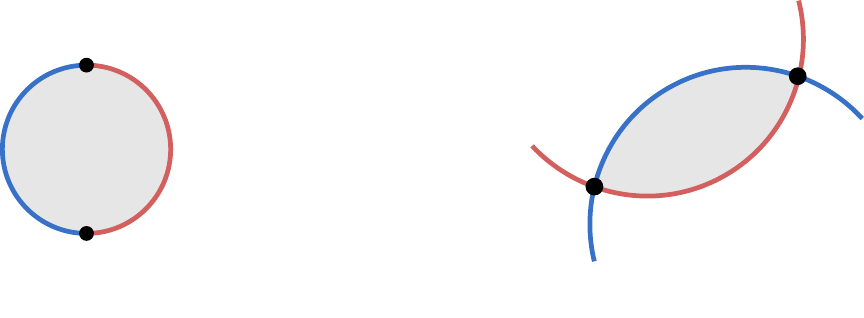}}%
    \put(0.0965077,0.04099934){\color[rgb]{0,0,0}\makebox(0,0)[t]{\smash{\begin{tabular}[t]{c}$-i$\end{tabular}}}}%
    \put(0.1001247,0.31588745){\color[rgb]{0,0,0}\makebox(0,0)[t]{\smash{\begin{tabular}[t]{c}$i$\end{tabular}}}}%
    \put(0.71583209,0.10847242){\color[rgb]{0,0,0}\makebox(0,0)[t]{\smash{\begin{tabular}[t]{c}$x$\end{tabular}}}}%
    \put(0.59522543,0.2264216){\color[rgb]{0.82745098,0.37254902,0.37254902}\makebox(0,0)[t]{\smash{\begin{tabular}[t]{c}$\mathbb{T}_\alpha$\end{tabular}}}}%
    \put(0.68517647,0.00773547){\color[rgb]{0.21568627,0.44313725,0.78431373}\makebox(0,0)[t]{\smash{\begin{tabular}[t]{c}$\mathbb{T}_\beta$\end{tabular}}}}%
    \put(0.94184775,0.23063059){\color[rgb]{0,0,0}\makebox(0,0)[t]{\smash{\begin{tabular}[t]{c}$y$\end{tabular}}}}%
    \put(0,0){\includegraphics[width=\unitlength,page=2]{holo-disks.pdf}}%
  \end{picture}%
\endgroup%
 
\end{minipage} \quad \
\end{figure}

In the hat flavor, the differential only counts holomorphic Whitney disks that avoid the subspaces $\{w\} \times \operatorname{Sym}^{g-1}(\Sigma)$ and $\{z\} \times \operatorname{Sym}^{g-1}(\Sigma)$. This restriction is dropped in the minus flavor's chain complex $\cflm$  over $\rcal=\ff_2[U,V]$, where Whitney disks' intersections with  $\{w\} \times \operatorname{Sym}^{g-1}(\Sigma)$ and $\{z\} \times \operatorname{Sym}^{g-1}(\Sigma)$ are incorporated into the differential via powers of the variables $U$ and $V$, respectively.

\begin{remark}
    In much of the literature, it is sufficient to work with a version of $\cflm$ over $\ff_2[U]$ whose differential only counts Whitney disks avoiding $\{z\} \times \operatorname{Sym}^{g-1}(\Sigma)$. 
\end{remark}




\smallskip

\subsection{Link cobordisms and functoriality}\label{subsec:functoriality}
Suppose that $\Sigma \subset S^3 \times [0,1]$ is a link cobordism from $L_0 \subset S^3 \times 0$ to $L_1 \subset S^3 \times 1$. After a perturbation of $\Sigma$ (rel boundary), we may assume that the projection $f: S^3 \times [0,1] \to [0,1]$ restricts to a Morse function on $\Sigma$ and that all critical points have distinct values. Observe that if there are no critical points, then $\Sigma$ is the trace of an isotopy between $L_0$ and $L_1$. If there is at least one critical point, we may choose regular values $t_0=0, \ldots,t_n=1$ in $[0,1]$ such that each $t_i$ and $t_{i+1}$ are separated by exactly one critical value. The preimages of these points are links 
$$L_i=(f|_\Sigma)^{-1}(t_i) \ = \Sigma \cap \left(S^3 \times \{t_i\}\right)$$ 
decomposing $\Sigma$ into a union of subcobordisms $\Sigma=\Sigma_1 \cup \cdots \cup \Sigma_n$ defined by 
$$\Sigma_i = \Sigma \cap \left(S^3 \times [t_{i-1},t_{i}]\right).$$


An example is depicted schematically in Figure~\ref{fig:cob-cut-up}. Each of these $\Sigma_i$ consists of exactly one of the elementary cobordisms in Figure~\ref{fig:elementary-cobs}(a)-(d) (called \emph{birth}, \emph{merging saddle}, \emph{splitting saddle}, and \emph{death}, respectively) and some number of cylinders as in Figure~\ref{fig:elementary-cobs}(e). 
In Khovanov homology, we further decompose a link cobordism into a \emph{movie} of link diagrams, each related by one of the following:


\begin{wrapfigure}[7]{r}{.325\linewidth}
\vspace{-0.25in}
\centering
\includegraphics[width=.9\linewidth]{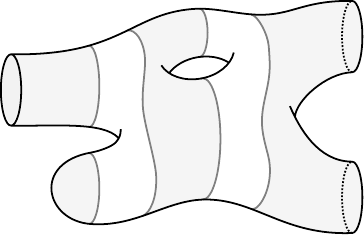}
\captionsetup{width=.835\linewidth} \caption{Decomposing a cobordism into elementary pieces.}\label{fig:cob-cut-up}
\end{wrapfigure}

\

\vspace{-20pt}

\begin{enumerate}[label=(\roman*),leftmargin=1.5cm,itemsep=-2pt]
\item birth
\item saddles (either merging or splitting)
\item death
\item Reidemeister I, II, or III moves
\item planar isotopy.
\end{enumerate}

\medskip

\noindent An example is shown in Figure~\ref{fig:tref-movie}.

\begin{figure}[tb]
\center
\def\svgwidth{.8\linewidth}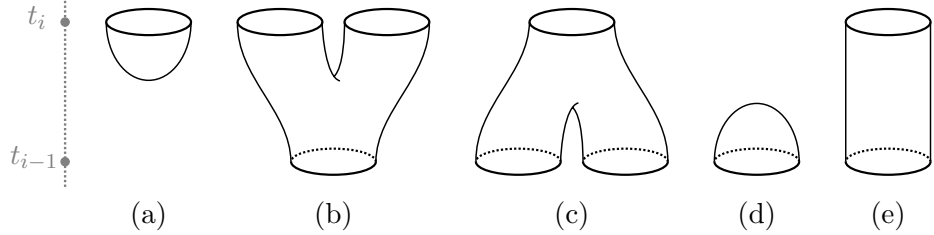
\caption{From left to right: birth, merging saddle, splitting saddle, death, and cylinder.}\label{fig:elementary-cobs}
\end{figure}

\begin{figure}[h]
\center
\def\svgwidth{.675\linewidth}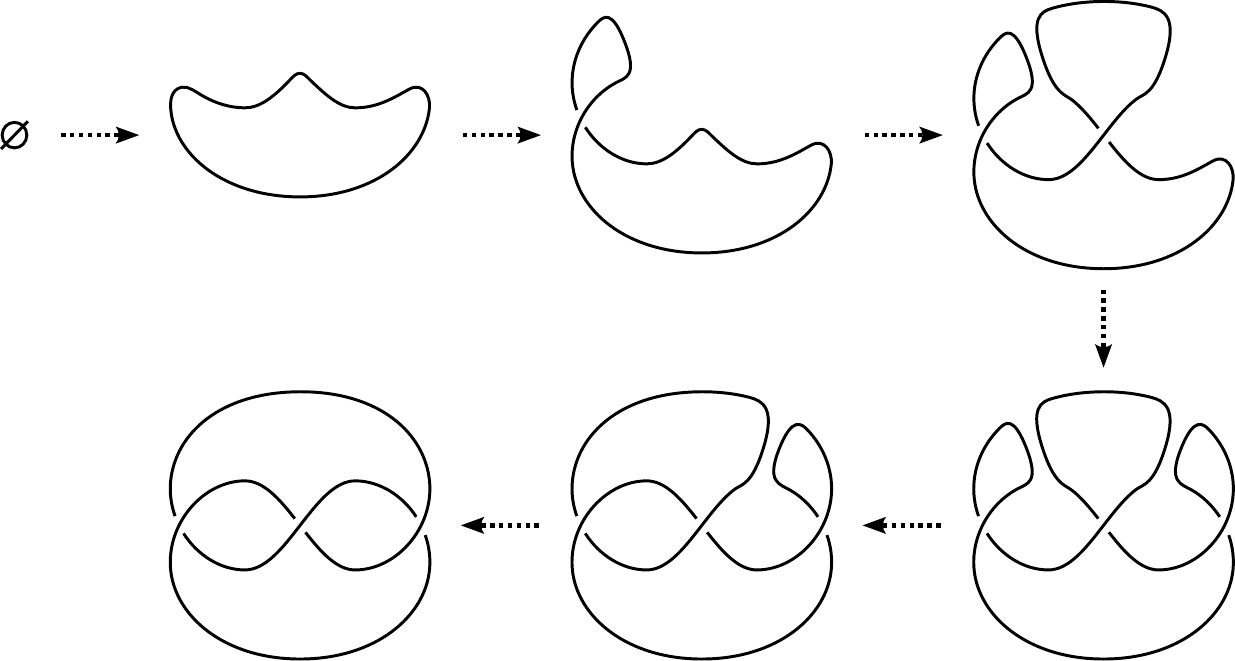
\caption{A movie 
of a link cobordism from the empty link  to  the left-handed trefoil.}\label{fig:tref-movie}
\end{figure}


\begin{wrapfigure}[8]{r}{.315\linewidth}
\hfill \def\svgwidth{.9\linewidth}
\begingroup%
  \makeatletter%
  \providecommand\color[2][]{%
    \errmessage{(Inkscape) Color is used for the text in Inkscape, but the package 'color.sty' is not loaded}%
    \renewcommand\color[2][]{}%
  }%
  \providecommand\transparent[1]{%
    \errmessage{(Inkscape) Transparency is used (non-zero) for the text in Inkscape, but the package 'transparent.sty' is not loaded}%
    \renewcommand\transparent[1]{}%
  }%
  \providecommand\rotatebox[2]{#2}%
  \newcommand*\fsize{\dimexpr\f@size pt\relax}%
  \newcommand*\lineheight[1]{\fontsize{\fsize}{#1\fsize}\selectfont}%
  \ifx\svgwidth\undefined%
    \setlength{\unitlength}{282.93733131bp}%
    \ifx\svgscale\undefined%
      \relax%
    \else%
      \setlength{\unitlength}{\unitlength * \real{\svgscale}}%
    \fi%
  \else%
    \setlength{\unitlength}{\svgwidth}%
  \fi%
  \global\let\svgwidth\undefined%
  \global\let\svgscale\undefined%
  \makeatother%
  \begin{picture}(1,0.75506567)%
    \lineheight{1}%
    \setlength\tabcolsep{0pt}%
    \put(0,0){\includegraphics[width=\unitlength,page=1]{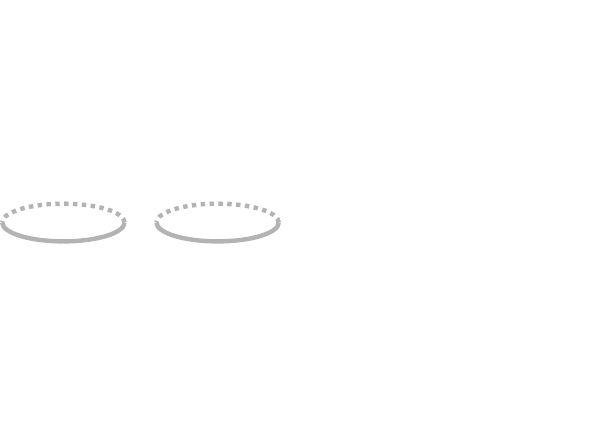}}%
    \put(0.63736113,0.40036719){\color[rgb]{0,0,0}\makebox(0,0)[t]{\smash{\begin{tabular}[t]{c}$\approx$\end{tabular}}}}%
    \put(0,0){\includegraphics[width=\unitlength,page=2]{birth-death.pdf}}%
  \end{picture}%
\endgroup%
 \vspace{-.05in}

\caption{}\label{fig:birth-death}
\end{wrapfigure}

The core strategy for defining cobordism maps $\hh(\Sigma)$ is to define chain maps $\chh(\Sigma_i)$ for these elementary cobordisms. Invariance under isotopy is then established  by comparing the maps induced by cobordisms that differ by appropriate elementary changes in these decompositions; see Figure~\ref{fig:birth-death} for a schematic example. (We note that cylinders corresponding to nontrivial isotopies generally induce nontrivial maps.)
The invariance statements for our two main theories are given below:

\begin{theorem}[Jacobsson, Bar-Natan, Khovanov, Morrison--Walker--Wedrich]
Let $\Sigma$  be a link cobordism in $S^3 \times I$ between links $L_0,L_1 \subset S^3$. There  is an induced map $$\kh(\Sigma): \kh(L_0) \to \kh(L_1)$$ that is bigraded of degree $(h,q)=(0,\chi(\Sigma))$,  well-defined up to sign, and invariant under isotopy of $\Sigma$ rel boundary.
\end{theorem}

\begin{remark}(a) The original proofs of invariance (up to sign) in \cite{jacobsson,barnatan,khovanov:invariant} only hold for surfaces in $\rr^3 \times I$. This was promoted to surfaces in $S^3 \times I$ in \cite{morrison-walker-wedrich} by establishing invariance under the \emph{sweep-around move}, which involves pushing a strand of a link past the point at infinity in $S^3$. Their proof was adapted to the Bar-Natan and Lee deformations in \cite[Proposition 3.7]{lipshitz-sarkar:mixed}.

(b) We will often wish to study smooth, compact, oriented, properly embedded surfaces $\Sigma \subset B^4$ bounded by a link in $\partial B^4$. In this setting, we may remove a small open ball from the interior of $B^4 \setminus \Sigma$ and view $\Sigma$ as a subset of the resulting copy of $S^3 \times I$.  Moreover, given an isotopy of $B^4$ carrying $\Sigma$ to another surface, we may modify the isotopy (if necessary) so that it is supported away from a small open ball, in which case the surfaces and isotopy may be realized in $S^3 \times I$. 
\end{remark}



To achieve functoriality in link Floer homology, one must work with \emph{decorated} links and link cobordisms. Roughly speaking, a link $L$ must be equipped with two sets of basepoints, denoted $w,z \subset L$. A cobordism $\Sigma : L_0 \to L_1$ must be decomposed into two subsurfaces $\Sigma_w$ and $\Sigma_z$ containing the sets of $w$- and $z$-basepoints, respectively, in $L_i$. 

\begin{theorem}[Juh\'asz, Juh\'asz--Marengon, Zemke]\label{thm:hfl-functorial}
Let $\Sigma$ be a decorated link cobordism in  $ S^3 \times I$  between links $L_0,L_1 \subset S^3$. There  is an induced map $$\hfl^\circ(\Sigma): \hfl^\circ(L_0) \to \hfl^\circ(L_1)$$ that is well-defined and invariant under isotopy of $\Sigma$ rel boundary.
\end{theorem}


The details of establishing invariance  differ sharply between different theories, reflecting the fundamental differences in the theories' constructions.




\medskip

\subsection*{Duality} 

 Under mirroring $L \rightsquigarrow -L$, both  Khovanov homology and link Floer homology satisfy a duality isomorphism at the \emph{chain level}: $\chh(-L) \cong (\chh(L))^*$. In more detail, at the level of bigraded chain groups,  we have
$$
\chh_{i,j}(-L) \cong \operatorname{Hom}_\rcal(\chh_{-i,-j}(L), \rcal).
$$
The cobordism (chain) maps extend this duality: Given a link cobordism $\Sigma: L_0 \to L_1$, the diagram
$$\begin{tikzcd}
\chh(-L_1) \arrow[r] \arrow[d, "\chh(-\Sigma)"'] & \left( \chh(L_1) \right)^* \arrow[d, "\left(\chh(\Sigma)\right)^*"] \\
\chh(-L_0) \arrow[r]                          & \left( \chh(L_0) \right)^*                                      
\end{tikzcd}$$
commutes up to chain homotopy; here $-\Sigma$ denotes the time-reversed mirror of $\Sigma$ (whose link movie is obtained from that of $\Sigma$ by mirroring each frame and playing it backwards in time); see Figure~\ref{fig:tref-movie-mirror} for an example.

For details, see \cite{Khovanov,rasmussen:s} and \cite[Section 3.3]{jmz:exotic} (building on  \cite[Proposition 3.7]{oz-sz:hfk} and \cite[Theorem 3.5]{oz-sz:four}).
 
\begin{figure}
\center
\def\svgwidth{.67\linewidth}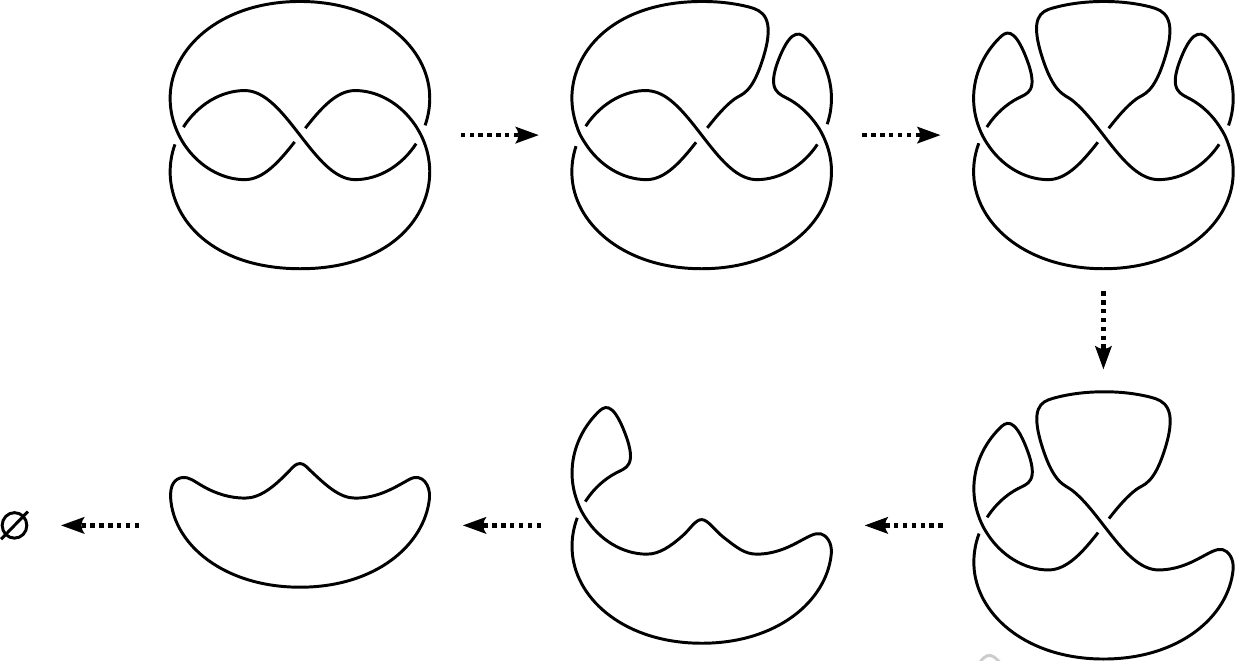
\caption{The time-reversed mirror of the movie from Figure~\ref{fig:tref-movie}.}\label{fig:tref-movie-mirror}
\end{figure}


\smallskip

\subsection{Closed surfaces}
Once functoriality has been established, it is natural to ask whether a link homology theory provides an effective invariant of knotted embeddings of closed surfaces $\Sigma$ in $\rr^4$ (or $S^4$) by fixing a copy of $S^3 \times I$ inside $\rr^4$ (or $S^4$) that contains $\Sigma$ and viewing $\Sigma$ as a link cobordism $\emptyset \to \emptyset$. Unfortunately (perhaps), it turns out that many variants of these theories behave uniformly on closed, connected, orientable surfaces of fixed genus.

\begin{example}[\cite{rasmussen:closed,tanaka:closed}] \label{ex:closed-kh}
If $\Sigma$ is any closed, connected, orientable, genus-$g$ surface in $S^3 \times I$, then 
\begin{align*}
&\kh(\Sigma)=\begin{cases} \times( \pm 2) & \ g = 1 \\ 0 &  \ g \neq 1\end{cases}& &
&\bn(\Sigma)=\begin{cases} \times \big(\! \pm 2^g \var^{(g-1)/2}\big) &  \text{odd }\, g \\ 0 &  \text{even } g.\end{cases}
\end{align*}
\end{example}

\begin{example}[\cite{zemke:gradings}]\label{ex:closed-hf} 
Let $\Sigma$ be any closed, connected, orientable, genus-$g$ surface in $S^3 \times I$. For $\widehat \hfl$ over $\ff_2$, we have
$$\widehat\hfl(\Sigma)=\begin{cases} \id & \ g = 0 \\ 0 &  \ g \neq 0.\end{cases}
$$
For $\hflm$ over $\rcal=\ff_2[U,V]$, if $\Sigma$ is decorated as a union of connected subsurfaces $\Sigma _w \cup \Sigma_z$ where $\partial \Sigma_w = \partial \Sigma_z$ is a single closed curve, then we have
\begin{align*}
\hflm(\Sigma): \rcal &\to \rcal   \\
 1 &\mapsto U^{g(\Sigma_w)} V^{g(\Sigma_z)}.
\end{align*}
\end{example}

Some valuable nonvanishing results follow from these rather formal properties; see \S\ref{subsec:boundary}. 


\begin{remark}
Inspired by the construction of the \emph{mixed invariant} in Heegaard Floer homology \cite{oz-sz:four}, Lipshitz and Sarkar \cite{lipshitz-sarkar:mixed} leverage the vanishing of the Bar-Natan (and Lee) cobordism maps for closed, non-orientable surfaces to construct an analogous invariant of closed, non-orientable surfaces $\Sigma$ with crosscap number $\geq 3$. They show that this invariant vanishes if $\Sigma$ is connected, though it remains possible that it is an effective invariant when $\Sigma$ is disconnected.\end{remark}


\medskip

\subsection{Surfaces with boundary}\label{subsec:boundary}
Despite their relatively uniform behavior for closed surfaces, it turns out that these theories provide very effective invariants of link cobordisms with nonempty boundary --- especially for surfaces in $B^4$ bounded by links in $S^3$. We say that a smooth, oriented, properly embedded surface $\Sigma \subset B^4$ is a \emph{slice surface} for an oriented link $L \subset S^3$ if $\partial \Sigma= L$ (and $\Sigma$ has no closed components).

\begin{theorem}[Juh\'asz--Zemke \cite{juhasz-zemke:disks}]
The cobordism maps on link Floer homology can distinguish pairs of slice disks $D,D' \subset B^4$ bounded by the same knot $K \subset S^3$, up to isotopy rel boundary. 
\end{theorem}

The first slice disks distinguished in \cite{juhasz-zemke:disks} are called \emph{roll-spun} disks, based on a construction first studied by Fox \cite{fox:rolling} and developed further by Litherland \cite{litherland:deform}. For a concrete example, let $K$ be the right-handed trefoil knot and let $D$ be the standard slice disk bounded by $K \# -K$. Let $C$ be the concordance from $K \# -K$ to itself swept out by the \emph{swallow-follow isotopy} depicted in Figure~\ref{fig:roll-spin}. Attaching $C$ to $D$ yields a new slice disk $D'$ that can be distinguished from $D$ up to isotopy rel boundary via their cobordism maps on link Floer homology \cite[Theorem~5.4]{juhasz-zemke:disks}. 

\begin{figure}
\center
\includegraphics[width=.9\linewidth]{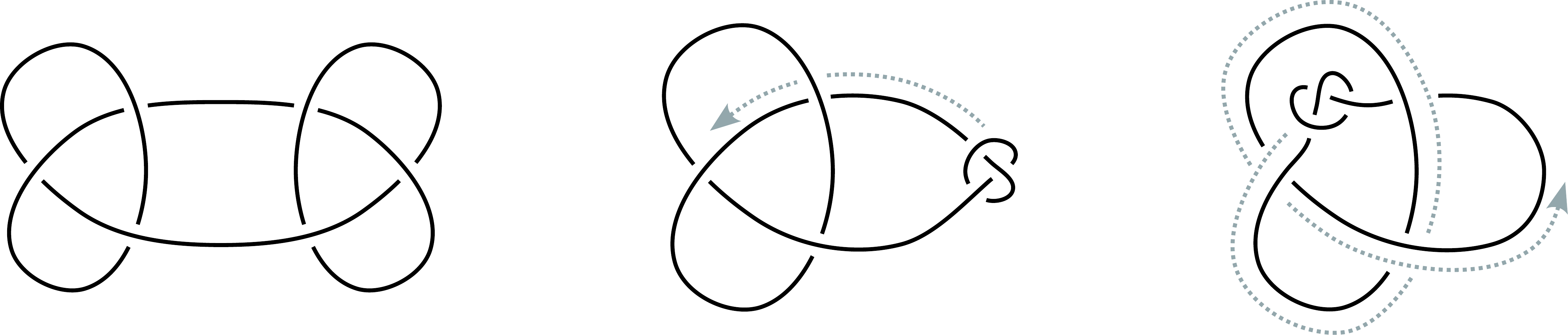}
\caption{An example of the swallow-follow isotopy that underlies roll-spinning.}\label{fig:roll-spin}
\end{figure}

A few years later, it was shown that Khovanov homology could also distinguish knotted slice disks up to isotopy rel boundary.

\begin{theorem}[Sundberg--Swann \cite{sundberg-swann}]
The cobordism maps on Khovanov homology can distinguish pairs of slice disks $D,D' \subset B^4$ bounded by the same knot $K \subset S^3$, up to isotopy rel boundary. 
\end{theorem}

The original pair of core examples studied in \cite{sundberg-swann} are depicted in Figure~\ref{fig:946}. Their approach is direct and consists of explicitly calculating the images $\kh(D)(1),\kh(D')(1) \in \kh(9_{46})$ and distinguishing them using homomorphisms out of $\kh(9_{46})$.

\begin{figure}[b]
\center
\def\svgwidth{.825\linewidth}
\begingroup%
  \makeatletter%
  \providecommand\color[2][]{%
    \errmessage{(Inkscape) Color is used for the text in Inkscape, but the package 'color.sty' is not loaded}%
    \renewcommand\color[2][]{}%
  }%
  \providecommand\transparent[1]{%
    \errmessage{(Inkscape) Transparency is used (non-zero) for the text in Inkscape, but the package 'transparent.sty' is not loaded}%
    \renewcommand\transparent[1]{}%
  }%
  \providecommand\rotatebox[2]{#2}%
  \newcommand*\fsize{\dimexpr\f@size pt\relax}%
  \newcommand*\lineheight[1]{\fontsize{\fsize}{#1\fsize}\selectfont}%
  \ifx\svgwidth\undefined%
    \setlength{\unitlength}{720.09597898bp}%
    \ifx\svgscale\undefined%
      \relax%
    \else%
      \setlength{\unitlength}{\unitlength * \real{\svgscale}}%
    \fi%
  \else%
    \setlength{\unitlength}{\svgwidth}%
  \fi%
  \global\let\svgwidth\undefined%
  \global\let\svgscale\undefined%
  \makeatother%
  \begin{picture}(1,0.40510148)%
    \lineheight{1}%
    \setlength\tabcolsep{0pt}%
    \put(0,0){\includegraphics[width=\unitlength,page=1]{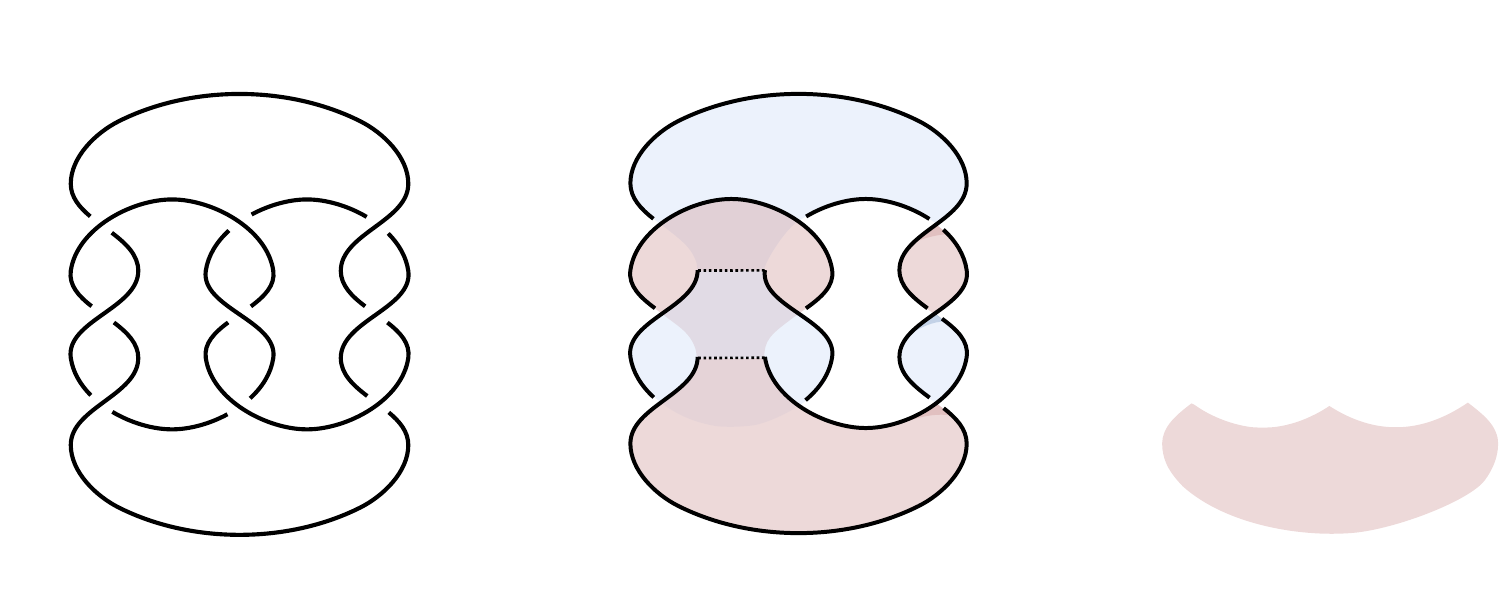}}%
    \put(0.15966056,0.00619768){\color[rgb]{0,0,0}\makebox(0,0)[t]{\smash{\begin{tabular}[t]{c}{\footnotesize $K=9_{46}$}\end{tabular}}}}%
    \put(0.53227691,0.00494999){\color[rgb]{0,0,0}\makebox(0,0)[t]{\smash{\begin{tabular}[t]{c}{\footnotesize$D$}\end{tabular}}}}%
    \put(0,0){\includegraphics[width=\unitlength,page=2]{946.pdf}}%
    \put(0.88639772,0.00494999){\color[rgb]{0,0,0}\makebox(0,0)[t]{\smash{\begin{tabular}[t]{c}{\footnotesize$D'$}\end{tabular}}}}%
  \end{picture}%
\endgroup%

\caption{The knot $9_{46}$, along with a pair of slice disks that it bounds.}\label{fig:946}
\end{figure}

We note that the pairs of surfaces distinguished in \cite{juhasz-zemke:disks} and \cite{sundberg-swann} could already be distinguished using more classical tools. For example, one can compare their \emph{peripheral maps} (cf.~\cite[Definition~3.9]{juhasz-zemke:disks}), the inclusion-induced map $\pi_1(S^3 \setminus \partial D) \to \pi_1(B^4 \setminus D)$. 

In contrast, as we now discuss, further work has shown that both link Floer homology and Khovanov homology can distinguish surfaces that are \emph{exotically knotted}, i.e., those that are topologically isotopic (rel boundary) but not smoothly isotopic.

\begin{theorem}[Juh\'asz--Miller--Zemke \cite{jmz:exotic}]
There are infinitely many knots in $S^3$ such that each bounds countably infinitely many properly embedded, smooth, orientable, genus-one surfaces in $B^4$ that are pairwise topologically isotopic (rel boundary) but there is no diffeomorphism of $B^4$ taking one to the other.
\end{theorem}

\begin{figure}[b]
\sidecaption
\def\svgwidth{.45\linewidth}
\begingroup%
  \makeatletter%
  \providecommand\color[2][]{%
    \errmessage{(Inkscape) Color is used for the text in Inkscape, but the package 'color.sty' is not loaded}%
    \renewcommand\color[2][]{}%
  }%
  \providecommand\transparent[1]{%
    \errmessage{(Inkscape) Transparency is used (non-zero) for the text in Inkscape, but the package 'transparent.sty' is not loaded}%
    \renewcommand\transparent[1]{}%
  }%
  \providecommand\rotatebox[2]{#2}%
  \newcommand*\fsize{\dimexpr\f@size pt\relax}%
  \newcommand*\lineheight[1]{\fontsize{\fsize}{#1\fsize}\selectfont}%
  \ifx\svgwidth\undefined%
    \setlength{\unitlength}{632.64757532bp}%
    \ifx\svgscale\undefined%
      \relax%
    \else%
      \setlength{\unitlength}{\unitlength * \real{\svgscale}}%
    \fi%
  \else%
    \setlength{\unitlength}{\svgwidth}%
  \fi%
  \global\let\svgwidth\undefined%
  \global\let\svgscale\undefined%
  \makeatother%
  \begin{picture}(1,0.53330883)%
    \lineheight{1}%
    \setlength\tabcolsep{0pt}%
    \put(0,0){\includegraphics[width=\unitlength,page=1]{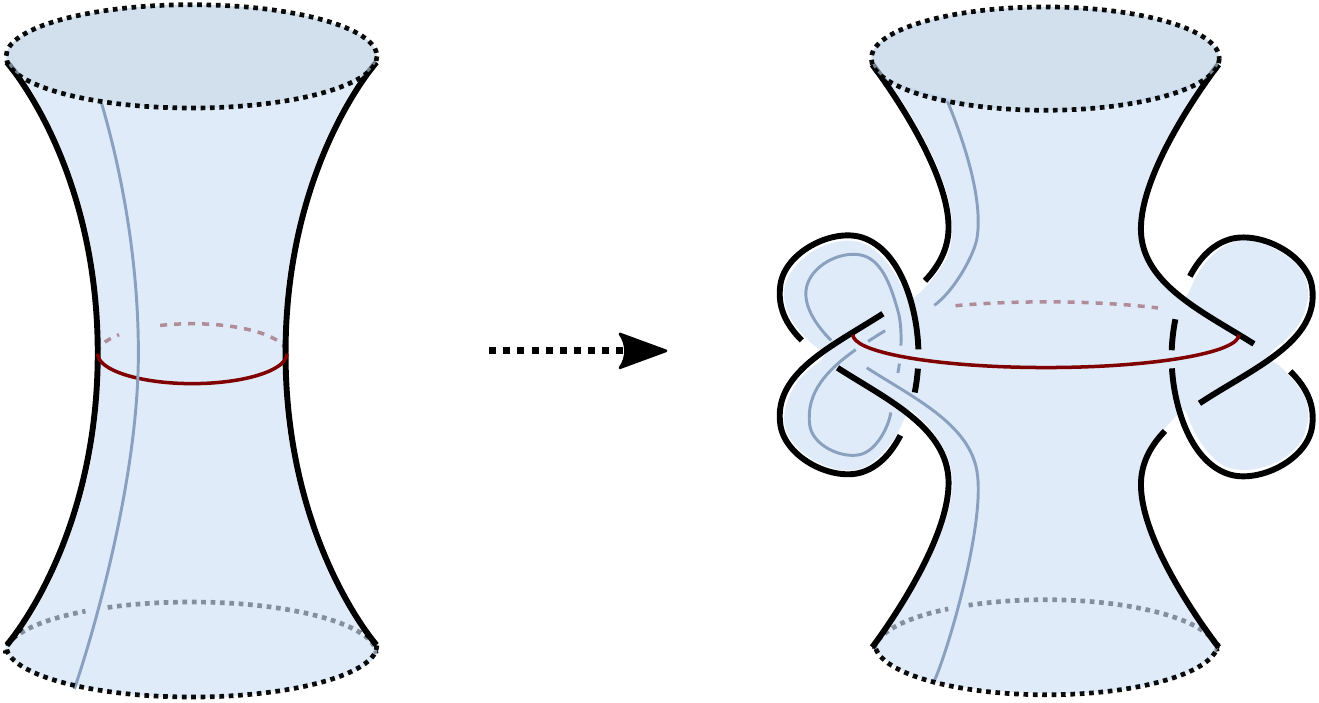}}%
    \put(0.03114912,0.25253727){\color[rgb]{0.50196078,0,0}\makebox(0,0)[t]{\smash{\begin{tabular}[t]{c}$\gamma$\end{tabular}}}}%
  \end{picture}%
\endgroup%

\caption{A schematic depiction of rim surgery along a 
curve $\gamma$ in an embedded surface.}\label{fig:rim-surgery}
\end{figure}

To prove this result, Juh\'asz--Miller--Zemke compute the effect of \emph{twisted rim surgery}  on a perturbed version of sutured link Floer homology. Rim surgery was introduced by Fintushel and Stern \cite{fs:surfaces}. Informally, it consists of choosing a simple closed curve $\gamma$ in an embedded surface $\Sigma$, fixing a tubular neighborhood $S^1 \times B^3$ of $\gamma$  that meets $\Sigma$ in $S^1 \times I$ (where $I$ is an unknotted arc between the north and south poles in $B^3$), and modifying $\Sigma$ within $S^1 \times B^3$ by tying a knot $K$ into each arc $\theta \times I $. This is depicted schematically in Figure~\ref{fig:rim-surgery}. The \emph{twisted} version of rim surgery includes a rotation of the knotted arc as one sweeps through the $S^1$ factor, as depicted in Figure~\ref{fig:twist}. 
The use of  \emph{1-twisted} rim surgery in \cite{jmz:exotic} is one of two key ingredients for ensuring that the modified surface $\Sigma'$ remains topologically isotopic to $\Sigma$; the second ingredient is that the curve $\gamma \subset \Sigma$ is arranged to bound a \emph{topologically} embedded disk $\Delta$ in the complement of $\Sigma$. (If $\Delta$ was smooth, then $\Sigma'$ would be smoothly isotopic to $\Sigma$.)

%

\begin{figure}
\center
\includegraphics[width=.9\linewidth]{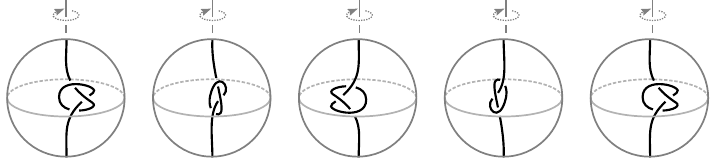}
\caption{Rotation of a knotted arc in $B^3$ as applied in twisted rim surgery.}\label{fig:twist}
\end{figure}


Different flavors of exotic surfaces were detected with Khovanov homology:

\begin{theorem}[Hayden--Sundberg \cite{hayden-sundberg}]
For all $g\geq 0$, there are infinitely many knots in $S^3$ that each bound a pair of smooth, orientable, genus-$g$ surfaces $\Sigma,\Sigma' \subset B^4$ that are exotically knotted and induce distinct maps $\kh(\Sigma)\neq \pm \kh(\Sigma')$.
\end{theorem}

In contrast with \cite{sundberg-swann}, the strategy in \cite{hayden-sundberg} is to view the surfaces as cobordisms $K \to \emptyset$ and find an explicit element in $\kh(K)$ that behaves differently under the maps $$\kh(\Sigma),\kh(\Sigma') : \kh(K)\to \zz.$$ The core examples treated in \cite{hayden-sundberg} are the positron disks from \cite{hayden:corks},  depicted in Figure~\ref{fig:positron-disks}. These disks are only exotically knotted \emph{rel boundary}, but can be upgraded to produce \emph{absolutely} exotic disks and higher-genus surfaces as in \cite{hayden-sundberg}. The existence of a topological isotopy (rel boundary) is obtained via the following:

\begin{figure}
\center
\def\svgwidth{.85\linewidth}
\begingroup%
  \makeatletter%
  \providecommand\color[2][]{%
    \errmessage{(Inkscape) Color is used for the text in Inkscape, but the package 'color.sty' is not loaded}%
    \renewcommand\color[2][]{}%
  }%
  \providecommand\transparent[1]{%
    \errmessage{(Inkscape) Transparency is used (non-zero) for the text in Inkscape, but the package 'transparent.sty' is not loaded}%
    \renewcommand\transparent[1]{}%
  }%
  \providecommand\rotatebox[2]{#2}%
  \newcommand*\fsize{\dimexpr\f@size pt\relax}%
  \newcommand*\lineheight[1]{\fontsize{\fsize}{#1\fsize}\selectfont}%
  \ifx\svgwidth\undefined%
    \setlength{\unitlength}{756.48758535bp}%
    \ifx\svgscale\undefined%
      \relax%
    \else%
      \setlength{\unitlength}{\unitlength * \real{\svgscale}}%
    \fi%
  \else%
    \setlength{\unitlength}{\svgwidth}%
  \fi%
  \global\let\svgwidth\undefined%
  \global\let\svgscale\undefined%
  \makeatother%
  \begin{picture}(1,0.32716057)%
    \lineheight{1}%
    \setlength\tabcolsep{0pt}%
    \put(0,0){\includegraphics[width=\unitlength,page=1]{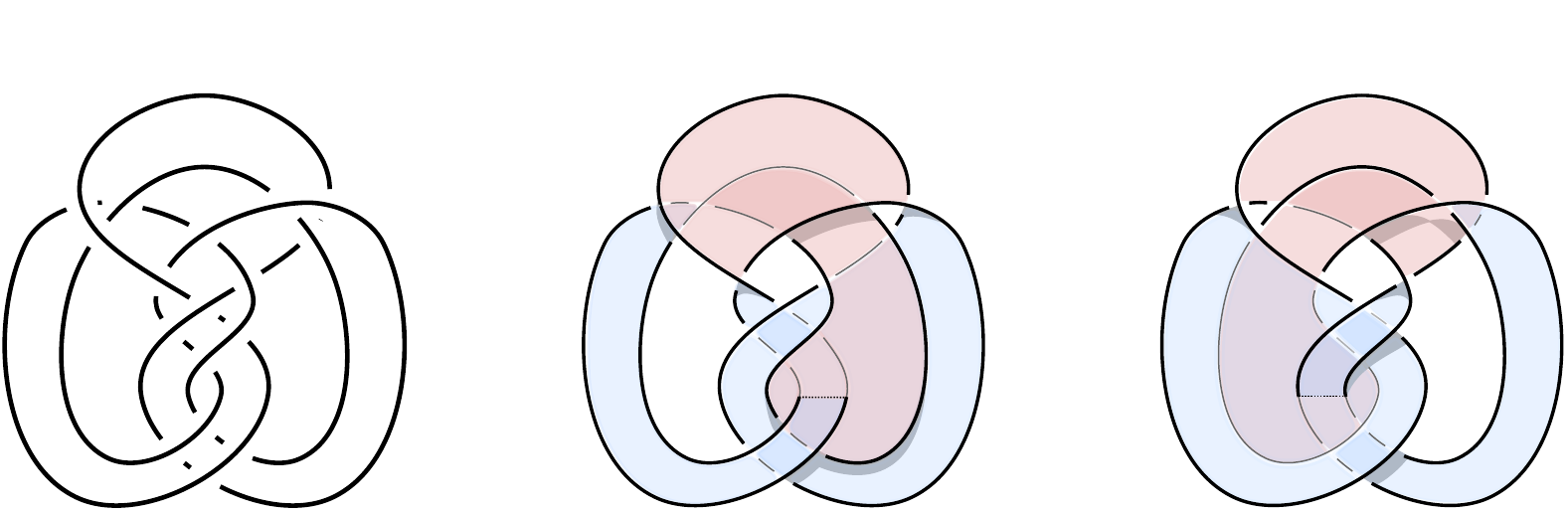}}%
    \put(0.18607569,0.2959062){\color[rgb]{0.4,0.4,0.4}\makebox(0,0)[t]{\smash{\begin{tabular}[t]{c}{\footnotesize $\tau$}\end{tabular}}}}%
    \put(0,0){\includegraphics[width=\unitlength,page=2]{positron-disks.pdf}}%
  \end{picture}%
\endgroup%

\caption{The positron knot and a pair of exotically knotted slice disks that it bounds.}\label{fig:positron-disks}
\end{figure}

\begin{theorem}[Conway--Powell \cite{conway-powell}]
Any smooth, properly embedded disks  $D$ and $D'$ in $B^4$ with the same boundary and whose complements have $\pi_1 \cong \zz$ 
are topologically isotopic rel boundary.
%
%
\end{theorem}

One of the first successful applications of Khovanov homology to an open question about knotted surfaces concerned a question of Livingston \cite{livingston}, who asked whether any two Seifert surfaces of equal genus for a knot $K$ are isotopic through properly embedded surfaces in $B^4$.

\begin{theorem}[Hayden--Kim--Miller--Park--Sundberg \cite{hkmps:seifert}]
There exist Seifert surfaces in $S^3$ with the same boundary and genus that are not smoothly isotopic through properly embedded surfaces in $B^4$.
\end{theorem}

One of the simplest examples from \cite{hkmps:seifert} is a pair of Seifert surfaces for $\wh(3_1)$, the Whitehead double of the right-handed trefoil $3_1$, depicted in Figure~\ref{fig:wh(tref)}: part (a)  shows the standard genus-1 Seifert surface for $3_1$; (b) shows a genus-2 Seifert surface $\Sigma$ for $\wh(3_1)$ obtained by applying the Whitehead doubling operation to the original Seifert surface itself; (c) shows an alternative genus-2 Seifert surface $\Sigma'$ obtained by adding a tube to the standard genus-1 Seifert surface for $\wh(3_1)$. The authors show that, over $\zz_2$ coefficients, the associated cobordism maps satisfy  $\kh(\Sigma)\neq0$ but $\kh(\Sigma')=0$. See Exercise~\ref{exer:livingston} for more.

\begin{figure}
\center
\def\svgwidth{.9\linewidth}
\begingroup%
  \makeatletter%
  \providecommand\color[2][]{%
    \errmessage{(Inkscape) Color is used for the text in Inkscape, but the package 'color.sty' is not loaded}%
    \renewcommand\color[2][]{}%
  }%
  \providecommand\transparent[1]{%
    \errmessage{(Inkscape) Transparency is used (non-zero) for the text in Inkscape, but the package 'transparent.sty' is not loaded}%
    \renewcommand\transparent[1]{}%
  }%
  \providecommand\rotatebox[2]{#2}%
  \newcommand*\fsize{\dimexpr\f@size pt\relax}%
  \newcommand*\lineheight[1]{\fontsize{\fsize}{#1\fsize}\selectfont}%
  \ifx\svgwidth\undefined%
    \setlength{\unitlength}{439.8152413bp}%
    \ifx\svgscale\undefined%
      \relax%
    \else%
      \setlength{\unitlength}{\unitlength * \real{\svgscale}}%
    \fi%
  \else%
    \setlength{\unitlength}{\svgwidth}%
  \fi%
  \global\let\svgwidth\undefined%
  \global\let\svgscale\undefined%
  \makeatother%
  \begin{picture}(1,0.32072388)%
    \lineheight{1}%
    \setlength\tabcolsep{0pt}%
    \put(0,0){\includegraphics[width=\unitlength,page=1]{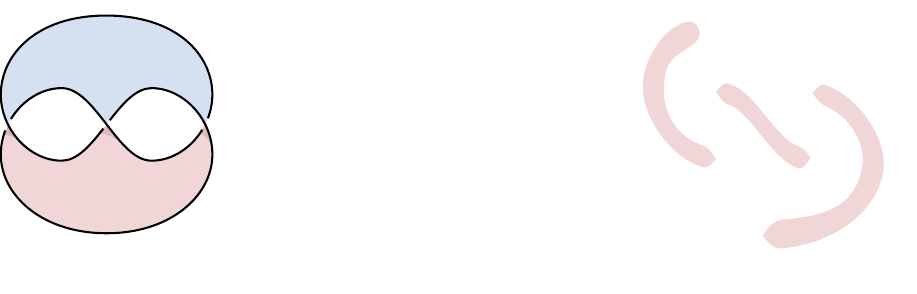}}%
    \put(0.1163239,0.00358126){\color[rgb]{0,0,0}\makebox(0,0)[t]{\smash{\begin{tabular}[t]{c}(a)\end{tabular}}}}%
    \put(0,0){\includegraphics[width=\unitlength,page=2]{wh_tref_-surfaces.pdf}}%
    \put(0.4664411,0.00358126){\color[rgb]{0,0,0}\makebox(0,0)[t]{\smash{\begin{tabular}[t]{c}(b)\end{tabular}}}}%
    \put(0,0){\includegraphics[width=\unitlength,page=3]{wh_tref_-surfaces.pdf}}%
    \put(0.83245122,0.00358126){\color[rgb]{0,0,0}\makebox(0,0)[t]{\smash{\begin{tabular}[t]{c}(c)\end{tabular}}}}%
    \put(0,0){\includegraphics[width=\unitlength,page=4]{wh_tref_-surfaces.pdf}}%
  \end{picture}%
\endgroup%

\caption{(a) The standard Seifert surface for the right-handed trefoil $3_1$. (b-c) Seifert surfaces for the positive Whitehead double $\wh(3_1)$.}\label{fig:wh(tref)}
\end{figure}

The operation of adding a tube as above is called \emph{internal stabilization}. More precisely, given a smoothly embedded, connected, orientable surface $\Sigma$ in $B^4$, this consists of choosing an embedded 3-dimensional 1-handle $h \approx [-1,1] \times D^2$ in $B^4$ that intersects $\Sigma$ only along $\{\pm 1\} \times D^2$, then removing $\{\pm 1\} \times \mathring{D}^2$ from $\Sigma$ and gluing in $[-1,1] \times \partial D^2$ to yield a new surface of larger genus. (Here we will only consider internal stabilization that preserves orientability.) 

A result of Baykur and Sunukjian \cite{baykur-sunukjian} (cf.~\cite[Appendix A]{conway-powell:cyclic}) shows that any two smoothly embedded surfaces in the same homology class in a 4-manifold become smoothly isotopic after each has been internally stabilized a sufficient number of times. Internal stabilization has the effect of killing the cobordism maps on $\kh$ and $\widehat\hfl$ with $\zz_2$ coefficients, making this a difficult phenomenon to probe. On the other hand, for the refined theories $\bn$ over $\ff_2[U]$ and $\hflm$ over $\ff_2[U,V]$, stabilization has the effect of multiplying the original cobordism map by appropriate powers the extra variables (as one might expect, given Examples~\ref{ex:closed-kh}-\ref{ex:closed-hf}). Juh\'asz and Zemke \cite{juhasz-zemke:stabilization} used this to show that the torsion order of elements of $\hfl^-$ (with respect to multiplication by $U$ or $V$) can provide a lower bound on the \emph{stabilization distance} between surfaces, i.e., the number of stabilizations needed to make two surfaces isotopic. Refining this, Guth detected the first examples of exotically knotted surfaces that remain distinct after internal stabilization.

\begin{theorem}[Guth \cite{guth}]
For any $n \geq 0$, there exists a knot in $S^3$ that bounds a pair of slice disks in $B^4$ that are exotically knotted (rel boundary) and whose stabilization distance is at least $n$.
\end{theorem}

The slice knots studied in \cite{guth} are the $(n,1)$-cables of the positron knot $J$ from Figure~\ref{fig:positron-disks}, and the slice disks are obtained by applying the $(n,1)$-cabling operation to the disks $D$ and $D'$ from Figure~\ref{fig:positron-disks}. Building on the computations of $\widehat\hfl(J)$ and the maps $\widehat\hfl(D)$ and $\widehat\hfl(D')$ from \cite{dms:equivariant}, Guth calculates $\hflm(D)$ and $\hflm(D')$ (up to a small indeterminacy). Here a key step is to understand the action of the symmetry $\tau$ from Figure~\ref{fig:positron-disks} on $\hflc$. This is then paired with arguments from bordered Floer homology that use the bordered Floer homology of the $(n,1)$-cabling pattern to (partially) upgrade these calculations to the cobordism maps induced by the cabled disks. 

Building on the above applications of Whitehead doubling and cabling to the study of knotted surfaces, a broader analysis of satellite operations on cobordism maps was undertaken in \cite{guth-hayden-kang-park}. There it is shown that, for example, if a pair of slice disks $D,D'$ for a knot $K$ induce distinct maps on $\widehat \hfl$, then so do the satellite disks $P(D)$ and $P(D')$ for the satellite knot $P(K)$ for any satellite pattern $P$ satisfying a certain \emph{non-cancellation condition} from bordered Floer homology. This is satisfied by Whitehead doubles, $(n,1)$-cables, the Mazur pattern, and many other patterns. Moreover, whenever $D$ and $D'$ induce distinct maps on $\widehat \hfl$, the stabilization distance between the $(n,1)$-cables of $D$ and $D'$ is at least $n$. 

%

It is also known that the cobordism maps on Khovanov homology can distinguish knotted surfaces that have been internally stabilized --- but it remains open whether it can distinguish pairs of surfaces with arbitrarily large stabilization distance.

\begin{theorem}[Hayden \cite{hayden:atomic}]
For all $g \geq 1$, there are exotically knotted pairs of properly embedded, genus-$g$ surfaces in $B^4$ that induce distinct maps on Bar-Natan homology, even after one internal stabilization.
\end{theorem}


The proof using Bar-Natan homology blends a manual argument with a calculation from Schuetz' program \textsl{KnotJob} \cite{knotjob}. A related example is discussed in \S\ref{sec:tqft}.

\smallskip
  
\subsection{Machine computation}

The link homology theories considered here, namely the mainstream flavors of link Floer homology and Khovanov homology, admit combinatorial descriptions that render them algorithmically computable. At the time of writing, efficient software exists to compute the hat-flavored knot Floer homology $\widehat \hfk$  of knots with $>100$ crossings and Khovanov homology $\kh$ of links with $>50$ crossings relatively quickly:

\begin{itemize}[leftmargin=25pt,itemsep=-3pt]

\item  knot Floer homology: \textsl{HFK-Calculator} \cite{hfk-calc}, implemented in \textsl{SnapPy} \cite{snappy}

\item  Khovanov homology: \textsl{KnotJob} \cite{knotjob}

\end{itemize}

\noindent While the implementation of algorithms to compute cobordism maps in these theories is in a nascent state (see, e.g., \cite{sundberg:github}), the computability of the homology groups themselves can be leveraged in certain calculations of the cobordism maps. Some examples of this are illustrated in some of the applications discussed below, and toy examples are given in the exercises.

\smallskip

\subsection{Exercises}

We close this introduction with a handful of exercises.

%
%
%



\begin{exercise}
\begin{enumerate}[label=\bfseries(\alph*),leftmargin=20pt]
\item 

Draw movies of link diagrams encoding the  surfaces in Figure~\ref{fig:embedded-ribbon}, which are depicted as embedded or ribbon-immersed surfaces in $S^3$.

\begin{figure}[b]\center
\includegraphics[width=.825\linewidth]{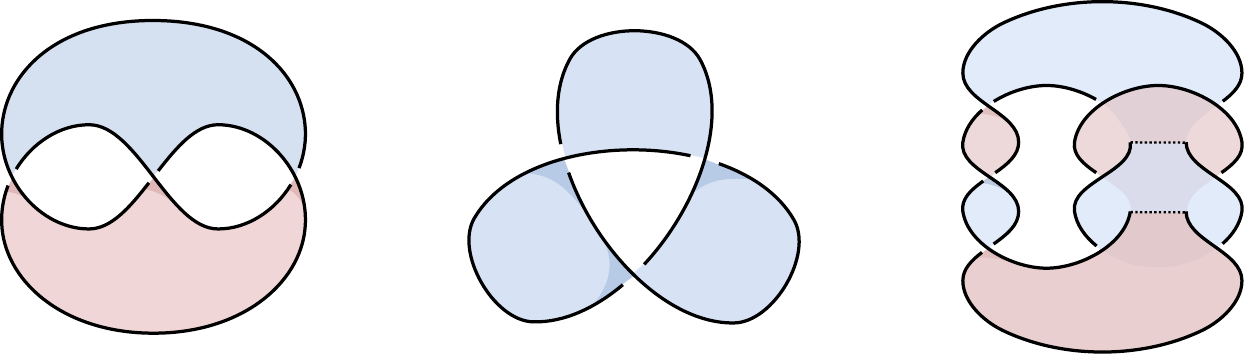}
\captionsetup{width=\linewidth}
\caption{A Seifert surface for the trefoil $3_1$ (left), a non-orientable spanning surface for $3_1$ (middle), and a slice disk bounded by the pretzel knot $P(-3,3,-3)$, also called $m(9_{46})$ (right).}\label{fig:embedded-ribbon}
\end{figure}

\item Draw movies for the slice disks indicated in Figure~\ref{fig:slice-band}, which are depicted as knot diagrams decorated with bands such that performing the associated band move transforms the knot into an unlink, which can be capped off with disks.

\begin{figure}
\center 
\def\svgwidth{.55\linewidth}
\begingroup%
  \makeatletter%
  \providecommand\color[2][]{%
    \errmessage{(Inkscape) Color is used for the text in Inkscape, but the package 'color.sty' is not loaded}%
    \renewcommand\color[2][]{}%
  }%
  \providecommand\transparent[1]{%
    \errmessage{(Inkscape) Transparency is used (non-zero) for the text in Inkscape, but the package 'transparent.sty' is not loaded}%
    \renewcommand\transparent[1]{}%
  }%
  \providecommand\rotatebox[2]{#2}%
  \newcommand*\fsize{\dimexpr\f@size pt\relax}%
  \newcommand*\lineheight[1]{\fontsize{\fsize}{#1\fsize}\selectfont}%
  \ifx\svgwidth\undefined%
    \setlength{\unitlength}{1791.81492603bp}%
    \ifx\svgscale\undefined%
      \relax%
    \else%
      \setlength{\unitlength}{\unitlength * \real{\svgscale}}%
    \fi%
  \else%
    \setlength{\unitlength}{\svgwidth}%
  \fi%
  \global\let\svgwidth\undefined%
  \global\let\svgscale\undefined%
  \makeatother%
  \begin{picture}(1,0.41738986)%
    \lineheight{1}%
    \setlength\tabcolsep{0pt}%
    \put(0,0){\includegraphics[width=\unitlength,page=1]{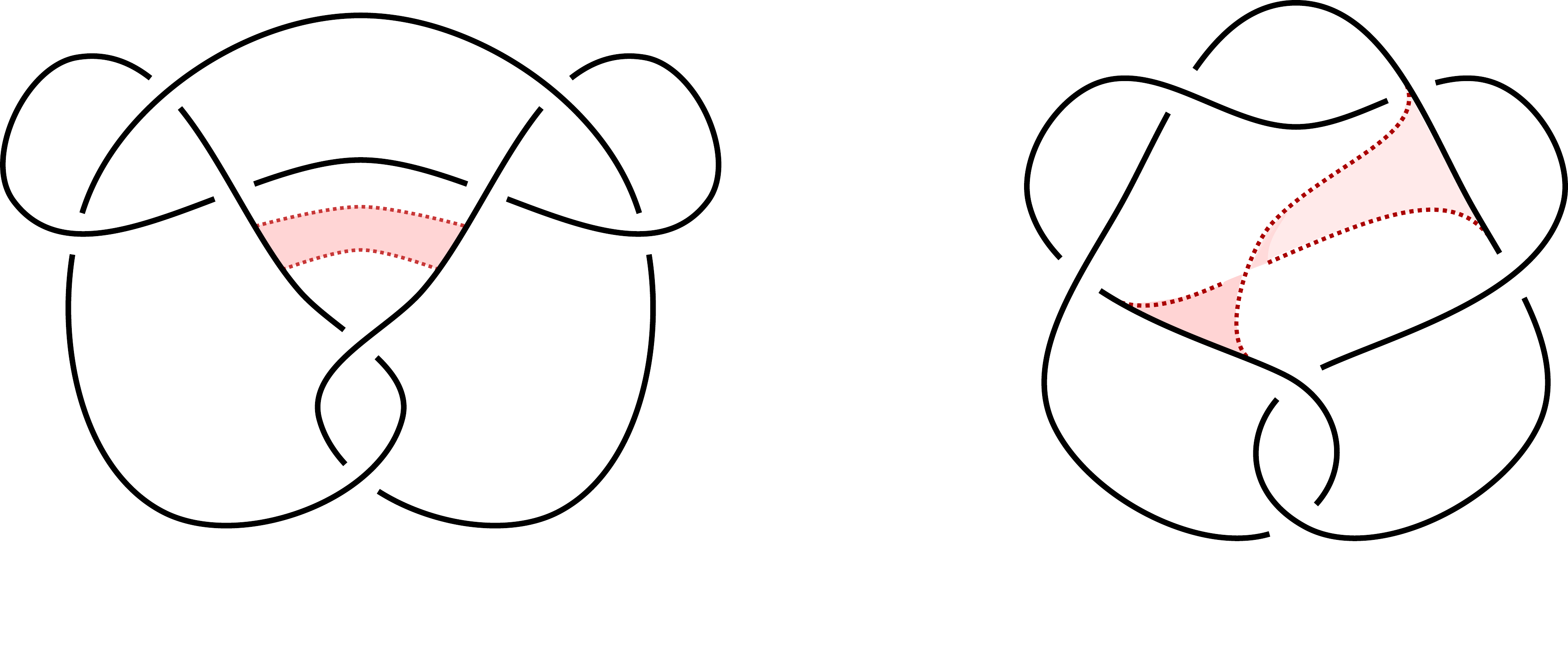}}%
    \put(0.83481216,0.00692512){\color[rgb]{0,0,0}\makebox(0,0)[t]{\smash{\begin{tabular}[t]{c}$6_1$\end{tabular}}}}%
    \put(0.23662571,0.01223317){\color[rgb]{0,0,0}\makebox(0,0)[t]{\smash{\begin{tabular}[t]{c}$8_{20}$\end{tabular}}}}%
  \end{picture}%
\endgroup%

\caption{Band moves encoding slice disks for the knots $8_{20}$ and $6_1$.}\label{fig:slice-band}

\end{figure}

\item Find a diagram of $6_1$ such that the associated movie for the aforementioned disk involves no Reidemeister III moves. (How do we make sure it is the same disk?)
\end{enumerate}
\end{exercise}

\begin{exercise} This problem introduces you to using SnapPy \cite{snappy} and KnotJob \cite{knotjob} to calculate link homologies:

\begin{quote}
{\small $\bullet$ } {\normalsize\href{https://snappy.computop.org/installing.html}{snappy.computop.org/installing.html}}

\vspace{-3pt}

{\small $\bullet$ } {\normalsize\href{https://www.maths.dur.ac.uk/users/dirk.schuetz/knotjob.html}{maths.dur.ac.uk/users/dirk.schuetz/knotjob.html}}
\end{quote}

\begin{enumerate}[label=\bfseries(\alph*),leftmargin=20pt]
\item Use SnapPy to find the knot Floer homology of the knots $3_1$, $4_1$, $6_1$, and $9_{46}$. SnapPy has great documentation, but here's a head start on the commands:

\begin{quote}\normalsize
\texttt{K = Link(`3\_1')}

\vspace{-0.075in}

\texttt{K.knot\_floer\_homology()}
\end{quote}
\medskip

\item Use KnotJob to find Khovanov homology over $\zz$ for $3_1$, $4_1$, $6_1$, and $9_{46}$. 


For example, click \compfont{File $\to$ New Link $\to$ Torus Link}. Set $(p,q)=(3,2)$ for $3_1$, then click \compfont{OK}. Select \compfont{T(3,2)} from the side menu, then click \compfont{Char 0 Khovanov Cohomology $\to$ Integral $\to$ OK}. After the calculation finishes, click the \compfont{unreduced} button next to \compfont{Khovanov homology}.

 \emph{\small Note: Observe that $6_1$ and $9_{46}$ satisfy $\widehat{\hfl}(6_1) \cong \widehat{\hfl}(9_{46})$ and $\kh(6_1) \cong \kh(9_{46})$.}


 \smallskip
 
\item At the chain level, recall the duality isomorphism $\ckh(-K) \cong \left(\ckh(K)\right)^*$, where
$$
\ckh^{h,q}(-K) \cong \operatorname{Hom}_\rcal(\ckh^{-h,-q}(K), \rcal).
$$
Use this and the Universal Coefficients Theorem to find $\kh(-3_1)$, based the above calculation of $\kh(3_1)$. Then check your answer using KnotJob (via \compfont{New Link $\to$ Mirror Links $\to$ T(3,2)}).

\vspace{-0.04in}

\emph{Note: Keep in mind that $\ckh$ is a \underline{co}chain complex, i.e., $(h,q) \overset{\partial}{\rightsquigarrow} (h+1,q)$.}

\end{enumerate}

\end{exercise}

\smallskip

\begin{exercise}
Show that any slice disk $D \subset B^4$ induces nonzero maps on both $\widehat\hfl$ and $\kh$ (whether it is viewed as a cobordism $\emptyset \to \partial D$ or $\partial D \to \emptyset$).
\end{exercise}



\smallskip
\smallskip

\begin{exercise}
Consider the knot $K=10_{140}$, which is the pretzel knot $P(3,4,-3)$.

\begin{enumerate}[label=\bfseries(\alph*),leftmargin=25pt]
\item Show that $K$ bounds a slice disk.
\item Look up or calculate $\kh(K)$, then prove that every slice disk for $K$ induces the same map on Khovanov homology up to overall sign (whether viewed as a  cobordism $\emptyset \to K$ or $K \to \emptyset$).


\emph{Hint: This holds for any slice disk bounded by knot with sufficiently simple $\kh$. It may be useful to first prove it for disks bounded by the unknot.}

\end{enumerate}

\end{exercise}

\bigskip
\bigskip
\bigskip

\section{Khovanov homology}\label{sec:khovanov}

\medskip


This section provides an introduction to the basics of Khovanov homology. Our goal is to give a condensed exposition that provides enough information to discuss the desired properties and applications of the cobordism maps. We draw heavily on \cite{kwz:immersed,hayden-sundberg}. We begin in \S\ref{subsec:khovanov} by  defining the Khovanov chain complex, then proceed to discuss the maps induced by link cobordisms in \S\ref{subsec:induced}. Our goal is to provide a number of concrete examples and highlight some of the applications from \S\ref{hsec:intro}. 



\smallskip

\subsection{The Khovanov chain complex}\label{subsec:khovanov}


\subsection*{Cube of resolutions } Let $\dcal$ be an $n$-crossing diagram of a link $L$. A \textit{smoothing} of $\dcal$  is a planar $1$-manifold obtained by replacing each crossing $\crossing$ in $D$ with either a $0$-smoothing $\zsmooth$ or a $1$-smoothing $\osmooth$ \!; see Figure~\ref{fig:resolutions}(a). (For an example with the right-handed trefoil, see Figure~\ref{fig:tref-complex}.)  
After fixing an ordering of the crossings of $\dcal$, we may associate to each vertex $v \in \{0,1\}^n$ of the $n$-dimensional cube a smoothing $\dcal(v)$ of $\dcal$   by applying the $v_i$-resolution to the $i^{\text{th}}$ crossing of $\dcal$ for $i = 1,\ldots,n$.

In this \emph{cube of resolutions}, each edge can be viewed as an arrow $v\to v'$ between vertices that differ at a single coordinate $i$, with $v_i=0$ and $v'_i=1$. To this edge  we assign a cobordism from $\dcal(v)$ to $\dcal(v')$ given by a saddle cobordism in a neighborhood of the $i^{\text{th}}$ as in Figure~\ref{fig:resolutions}(b), and given by a product outside the crossing region.

\medskip

\begin{figure}
\center
\def\svgwidth{.8\linewidth}
\begingroup%
  \makeatletter%
  \providecommand\color[2][]{%
    \errmessage{(Inkscape) Color is used for the text in Inkscape, but the package 'color.sty' is not loaded}%
    \renewcommand\color[2][]{}%
  }%
  \providecommand\transparent[1]{%
    \errmessage{(Inkscape) Transparency is used (non-zero) for the text in Inkscape, but the package 'transparent.sty' is not loaded}%
    \renewcommand\transparent[1]{}%
  }%
  \providecommand\rotatebox[2]{#2}%
  \newcommand*\fsize{\dimexpr\f@size pt\relax}%
  \newcommand*\lineheight[1]{\fontsize{\fsize}{#1\fsize}\selectfont}%
  \ifx\svgwidth\undefined%
    \setlength{\unitlength}{1204.93271715bp}%
    \ifx\svgscale\undefined%
      \relax%
    \else%
      \setlength{\unitlength}{\unitlength * \real{\svgscale}}%
    \fi%
  \else%
    \setlength{\unitlength}{\svgwidth}%
  \fi%
  \global\let\svgwidth\undefined%
  \global\let\svgscale\undefined%
  \makeatother%
  \begin{picture}(1,0.25411272)%
    \lineheight{1}%
    \setlength\tabcolsep{0pt}%
    \put(0,0){\includegraphics[width=\unitlength,page=1]{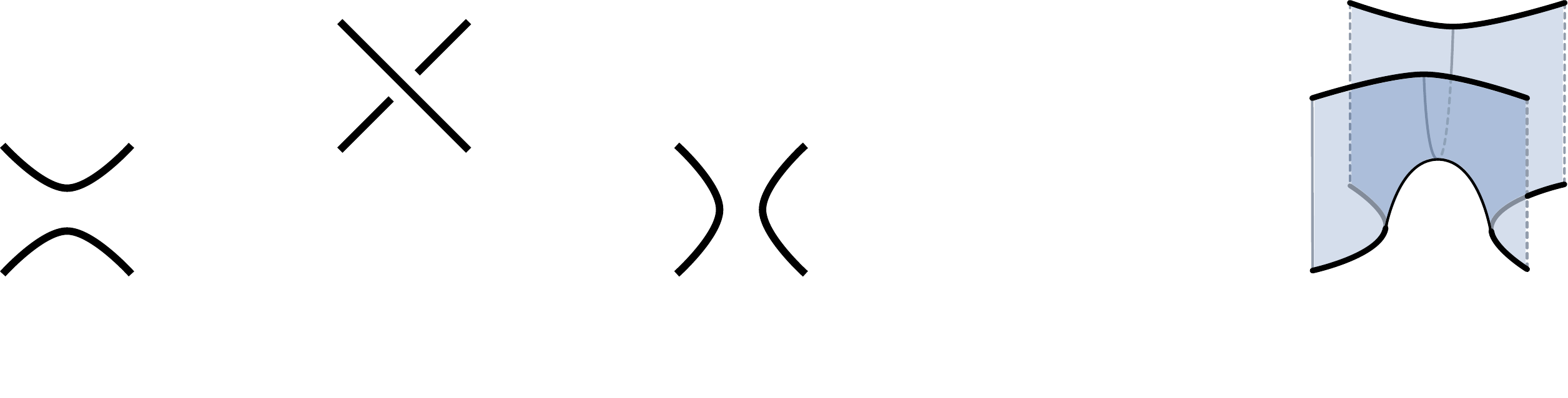}}%
    \put(0.25787223,0.00497414){\color[rgb]{0,0,0}\makebox(0,0)[t]{\smash{\begin{tabular}[t]{c}(a)\end{tabular}}}}%
    \put(0.14407513,0.17762925){\color[rgb]{0,0,0}\makebox(0,0)[t]{\smash{\begin{tabular}[t]{c}$0$\end{tabular}}}}%
    \put(0.36902266,0.17762925){\color[rgb]{0,0,0}\makebox(0,0)[t]{\smash{\begin{tabular}[t]{c}$1$\end{tabular}}}}%
    \put(0.91766142,0.00497414){\color[rgb]{0,0,0}\makebox(0,0)[t]{\smash{\begin{tabular}[t]{c}(b)\end{tabular}}}}%
    \put(0,0){\includegraphics[width=\unitlength,page=2]{resolutions.pdf}}%
  \end{picture}%
\endgroup%

\caption{Crossing resolutions, and a saddle cobordism between them.}\label{fig:resolutions}
\end{figure}

\begin{figure}\center
\bigskip
\bigskip

\def\svgwidth{.925\linewidth}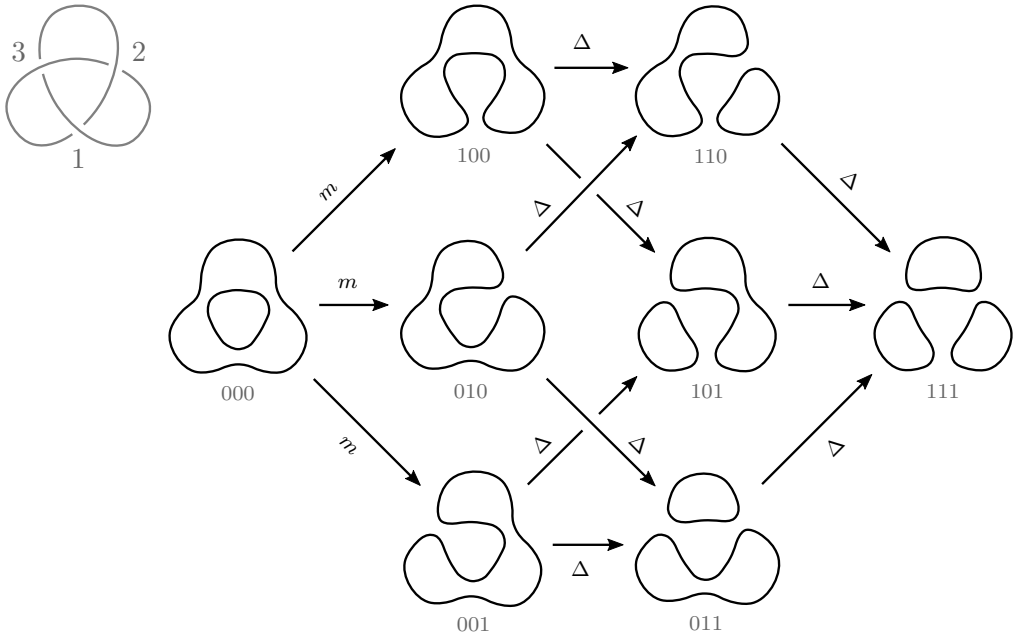
\caption{The cube of resolutions for a diagram of the trefoil, with a choice of enumeration of the crossings. Smoothings are labeled with their binary coordinates in $\{0,1\}^{n=3}$ and arrows are labeled as merges ($m$) or splits ($\Delta$).}\label{fig:tref-complex}
\end{figure}



\subsection*{Chain complex } 
Viewing the cube of resolutions as a collection of closed 1-manifolds and 2-dimensional cobordisms between them, we now wish to convert the 1-manifolds to $\rcal$-modules and the 2-dimensional cobordisms to $\rcal$-module homomorphisms as follows: 
\begin{itemize}
\item Each circle in a smoothing is assigned the free $\rcal$-module $\acal=\rcal \langle \bfo,\bfx\rangle$.

\item Each smoothing $\dcal(v)$ is assigned the tensor product $\acal_{\dcal(v)}=\acal^{\otimes | \dcal(v)|}$ of the copies of $\acal$ assigned to its component circles. 


\item For each edge $v\to v'$, the corresponding cobordism $\dcal(v) \to \dcal(v')$ is assigned a map $\acal^{\otimes |\dcal(v)|} \to \acal^{\otimes |\dcal(v')|}$ by applying the following rules to its connected components:
%

\begin{align*}
\raisebox{-.28cm}{\includegraphics[height=0.65cm]{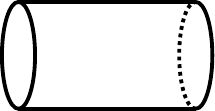}}  \qquad \quad \ \, 
\id: \acal \to \acal \, ; \quad \ 
& \ \id(\mathbf{a})=\mathbf{a}
\tag{identity}
\\
\\
\raisebox{-.51cm}{\includegraphics[height=1.25cm]{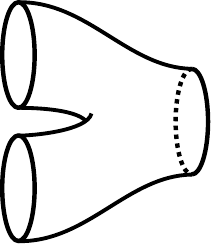}} \qquad 
m: \acal \otimes \acal \to \acal \,; \quad \ 
&\begin{cases} m(\bfo \otimes \bfo) = \bfo  
\\
m(\bfo \otimes \bfx) = \bfx
\\
m(\bfx \otimes \bfo)=\bfx
\\
m(\bfx \otimes \bfx)= 0
\end{cases}
\tag{merge}
\\
\\
\raisebox{-.51cm}{\includegraphics[height=1.25cm]{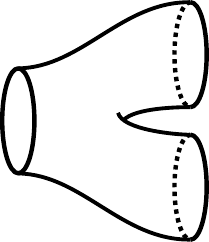}} \qquad 
\Delta: \acal \to \acal \otimes \acal \, ; \quad \ 
&\begin{cases} \Delta(\bfo) = \bfo \otimes \bfx + \bfx \otimes \bfo 
\\
\Delta(\bfx) = \bfx \otimes \bfx
\end{cases}
\tag{split}
\end{align*} 


\medskip

\hspace{1em}  With these definitions alone, each 2-dimensional face of the cube would induce a commutative square. To obtain a differential when $2 \neq 0$ in $\rcal$, we must sprinkle signs across the cube to ensure these squares anti-commute. 
This can be achieved by multiplying the map induced by $\dcal(v)\to \dcal(v')$ by $(-1)^{v_1 +\cdots + v_{i-1}}$. Now we define the differential $d$ at each vertex of the cube by summing over all the outgoing (signed) edge maps.
\end{itemize}
\medskip

\noindent We obtain a chain complex $\ckh(\dcal)$ with underlying $\rcal$-module $\boldsymbol{\oplus}_{v \in \{0,1\}^n} \, \acal_{\dcal(v)}$ and whose differential $d$ is defined by the signed edge maps above.  Taking homology yields a $\rcal$-module that we denote by $\kh(L)$.

\begin{theorem}[Khovanov \cite{Khovanov}]
The isomorphism type of the $\rcal$-module $\kh(L)$ is an invariant of $L$ up to isotopy.
\end{theorem}


We will say a bit more about the proof of invariance in \S\ref{sec:tqft}. 

\begin{remark}
We will define a $\zz\oplus \zz$-bigrading on $\ckh(\dcal)$ and thus $\kh(L)$ shortly below, and the proof of invariance holds at the level of bigraded $\rcal$-modules. (As defined up to this point, $\ckh(\dcal)$ has a \emph{cubical filtration} arising from the lexicographic partial order on the cube's vertices $v\in \{0,1\}^n$; the differential is filtered in the sense that the cube has a directed edge $v \to v'$ only if $v \leq v'$. The homological grading on $\ckh(\dcal)$ will arise from a ``flattening'' of this partial order.)
\end{remark}

Note that we may think of the generators of the chain complex as smoothings of $\dcal$ whose components are labeled with $\bfo$'s and $\bfx$'s. For convenience, we sometimes further decorate these with light gray arcs to indicate 0-resolutions as in Figure~\ref{fig:tref-chains} below.  For a generator represented by a single labeled smoothing, there is a simple criterion for checking whether it is a cycle (c.f., \cite[Prop. 3.2]{elliott}):

\begin{figure}
\center
\def\svgwidth{\linewidth}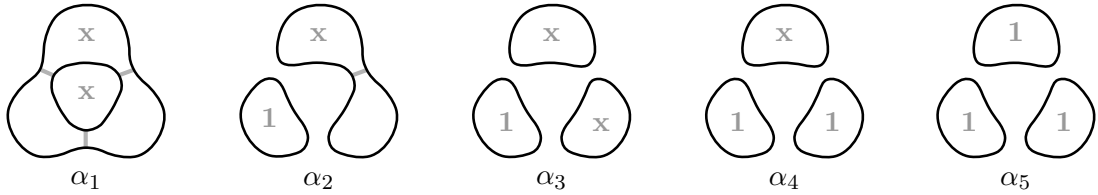
\caption{Examples of generators in $\ckh(3_1)$.}\label{fig:tref-chains}
\end{figure}

 \begin{proposition}
A labeled smoothing $\alpha$ is a cycle if and only if every $0$-resolution in $\alpha$, when changed to a $1$-resolution, merges two $\bfx$-labeled loops.
 \end{proposition}


\begin{example}
Consider the chain elements $\alpha_1,\ldots,\alpha_5 \in \ckh(3_1)$ depicted in Figure~\ref{fig:tref-chains}; here we use light gray arcs to indicate 0-resolutions. The chains $\alpha_1$, $\alpha_3$, $\alpha_4$, and $\alpha_5$ are cycles, whereas $\alpha_2$ is not. Observe that $\alpha_3$ is the boundary of $\alpha_2$ (up to sign). On the other hand $\alpha_1$, $\alpha_4$, and $\alpha_5$ are not boundaries, hence define nonzero classes in $\kh(3_1)$. One can further show that $\alpha_4$ represents a 2-torsion class (at least over $\rcal=\zz$).
\end{example}

\smallskip
 
 

\subsection*{Gradings }
The chain complex carries a \textit{homological} grading $h$ and \textit{quantum} grading $q$, decomposing as $$\ckh(\dcal)=\bigoplus_{(h,q)\in \zz \times \zz} \ckh^{h,q}(\dcal),$$ where the differential shifts $h$ by $+1$ and preserves $q$. To define these gradings, first let $n_+$ and $n_-$ denote the number of positive and negative crossings in $\dcal$, respectively. A generator $\alpha$ of $\ckh(\dcal)$ can be viewed as an elementary tensor in $\boldsymbol{\oplus}_{v \in \{0,1\}^n} \, \acal_{\dcal(v)}$ or as a smoothing of $\dcal$ with components labeled by $\bfo$ or $\bfx$. Let $|\alpha|$ denote the number of $1$-resolutions in $\alpha$, and $v_+$ and $v_-$ denote the number of $\bfo$-labels and $\bfx$-labels in $\alpha$. Then $h$ and $q$ are defined on $\alpha$ by
\begin{align*}
	h(\alpha) &= |\alpha| - n_- \\
	q(\alpha) &= v_+(\alpha) - v_-(\alpha) + h(\alpha) + n_+ - n_- 
	\\
	&= \deg(\alpha) + h(\alpha) + n_+ + n_-,
\end{align*}
where $\deg(\alpha)$ is the grading on $\acal_{\dcal(v)}$ defined in Section~\ref{sec:tqft}.


\begin{example}
The generators $\alpha_1$, $\alpha_2$, $\alpha_3$, $\alpha_4$, and $\alpha_5$ in $\ckh(3_1)$ depicted in Figure~\ref{fig:tref-chains} have bigradings $(h,q)$ given by $(0,1)$, $(2,5)$, $(3,5)$, $(3,7)$, and $(3,9)$, respectively.
\end{example}


\begin{exercise}Calculate the Khovanov homology of:
\begin{enumerate}[label=\bfseries(\alph*),leftmargin=25pt,itemsep=-3pt]
\item the unknot, using your favorite diagram with one crossing;

\item the Hopf link, with your favorite choice of orientation; and

\item the right-handed trefoil, using the cube of resolutions in Figure~\ref{fig:tref-complex}.
\end{enumerate}
\end{exercise}


\subsection{Induced maps on Khovanov homology} \label{subsec:induced}

\smallskip

We now turn to the central tool in this discussion: the cobordism maps on Khovanov homology. These maps were first defined for link cobordisms in $\rr^3 \times I$ by Khovanov \cite{Khovanov} and proven to be invariant (up to sign) under isotopy rel boundary by Jacobsson \cite{jacobsson}; alternative proofs of invariance were given by Bar-Natan \cite{barnatan} and Khovanov \cite{khovanov:invariant}. The invariance under isotopy rel boundary in $S^3 \times I$ was established later by Morrison--Walker--Wedrich \cite{morrison-walker-wedrich}.

\smallskip

\begin{theorem}
Let $\Sigma \subset S^3 \times I$ be a link cobordism between links $L_0,L_1 \subset S^3$. There  is an induced map $\kh(\Sigma): \kh(L_0) \to \kh(L_1)$ that is bigraded of degree $(h,q)=(0,\chi(\Sigma))$,  well-defined up to sign, and invariant under isotopy of $\Sigma$ rel boundary.
\end{theorem}

\smallskip

In this subsection, we will focus on the definition of the cobordism maps; we will say a few words about the proofs of invariance in \S\ref{sec:tqft}. For the remainder of this subsection, we suppose that the  link cobordism $\Sigma: L_0\to L_1$ is expressed as a movie of link diagrams $\dcal_0=\dcal_{t_0},\dcal_{t_1},\ldots,\dcal_{t_n}=\dcal_1$ as discussed in \S\ref{subsec:functoriality}.   Following the discussion there, cobordism chain maps are defined based on the maps assigned to births (i.e., local minima), saddles (either merging or splitting), deaths (i.e., local maxima), Reidemeister moves (either I, II, or III), and planar isotopy.


\medskip


\emph{Critical points and planar isotopy}

For planar isotopy, the induced chain map is the obvious identification obtained by applying the planar isotopy at the level of smoothings. For cobordisms corresponding to critical points, we first note that we may generically assume these critical points (in particular, saddles) occur away from crossings in the diagram. For saddles, we apply the merge and split maps just as defined above. For births and deaths, we introduce two more maps:
\begin{align*}
\raisebox{-.28cm}{\includegraphics[height=0.715cm]{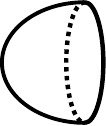}} \quad \qquad \qquad 
\iota: \rcal \to \acal \, ; \quad \ 
& \quad \iota(1) = \bfo \tag{birth}
\\
\\
\raisebox{-.28cm}{\includegraphics[height=0.715cm]{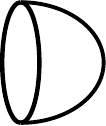}} \quad \qquad \qquad 
\epsilon: \acal \to \rcal \, ; \quad \ 
&\begin{cases} \epsilon(\bfo) = 0 
\\
\epsilon(\bfx) = 1
\end{cases}
\tag{death}
\end{align*}

\medskip


\noindent That is, a birth introduces a $\bfo$-labeled circle, and a death maps a $\bfo$-labeled circle to $0$ and an $\bfx$-labeled circle to $1$. For notational convenience, we borrow Bar-Natan's \emph{ornaments}, denoting births by $\kcup$, deaths by $\kcap$, and saddles by $\ksmooth$; see Figure~\ref{fig:can-cobs}. 

\begin{figure}[h!]\center
\bigskip
\raisebox{-.15in}{\def\svgwidth{.684\linewidth}
\begingroup%
  \makeatletter%
  \providecommand\color[2][]{%
    \errmessage{(Inkscape) Color is used for the text in Inkscape, but the package 'color.sty' is not loaded}%
    \renewcommand\color[2][]{}%
  }%
  \providecommand\transparent[1]{%
    \errmessage{(Inkscape) Transparency is used (non-zero) for the text in Inkscape, but the package 'transparent.sty' is not loaded}%
    \renewcommand\transparent[1]{}%
  }%
  \providecommand\rotatebox[2]{#2}%
  \newcommand*\fsize{\dimexpr\f@size pt\relax}%
  \newcommand*\lineheight[1]{\fontsize{\fsize}{#1\fsize}\selectfont}%
  \ifx\svgwidth\undefined%
    \setlength{\unitlength}{822.90469673bp}%
    \ifx\svgscale\undefined%
      \relax%
    \else%
      \setlength{\unitlength}{\unitlength * \real{\svgscale}}%
    \fi%
  \else%
    \setlength{\unitlength}{\svgwidth}%
  \fi%
  \global\let\svgwidth\undefined%
  \global\let\svgscale\undefined%
  \makeatother%
  \begin{picture}(1,0.38004671)%
    \lineheight{1}%
    \setlength\tabcolsep{0pt}%
    \put(0,0){\includegraphics[width=\unitlength,page=1]{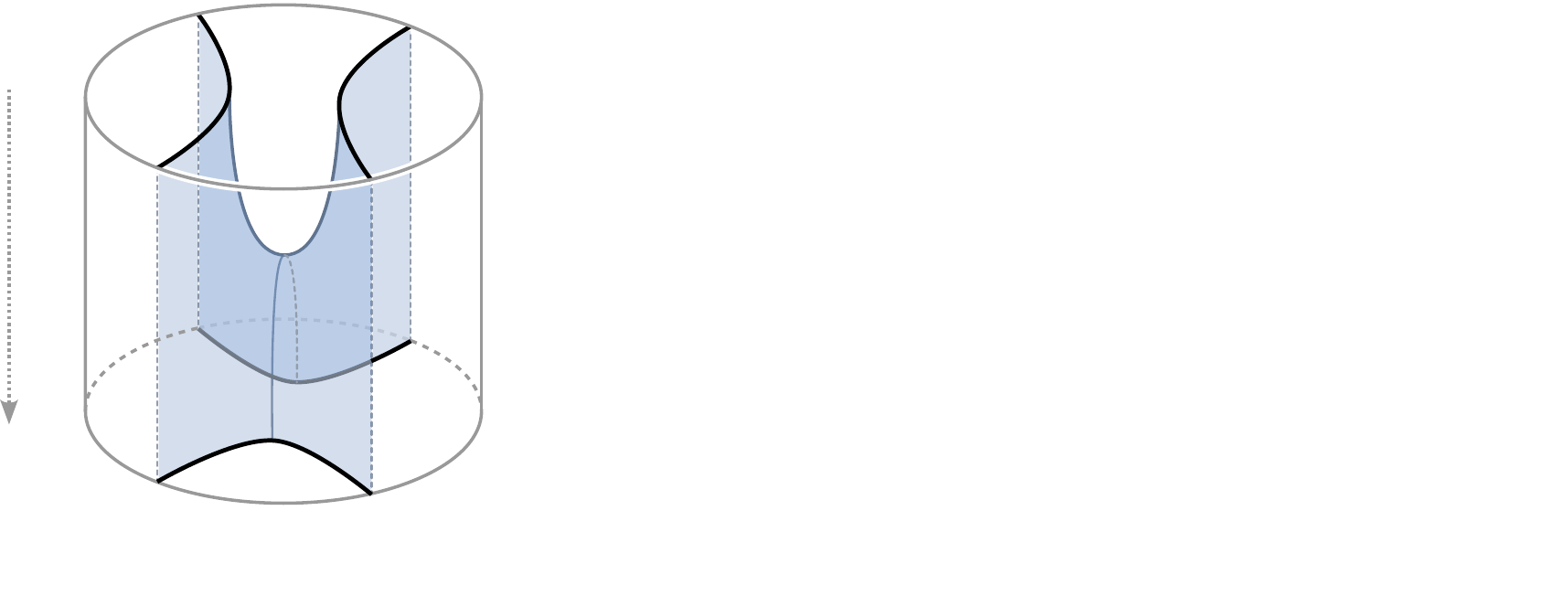}}%
    \put(0.18085751,0.00529278){\color[rgb]{0,0,0}\makebox(0,0)[t]{\smash{\begin{tabular}[t]{c}$\ksmooth$\end{tabular}}}}%
    \put(0,0){\includegraphics[width=\unitlength,page=2]{can-cob.pdf}}%
    \put(0.53481345,0.00529278){\color[rgb]{0,0,0}\makebox(0,0)[t]{\smash{\begin{tabular}[t]{c}$\kcup$\end{tabular}}}}%
    \put(0,0){\includegraphics[width=\unitlength,page=3]{can-cob.pdf}}%
    \put(0.87054124,0.00529278){\color[rgb]{0,0,0}\makebox(0,0)[t]{\smash{\begin{tabular}[t]{c}$\kcap$\end{tabular}}}}%
  \end{picture}%
\endgroup%
} \quad 
\captionsetup{width=.9\linewidth}

 \caption{Saddle, birth, and death cobordisms, along with  ornaments \raisebox{-1pt}{$\ksmooth$}, \raisebox{-1pt}{$\kcup$}, and \raisebox{-1pt}{$\kcap$}.  Here we follow Bar-Natan's convention of reading such diagrams top-to-bottom.}\label{fig:can-cobs}
\end{figure}


\begin{example}
For a standardly embedded torus, viewed as a cobordism $T^2 : \emptyset \to \emptyset$, we can calculate the induced map $\kh(T^2): \zz \to \zz$ as in Figure~\ref{fig:torus}.

\begin{remark}
By excluding the topmost step in Figure~\ref{fig:torus}, we see that a standard once-punctured torus, viewed as a cobordism from the empty set to the unknot, induces a map sending $1 \in \zz$ to a circle labeled $2\bfx$.
\end{remark}

\begin{figure}
\def\svgwidth{.649\linewidth}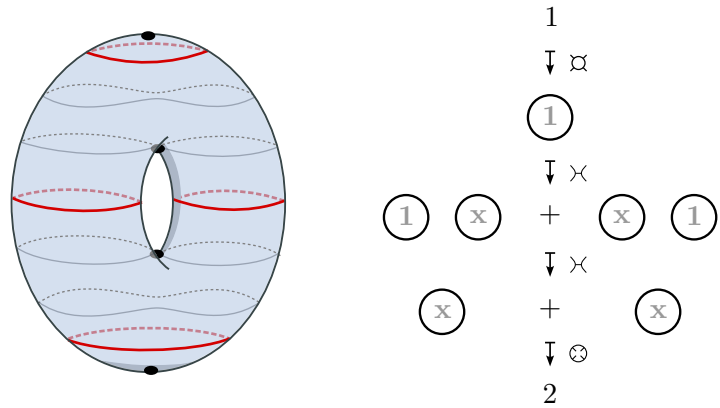
\caption{Calculating the map induced by a standardly embedded torus.}\label{fig:torus}
\end{figure}
\end{example}

\begin{wrapfigure}[8]{r}{.2\linewidth}
\vspace{-0.175in}
\hfill \includegraphics[width=.95\linewidth]{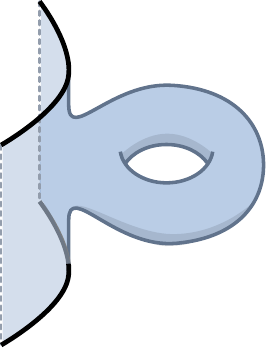}
\caption{$\blacksquare$}\label{fig:curtain}
\end{wrapfigure}

This brings us to the introduction of one additional ornament: We use a black square \raisebox{0pt}{\scalebox{0.8}{$\blacksquare$}} on a strand to denote the cobordism obtained by locally summing the product cobordism with a torus.

\begin{exercise}\label{exer:stab}
Compute the local effect of this cobordism. 
 In particular, assuming that the connected sum is performed near a point $p \in L$ that lies away from any crossing in the diagram, show that the chain level effect is to multiply the label on the component containing $p$ by $2 \bfx$. (That is, a $\bfo$-label is mapped to $2 \bfx$, and an $\bfx$-label is mapped to 0.) 
\end{exercise}

\medskip

\emph{Reidemeister moves}

To each Reidemeister move, one can associate a local cobordism that will help encode the chain map induced by the Reidemeister move; the precise role of these local cobordisms will be explained later in \S\ref{sec:tqft}. For now we will interpret these local cobordisms directly in terms of the basic cobordisms \raisebox{-0.5pt}{$\ksmooth$}, \raisebox{-0.5pt}{$\kcup$}, \raisebox{-0.5pt}{$\kcap$}, and $\ksquare$ described above. 
This framework is simplest for Reidemeister I and II moves, which will be the most essential moves for the examples we wish to consider.

Tables~\ref{table:R1} and \ref{table:R2} describe these induced maps, and are based on a combination of \cite[Tables 3-4]{hayden-sundberg} and \cite[Figures 5-6]{barnatan}. These tables are first organized by Reidemeister move (I versus II,  and increasing or decreasing crossing number). For each of these, we describe the possible local smoothings and the corresponding local cobordism to be applied; these cobordisms are directed from top to bottom, so that composition is downward. The final column of each table presents the cobordism in a shorthand notation from \cite{barnatan}. Here the use of the ornaments $\kcup$, $\kcap$, and \raisebox{0pt}{\scalebox{0.8}{$\blacksquare$}} is self-explanatory, and the shorthand diagrams include arcs to encode saddle moves $\ksmooth$. 
Note that, for shorthand diagrams involving multiple ornaments, there is either a unique order in which the constituent cobordisms can be composed, or changing the order does not affect the map.  

\medskip

\begin{example} Recall the link movie depicting the standard Seifert surface $\Sigma$ for the left-handed trefoil $-3_1$ in Figure~\ref{fig:tref-movie}, viewed as a link cobordism $\emptyset \to K$. We calculate the associated map $\zz=\kh(\emptyset) \to \kh(-3_1)$ in Figure~\ref{fig:tref-movie-Kh}: The first step is a birth (\raisebox{-0.5pt}{$\kcup$}), the next three steps are (negative) Reidemeister I moves (each realized as another \raisebox{-0.5pt}{$\kcup$}, per Table~\ref{table:R1}), and the final two steps are saddles (each merging two $\bfo$-labeled circles into one). 

We claim the element $\kh(\Sigma)(1) \in \kh(-3_1)$ is nonzero.\footnote{The class $\kh(\Sigma)(1)$ is called the \emph{Khovanov-Jacobsson} class of $\Sigma$ in \cite{sundberg-swann}.} To see this, note that in the smoothing that underlies $\kh(\Sigma)(1)$ in Figure~\ref{fig:tref-movie-Kh}, all three crossings are 1-resolved. The only incoming differentials are split maps, but the split map $\acal \to \acal \otimes \acal$ never hits the element $\bfo \otimes \bfo$.

\begin{figure}[t]
\center
\def\svgwidth{.85\linewidth}\input{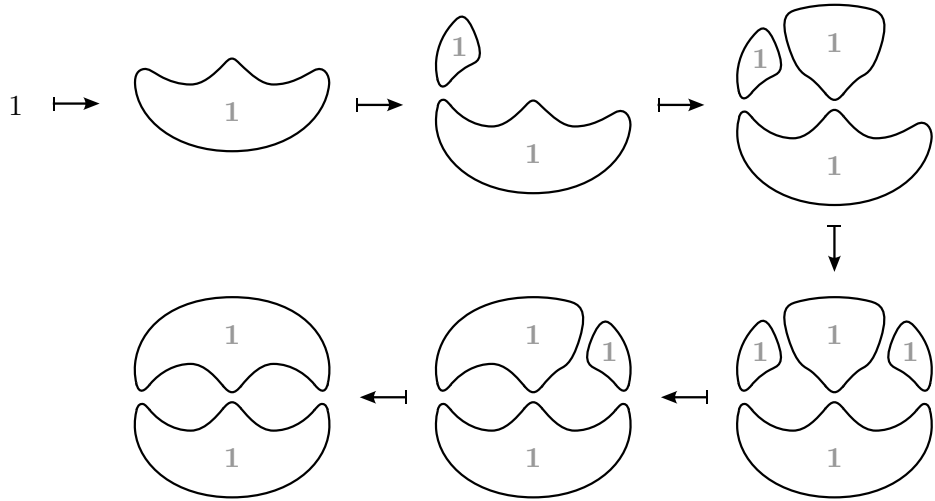}
\captionsetup{width=.85\linewidth}
\caption{Calculating the  map induced by the standard Seifert surface for $-3_1$ (viewed $\emptyset \to -3_1$).}\label{fig:tref-movie-Kh}
\end{figure}
\end{example}

\medskip

\begin{exercise}
Use Table~\ref{table:R1} to justify the following case of the Reidemeister I cobordism map, which is expressed at the level of labeled smoothings:


\begin{figure}[h!]
\def\svgwidth{\linewidth}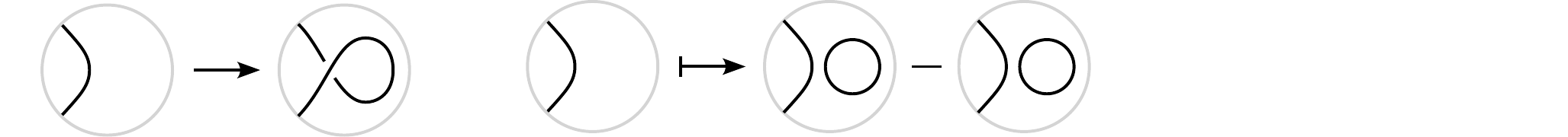
\end{figure}
\end{exercise}

\newpage



\begin{minipage}{1.1\linewidth}
\begin{table}[H]

\vspace{0.15in}

\hspace{-0.1\linewidth} \def\svgwidth{.95\linewidth}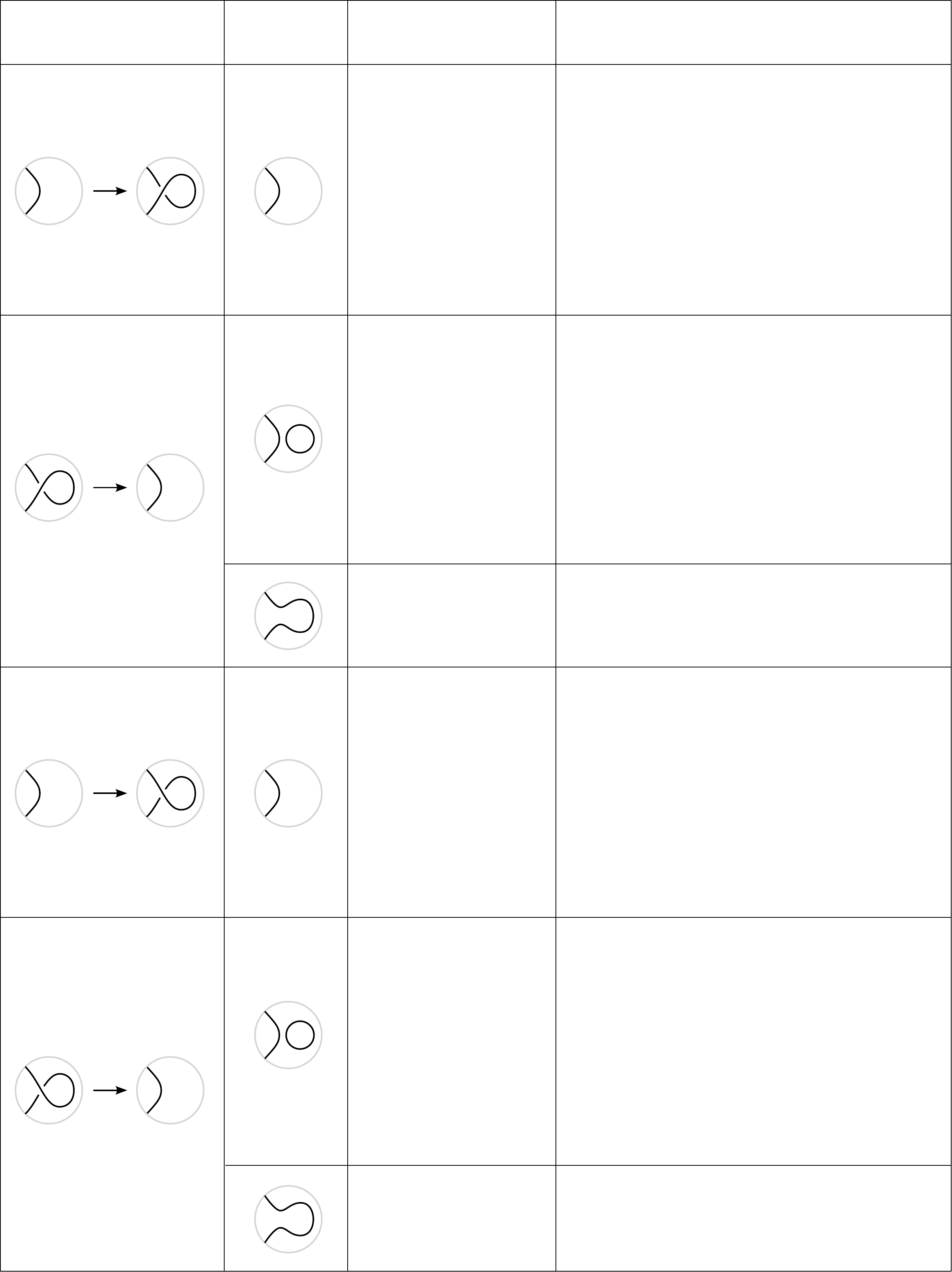

\smallskip

\caption{Chain maps induced by Reidemeister I moves.}\label{table:R1}
\end{table}
\end{minipage}

\begin{minipage}{1.05\linewidth}

\vspace{0.75in}

\begin{table}[H]\center

\hspace{-0.05\linewidth}\def\svgwidth{\linewidth}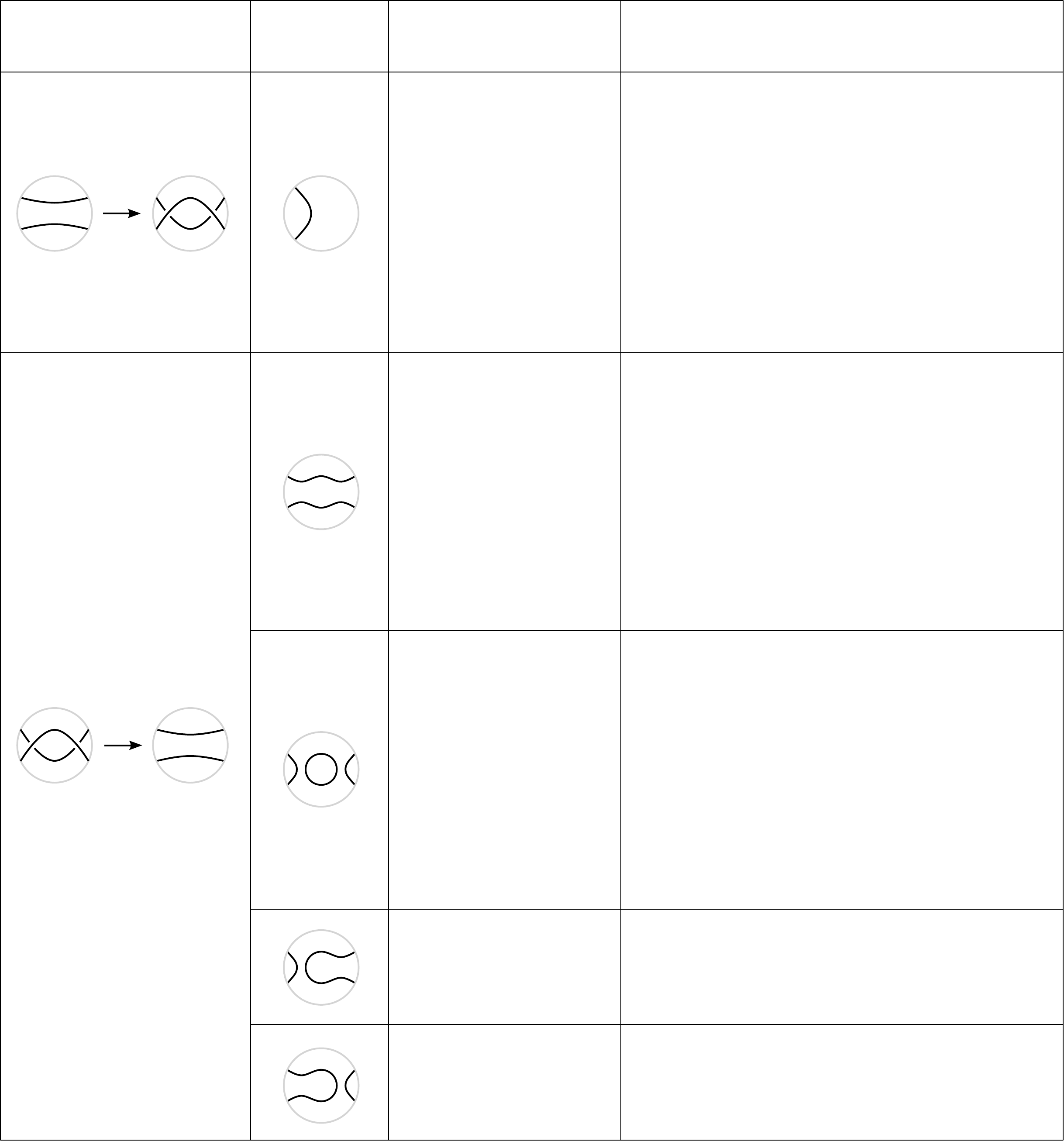

\smallskip

\caption{Chain maps induced by Reidemeister II moves.}\label{table:R2}
\end{table}
\bigskip
\end{minipage}

\newpage

\begin{exercise}
\begin{enumerate}[leftmargin=25pt,itemsep=-3pt,label=\bfseries(\alph*)]
\item Show that the standard Seifert surface for the right-handed trefoil $3_1$ induces the zero map when viewed as $\emptyset \to 3_1$. (Hint: Use the previous exercise to simplify your calculation.)


\vspace{-0.03in}

\item Show that it induces a nontrivial map when viewed as $3_1 \to \emptyset$.
\end{enumerate}
\end{exercise}


\begin{example}
We now distinguish the cobordism maps induced by the mirrors $-D$ and $-D'$ of the slice disks from Figure~\ref{fig:946}, viewed as cobordisms from the mirrored knot $-9_{46}$ to the empty set. Let $\phi \in \kh(-9_{46})$ denote the element represented by the chain element in the middle of Figure~\ref{fig:946-movie-other}. (The reader should verify that this chain element is indeed a cycle.) We claim that $\kh(D)(\phi)=1$ yet $\kh(-D')(\phi)=0$, which will imply that $-D$ and $-D'$ are not smoothly isotopic rel boundary.  

We begin with the easier calculation: $\kh(-D')(\phi)=0$. To see this, consider the first diagram in  Figure~\ref{fig:946-movie-other}, which shows a band representing the initial saddle move in a  movie diagram of the cobordism. At the chain level, this saddle induces a merge between two distinct $\bfx$-labeled circles, hence the cobordism map sends $\phi$ to 0.

On the other hand, consider Figure~\ref{fig:946-movie}. It   expresses $-D$ as a movie of link diagrams, together with the image of $\phi$ under the map induced by each successive step of the cobordism. It shows that $\kh(-D)(\phi)=1$. \hfill $\diamond$

\begin{figure}[b]
\center
\def\svgwidth{.825\linewidth}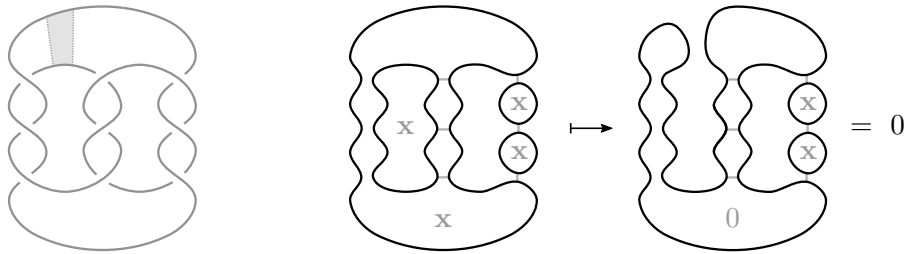
\caption{Calculating the image of $\phi \in \kh(-9_{46})$ under the map induced by $-D'$.}\label{fig:946-movie-other}
\end{figure}

\begin{figure}
\center
\def\svgwidth{.85\linewidth}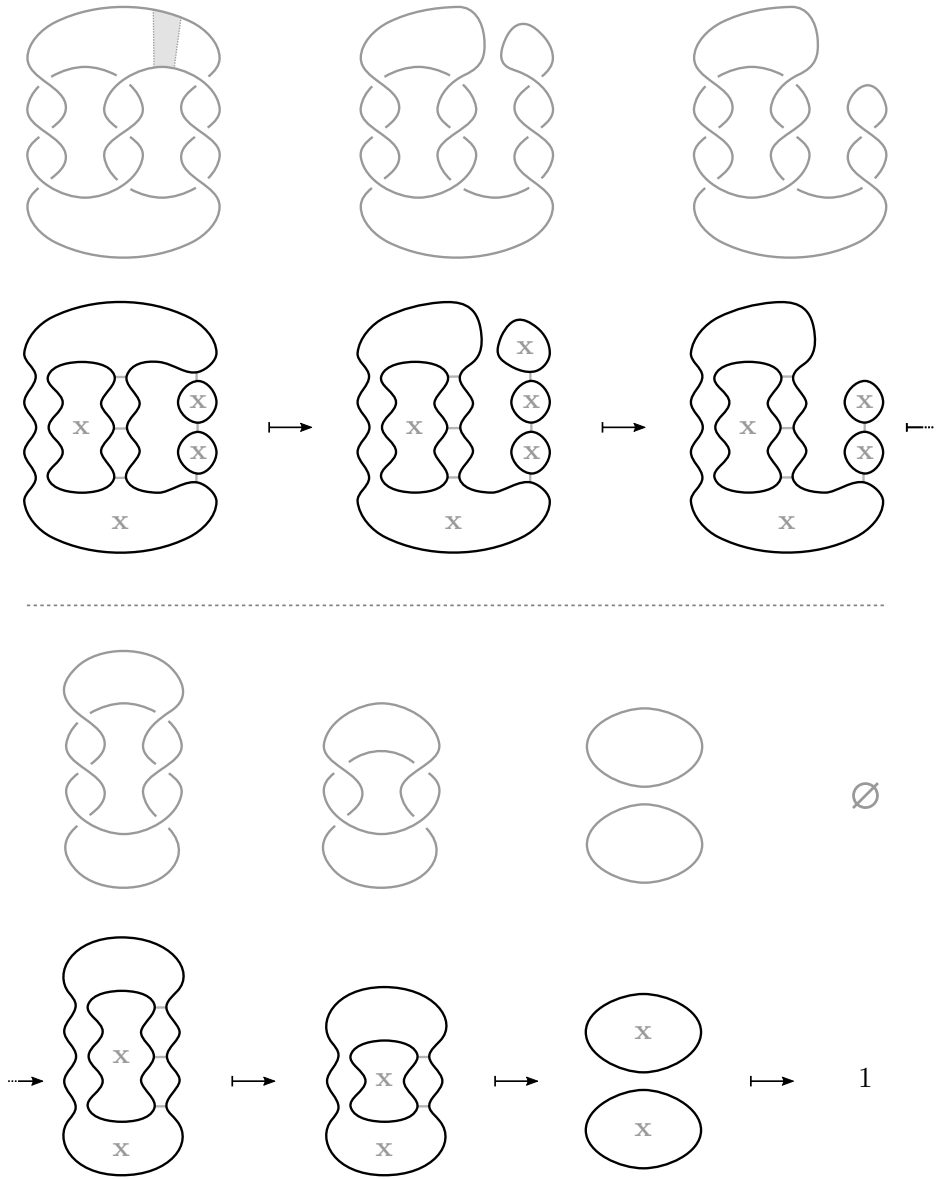

\medskip

\caption{Calculating the image of $\phi \in \kh(-9_{46})$ under the map induced by $-D$.}\label{fig:946-movie}

\end{figure}

\end{example}

Let us now walk through how one might independently arrive at the class $\phi$ used to distinguish the maps $\kh(-D)$ and $\kh(-D')$. These maps shift the bigrading by $$(h,q)=(0,\chi(-D))=(0,\chi(-D'))=(0,1),$$ and the target $\kh(\emptyset)$ is supported in bigrading $(0,0)$. Thus we must consider classes in $\kh^{0,-1}(-9_{46})$. A natural starting point in homological grading $h=0$ is the \emph{oriented resolution}; this is equivalent to choosing 0-resolutions at positive crossings and 1-resolutions at negative crossings. See the first step in Figure~\ref{fig:946-explore}. If our (perhaps overly optimistic) goal is to find an effective class that can be represented as a single labeled smoothing, then we must assign $\bfx$-labels to all circles joined by 0-resolution arcs. In the example shown in Figure~\ref{fig:946-explore}, it will then follow that the only way to achieve the desired quantum grading $q=-1$ is to assign $\bfo$-labels to the remaining two circles that are untouched by 0-resolution arcs.  This yields an element $\psi \in \ckh^{0,-1}(-9_{46})$ that is, by construction, a cycle. However, the class represented by $\psi$  cannot possibly distinguish the maps induced by $-D$ and $-D'$ because $\psi$ is symmetric with respect to the symmetry that relates $-D$ and $-D'$.

\begin{figure}
\center
\def\svgwidth{.85\linewidth}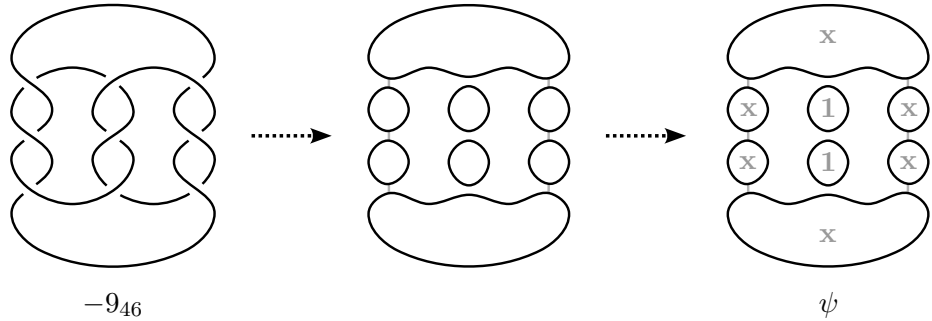
\caption{Producing the class $\psi \in \kh(-9_{46})$.}\label{fig:946-explore}
\end{figure}

\begin{exercise}
Show that the maps induced by $-D$ and $-D'$ both send $\psi$ to $\pm1$.
\end{exercise}

To find a labeled smoothing that behaves \emph{differently} under the two maps, we must change our underlying resolution. To preserve the homological grading, we must switch 0- and 1-resolutions in pairs. Here we are guided by a few topological insights, after considering the movie descriptions of the disks and the maps induced by Reidemeister moves and band moves. First, if  we want a class to map nontrivially under the cobordism map, it must satisfy the following:
\begin{itemize}

\item any crossing that will be eliminated by a Reidemeister I move is given its oriented resolution, and 

\item any pair of crossings that will be eliminated by a Reidemeister II move has both crossings given oriented resolutions or both crossings given disoriented resolutions.

\end{itemize}

In the movie for the disk $-D$, the three crossings on the righthand side of $-9_{46}$ are eliminated by Reidemeister I moves, so we preserve their (oriented) 0-resolutions. The remaining six crossings are canceled in pairs by Reidemeister II moves. If we flip 0- and 1-resolutions at all of these crossings, we obtain the smoothing that underlies the class $\phi$ from Figure~\ref{fig:946-movie-other}. (It is easy to check that if we flip the 0- and 1-resolutions on only one or two of these pairs of crossings, then the resulting smoothing will have 0-resolution arcs joining a circle of the resolution to itself, hence no individual generator supported by that smoothing will be a cycle.) In order to produce a chain element in the desired quantum grading $q=-1$, we have no choice but to assign an $\bfx$-label to each smoothing circle, yielding the class $\phi$.



\subsection{Additional applications and exercises}

\begin{exercise}
Consider the cobordisms $\Sigma$ and $\Sigma'$ between the 3-component unlink $L$ and the empty set depicted in Figure~\ref{fig:unlink}. Show that these induce distinct maps on Khovanov homology (in both directions $\emptyset \to L$ and $L \to \emptyset$).

\begin{figure}[b]\center
 \includegraphics[width=.75\linewidth]{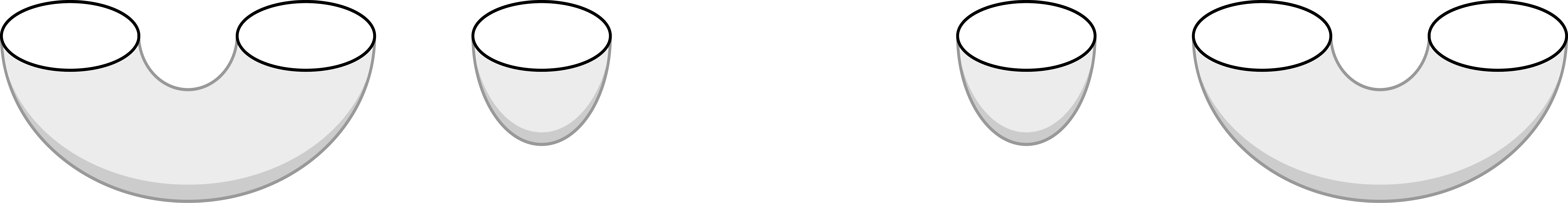}    
\caption{Surfaces $\Sigma$ (left) and $\Sigma'$ (right) bounded by the 3-component unlink.}\label{fig:unlink}
\end{figure}
\end{exercise}


\begin{exercise}
\textbf{(Plamenevskaya's invariant)} Let $L$ be an oriented link  expressed as the closure of an $n$-stranded braid $\beta \in B_n$. Plamenevskaya \cite{plamenevskaya:transverse-Kh}  defined an element $\psi(L) \in \kh(L)$ by considering the unique smoothing of $L$ that is itself braided, and assigning $\bfx$ to each of the $n$ components of the smoothing. (This has 0-resolutions at positive generators $\sigma_i$ and 1-resolutions at negative generators $\sigma_i^{-1}$, where $\sigma_i$ is a standard generator from Artin's presentation of the braid group $B_n$.) An example is shown in Figure~\ref{fig:plam} for the knot $10_{148}$, given as the closure of the braid 
$\beta = \sigma_2^{-2} \sigma_1 \sigma_2^{3}\sigma_1 \sigma_2^{-1} \sigma_1^2.$

 Plamenevskaya shows that $\pm \psi$ is functorial under braid isotopy and positive Markov (de)stabilization, i.e., $\beta  \in B_n \leftrightsquigarrow \beta \sigma_n \in B_{n+1}$. (This corresponds to the notion of  \emph{transverse isotopy} in contact topology; see \cite[\S2.4]{etnyre:knot-intro} for more background.)

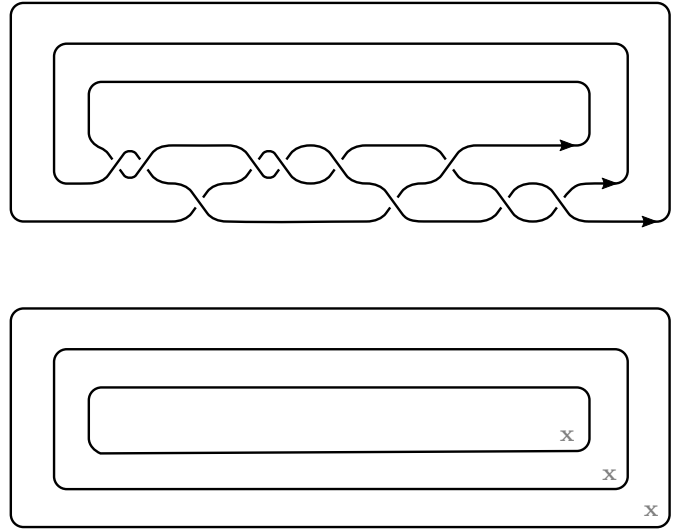
\begin{figure}[tb]
\center
\def\svgwidth{.66\linewidth}
\begingroup%
  \makeatletter%
  \providecommand\color[2][]{%
    \errmessage{(Inkscape) Color is used for the text in Inkscape, but the package 'color.sty' is not loaded}%
    \renewcommand\color[2][]{}%
  }%
  \providecommand\transparent[1]{%
    \errmessage{(Inkscape) Transparency is used (non-zero) for the text in Inkscape, but the package 'transparent.sty' is not loaded}%
    \renewcommand\transparent[1]{}%
  }%
  \providecommand\rotatebox[2]{#2}%
  \newcommand*\fsize{\dimexpr\f@size pt\relax}%
  \newcommand*\lineheight[1]{\fontsize{\fsize}{#1\fsize}\selectfont}%
  \ifx\svgwidth\undefined%
    \setlength{\unitlength}{847.73587445bp}%
    \ifx\svgscale\undefined%
      \relax%
    \else%
      \setlength{\unitlength}{\unitlength * \real{\svgscale}}%
    \fi%
  \else%
    \setlength{\unitlength}{\svgwidth}%
  \fi%
  \global\let\svgwidth\undefined%
  \global\let\svgscale\undefined%
  \makeatother%
  \begin{picture}(1,0.73329572)%
    \lineheight{1}%
    \setlength\tabcolsep{0pt}%
    \put(0,0){\includegraphics[width=\unitlength,page=1]{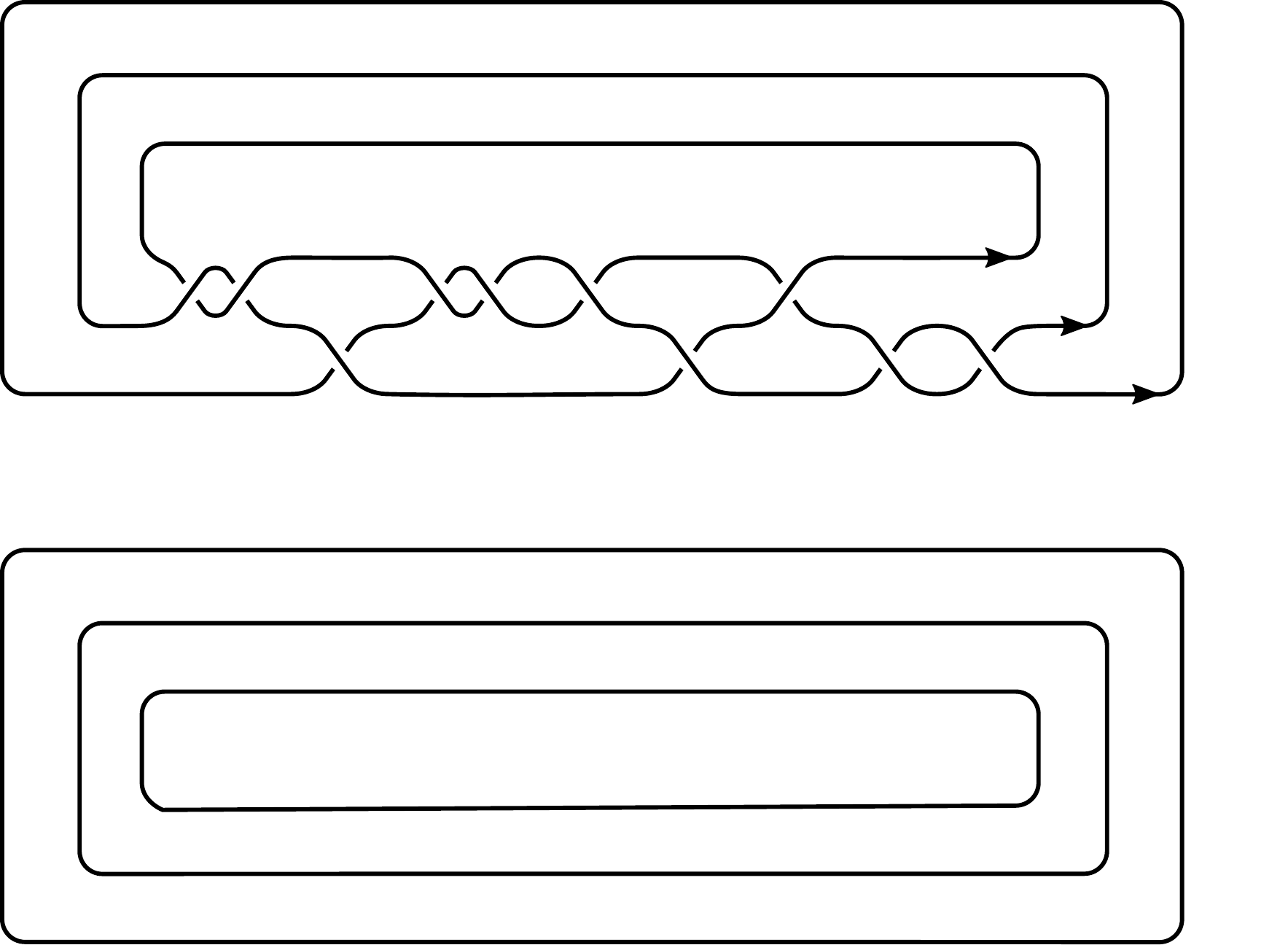}}%
    \put(0.76534332,0.12124091){\color[rgb]{0.50196078,0.50196078,0.50196078}\makebox(0,0)[lt]{\lineheight{1.25}\smash{\begin{tabular}[t]{l}{\footnotesize$\bfx$}\end{tabular}}}}%
    \put(0.82434486,0.06668982){\color[rgb]{0.50196078,0.50196078,0.50196078}\makebox(0,0)[lt]{\lineheight{1.25}\smash{\begin{tabular}[t]{l}{\footnotesize$\bfx$}\end{tabular}}}}%
    \put(0.88241688,0.01664307){\color[rgb]{0.50196078,0.50196078,0.50196078}\makebox(0,0)[lt]{\lineheight{1.25}\smash{\begin{tabular}[t]{l}{\footnotesize$\bfx$}\end{tabular}}}}%
  \end{picture}%
\endgroup%

\caption{Representing the knot $10_{148}$ as the closure of a quasipositive braid, along with the corresponding chain representative of Plamenevskaya's invariant.}\label{fig:plam}

\end{figure}

\vspace{-0.025in}

\begin{enumerate}[leftmargin=25pt, label=\bfseries(\alph*),itemsep=-3.5pt]
\item Show that $\psi$ always lies in bigrading $(h,q)=(0,w-n)$, where $w$ is the writhe of the braid and $n$ is the braid index.

\item Show that the chain element underlying $\psi$ is indeed a cycle.

\item Show that if $\beta$ is a positive braid, then $\psi(L) \neq 0$. 

\vspace{-0.05in}

\hfill \rotatebox{180}{\emph{Hint: Where does it lie in the cube of resolutions?}}

\medskip

\item A braid $\beta \in B_n$ is \emph{quasipositive} if it is a product of conjugates of the (positive) standard Artin generators, i.e., $\beta = \prod_j w_j \sigma_{i_j} w_j^{-1}$ for some words $w_j \in B_n$. For example, the braid above is quasipositive:
$$\beta = \sigma_2^{-2} \sigma_1 \sigma_2^{3}\sigma_1 \sigma_2^{-1} \sigma_1^2= (\sigma_2^{-2} \sigma_1 \sigma_2^{2})( \sigma_2 \sigma_1 \sigma_2^{-1}) \sigma_1^2.$$
Show that, if $\beta$ is quasipositive, then $L$ bounds a surface $\Sigma: L \to \emptyset$ such that  $\kh(\Sigma)(\psi(L)) \neq 0$. (Hence $\psi(L) \neq 0$, in particular.) 

\vspace{-.045in}

\hfill \rotatebox{180}{\emph{Hint: Start with saddles at the conjugated crossings $\sigma_{i}$ in $w \sigma_i w^{-1}$.}}

\end{enumerate}

\end{exercise}


\smallskip

\begin{exercise}
 Let $L$ be the $(3,-3)$-torus link but with one strand oriented oppositely from the other two.
 \vspace{-0.04in}
 
\begin{enumerate}[label=\bfseries(\alph*),leftmargin=23pt,itemsep=-3pt]
\item Show $L$ bounds  two different slice surfaces,  each comprised of a disk and annulus.
\item Show that these two slice surfaces induce distinct maps $\kh(L) \to \kh(\emptyset)$.

\end{enumerate}
\end{exercise}


\smallskip


\begin{exercise}\label{exer:livingston} \textbf{(Knotted Seifert surfaces)} In \cite{livingston}, Livingston asked whether any two Seifert surfaces of equal genus for a fixed knot are necessarily isotopic (rel $\partial$) through surfaces in $B^4$. This was resolved in \cite{hkmps:seifert}; let's walk through a counterexample.

Consider the right-handed trefoil $K$ as the closure of the positive braid $\sigma_1^3 \in B_2$. From this description, $K$ naturally bounds a genus-1 Seifert surface $\Sigma$ (built from 2 disks and 3 bands). 
 Figure~\ref{fig:cable-tref-braid} depicts the  (untwisted) 2-copy of $K$ (i.e., the union of $K$ and its Seifert pushoff) and the (untwisted) positive Whitehead double $\wh(K)$.

\begin{figure}[b]
\center
 \includegraphics[width=.55\linewidth]{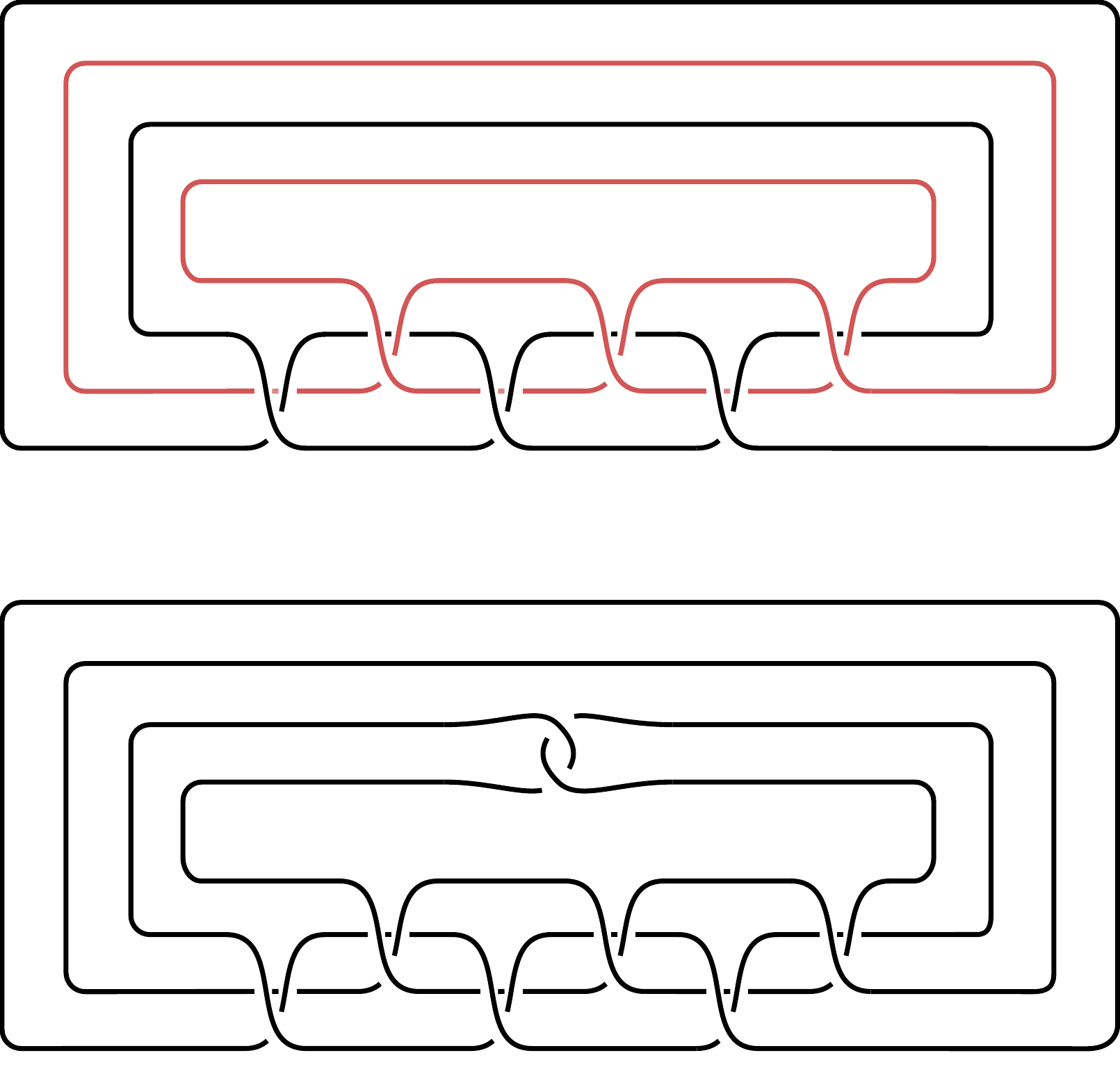}
\caption{The 2-copy of the right-handed trefoil (top) and the positive Whitehead double of the right-handed trefoil (bottom).} \label{fig:cable-tref-braid}
\end{figure}

\begin{enumerate}[label=\bfseries(\alph*),leftmargin=25pt,itemsep=-3pt]\item Convince yourself that the 2-copy of $K$ bounds a disconnected Seifert surface formed from two copies of the original Seifert surface $\Sigma$ for $K$. 

\item Observe that, if the 2-copy of $K$ is oriented as a braid, then it is a quasipositive braid, and it follows easily from the previous problem that the disconnected Seifert surface induces a nontrivial map on $\kh$ over $\zz_2$. Then show that this is true even if we reverse the orientation of one component of the 2-copy of $K$ (hence on one component of the disconnected Seifert surface).

\item Extend the above argument to show that $\wh(K)$ bounds a genus-2 Seifert surface --- one might call it $\wh(\Sigma)$ --- that induces a nontrivial map on $\kh$ over $\zz_2$. 

\item Find a genus-2 Seifert surface for $\wh(K)$ inducing the zero map on $\kh$ over $\zz_2$.

\vspace{-.1in}

\hfill \rotatebox{180}{\emph{Hint: Every Whitehead double has Seifert genus 1. Mind orientations!}}

\end{enumerate}
\end{exercise}


\begin{remark}
It follows that these two Seifert surfaces in $S^3$ are not smoothly isotopic rel $\partial$ in $B^4$. However, they \emph{are} topologically isotopic rel $\partial$. Indeed, Conway--Powell \cite{conway-powell} proved that any two Seifert surfaces of equal genus for a knot with trivial Alexander polynomial $\Delta(t)\equiv 1$ are topologically isotopic rel $\partial$.
\end{remark}

\bigskip

\section{Bar-Natan's theory and the TQFT approach}\label{sec:tqft}

\smallskip

\smallskip

We next consider Bar-Natan's deformation of Khovanov homology and use this as a segue to consider Bar-Natan's more topological perspective \cite{barnatan} (which beautifully underscores the way in which 2-dimensional cobordisms lie at the heart of Khovanov homology). This will enable us to  reboot from the TQFT perspective on Khovanov homology, including offering some insight into the proofs of invariance. 


\subsection{The Bar-Natan deformation}\label{subsec:bar-natan}


Khovanov's construction assigns a link homology theory to any appropriate choice of Frobenius algebra; this will be discussed further in \S\ref{subsec:tqft}. One such choice leads to the \emph{Bar-Natan deformation} of Khovanov homology over $\rcal=\ff_2[H]$, or simply \emph{Bar-Natan homology}. The only difference here is that we work over $\acal=\rcal[\bfx]/(\bfx^2 =H\bfx)$, with merge and split maps satisfying
\begin{align*}
m(\bfo \otimes \bfo) &= \bfo \hspace{100pt} \Delta(\bfo)=\bfo \otimes \bfx + \bfx \otimes \bfo \textcolor{red}{ \,+ \, \underline{H \cdot \bfo \otimes \bfo}}
\\
m(\bfo \otimes \bfx)&= m(\bfx \otimes \bfo)= \bfx  \hspace{45.75pt} \Delta(\bfx) = \bfx \otimes \bfx
\\
m(\bfx \otimes \bfx)&=\textcolor{red}{\underline{H\bfx}} 
\end{align*}
where the new terms are underlined and shown in red. 

Bar-Natan proves that this yields a functorial link invariant $\bn(\, \cdot \,)$ and that its cobordism maps are invariant up to isotopy rel boundary \cite{barnatan}. (Signs are irrelevant here because we chose the coefficient ring $\rcal = \ff_2[H]$; however, Bar-Natan also considers several related theories over rings where $2 \neq 0$, and he proves that their cobordism maps are well-defined up to sign.) The maps themselves are defined just as before, so the rules in Tables~\ref{table:R1}-\ref{table:R2} still apply.

\begin{wrapfigure}[11]{r}{.24\linewidth}
\vspace{-0.195in}
\center

\def\svgwidth{.825\linewidth}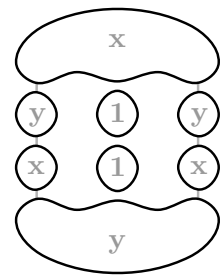
\captionsetup{width=.88\linewidth} \caption{A cycle $\theta$ in $\bn(-9_{46})$.}\label{fig:theta}
\end{wrapfigure}

Due to the additional terms in the differential, it is typically more difficult to explicitly identify cycles at the chain level in Bar-Natan homology. For example, while the labeled smoothings $\phi$ and $\psi$ from Figures~\ref{fig:946-movie-other} and \ref{fig:946-explore} are cycles when viewed in $\ckh(-9_{46})$, neither are cycles when viewed in $\cbn(-9_{46})$.  Although there is no simple way to circumvent this, a useful shift in perspective is to introduce the term $$\bfy=\bfx + H\! \cdot  \! \bfo.$$

\begin{example}
By replacing half of the $\bfx$-labels from $\psi$ with $\bfy$-labels, we obtain the element $\theta$ shown in Figure~\ref{fig:theta}.
\end{example}

\begin{remark}
The above perspective is closely related to the construction of Lee-type generators in Bar-Natan homology.
\end{remark}

\begin{exercise}\label{exercise:theta}
Show that the class $\theta \in \bn(-9_{46})$ is mapped to $1 \in \ff_2[H] =\bn(\emptyset)$ under the maps induced by both $-D$ and $-D'$.  
\end{exercise}

\begin{exercise}
Consider the two oriented copies of the unknot $U$ shown in Figure~\ref{fig:unknots}. 
\begin{enumerate}[label=\bfseries(\alph*),leftmargin=25pt,itemsep=-3pt]
\item Find a sequence of Reidemeister moves between them, i.e., an isotopy $U \rightsquigarrow U$ that reverses the orientation on $U$. 
\item Show that the induced map on Bar-Natan homology $\bn(U)$ over $\ff_2[H]$ swaps $\bfx$ with $\bfy$ and fixes $\bfo$.
\begin{figure}[h]
\center
\includegraphics[width=.275\linewidth]{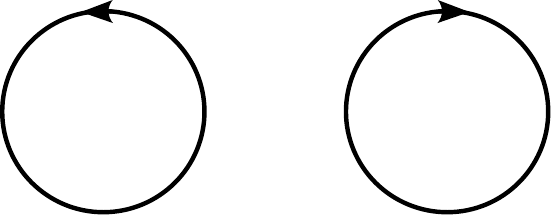}
\caption{The unknot, equipped with two different orientations.}\label{fig:unknots}
\end{figure}
\end{enumerate}
\end{exercise}

The next exercise considers the effect of local internal stabilization on cobordism maps in Bar-Natan homology.

\begin{exercise}
In contrast with Exercise~\ref{exer:stab}, show that $\bn(\Sigma \# T^2) = H \cdot \bn(\Sigma)$.
\end{exercise}

In fact, the same holds more broadly for arbitrary internal stabilizations (as shown in,  for example \cite[Corollary~2.3]{alishahi:unknotting} and \cite[Proposition~6.11]{lipshitz-sarkar:mixed}):

\begin{proposition}\label{prop:stab}
Given a link cobordism $\Sigma$ in $S^3  \times  I$, let $\Sigma'$ be obtained by internally stabilizing $\Sigma$ using an arc in $S^3  \times I$ whose endpoints lie on a single component of $\Sigma$ and whose interior is disjoint from $\Sigma$. The map induced by $\Sigma'$ satisfies $$\bn(\Sigma')=H \cdot \bn(\Sigma).$$
\end{proposition}


An important feature of Bar-Natan homology is that, for any knot $K$, we have
$$\bn(K) \cong \ff_2[H]^{0,s-1} \oplus \ff_2[H]^{0,s+1} \oplus (\text{$H$--torsion})$$
where superscripts denote the bigrading $(h,q)$ of the generator $1 \in \ff_2[H]$ and $s=s(K)$ is the analog of Rasmussen's invariant in Bar-Natan homology (c.f., \cite{kwz:immersed}). More generally, for a link $L$ with $|L|$ components, there are $2^{|L|}$ towers, i.e., summands of the form $\ff_2[H]$.

 Moreover, given  knots $K$ and $K'$ in $S^3$, any connected cobordism $\Sigma: K \to K'$ induces a nonzero map 
$$\bn(K) / (\text{$H$--torsion}) \to \bn(K')/(\text{$H$--torsion}).$$
This follows from an argument similar to the one given by Rasmussen for Lee's deformation of Khovanov homology \cite[Corollary~4.2]{rasmussen:s} (c.f., \cite[\S3]{lipshitz-sarkar:mixed}).

\smallskip
\smallskip

\subsubsection*{Distinguishing surfaces via Bar-Natan homology} \ 

Considering the case of link cobordisms $\Sigma: K \to \emptyset$, the above properties imply severe constraints on the map $\bn(\Sigma):\bn(K) \to \ff_2[H]$. To be concrete, we consider the case where $\Sigma$ is a slice disk $D$, in which case $s(K)=0$ because $s$ is a concordance invariant. First, since $\ff_2[H]$ is torsion-free, the map is determined by its behavior on any elements in $\bn(K)$ that generate the towers $ \ff_2[H]^{0,\pm1}$ in $\bn(K)/(\text{$H$--torsion})$. Second, since $\chi(D)=1$ and the ring $\ff_2[H]=\bn(\emptyset)$ is generated in bigrading $(0,0)$, the map $\bn(D)$ kills any generator of the ``upper'' tower $ \ff_2[H]^{0,1}$ for grading-shift reasons. Therefore, for any element $\theta \in \bn(K)$ that generates the ``lower'' tower $ \ff_2[H]^{0,-1}$ in $\bn(K)/(\text{$H$--torsion})$, the map $\bn(D)$ must send $\theta$ to the generator $1 \in \ff_2[H]$. 

As a consequence,  we cannot directly distinguish pairs of slice disks via  their  maps on Bar-Natan homology over $\ff_2[H]$ when viewing them as link cobordisms $D,D': K \to \emptyset$. In theory, this issue might be avoided by working over a truncated coefficient ring, such as $\ff_2[H]/H^2$. Alternatively, one may instead return to the original perspective of considering link cobordisms $\emptyset \to K$. In that case, an argument similar to the one above shows that the \emph{difference element} $$\delta:=\bn(D)(1)-\bn(D') (1) \ \ \in \ \bn(K)$$
must be an $H$--torsion element of $\bn(K)$. However, this brings with it the previous difficulties of computational complexity.

In practice, applications of the cobordism maps on Bar-Natan homology in \cite{hayden:atomic,guth-hayden-kang-park} have instead applied  a hybrid approach that combines by-hand calculations in Khovanov homology with computer calculations of Bar-Natan homology. We illustrate this by sketching a proof of the following, based on an example from \cite[\S5.2]{hayden:atomic}:

\begin{figure}[t]

\center
\def\svgwidth{\linewidth}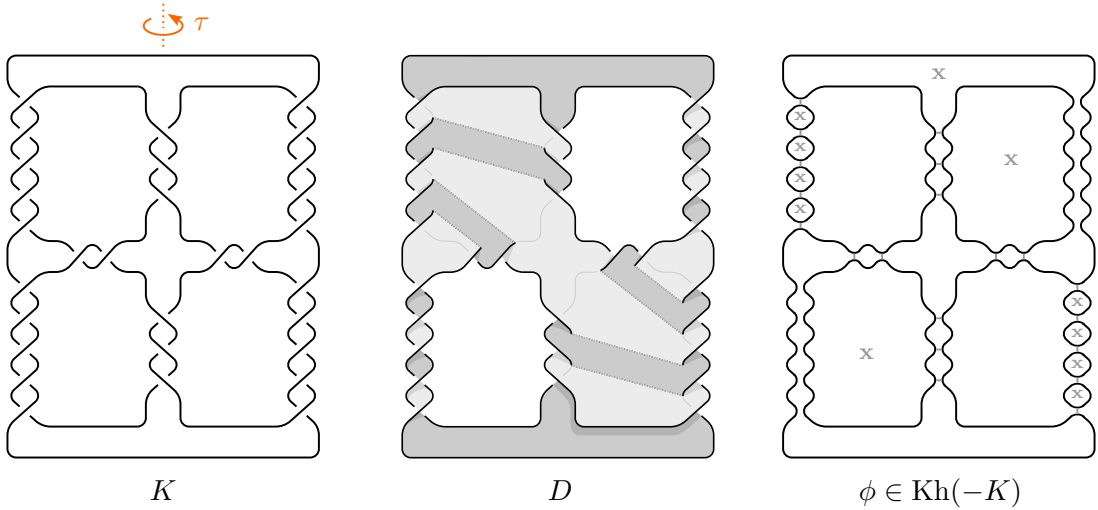

\caption{A strongly invertible knot $(K,\tau)$ bounding a slice disk $D$, along with class $\phi \in \kh(-K)$ that distinguishes $-D$ and $-\tau(D)$ up to isotopy rel boundary.}

\medskip

\label{fig:from946}
\end{figure}

\begin{proposition}
Let $(K,\tau)$ and $D$ denote the strongly invertible knot and its slice disk shown in Figure~\ref{fig:from946}, and let $D'=\tau(D)$. The slice disks $D,D': \emptyset \to K$ induce distinct maps $\bn(D) \neq \bn(D')$.  Moreover, $H \cdot \bn(D) \neq H \cdot \bn(D')$, so  $D$ and $D'$ are not isotopic rel boundary even after one internal stabilization.
\end{proposition}

\begin{proof}[Proof sketch] For convenience, we will work with Khovanov homology over $\ff_2$ and Bar-Natan homology over $\ff_2[H]$.

We begin by considering the dual disks $-D,-D': -K \to \emptyset$. The maps $\kh(-D)$ and $\kh(-D')$ are distinguished by considering their behavior on the element $\phi \in \kh(-K)$ depicted in Figure~\ref{fig:from946}, which is mapped to $1$ by $\kh(-D)$ and to 0 by $\kh(-D')$. (The reader is encouraged to check this as an exercise, making sure to mirror the diagram of $K$.)

Dualizing, it easily follows that $\kh(D) \neq \kh(D')$. In particular, we conclude that the difference element
 $$\delta_{\kh}:= \kh(D)(1) - \kh(D')(1) \ \ \in \ \kh^{0,1}(K)$$
is nonzero.

Next, we lift this to Bar-Natan homology. To do so, we note that the Khovanov chain complex $\ckh$ can be obtained from the Bar-Natan chain complex $\cbn$ by setting $H=0$, i.e., by taking the quotient of $\cbn$ by $H \cdot \cbn$. This chain-level quotient map $\pi$ induces a map $\pi_*$ on homology.  Moreover, this construction is compatible with the cobordism maps, leading to the following commutative diagram:
$$
\begin{tikzcd}
\ff_2[H]=\bn(\emptyset) \arrow{r}{\bn(D)-\bn(D')}   \arrow{d}{\pi_*} &[1.75cm] \bn(K) \arrow{d}{\pi_*}\\[.125cm]%
\ff_2=\kh(\emptyset)  \arrow{r}{\kh(D)-\kh(D')}  & \kh(K) 
\end{tikzcd}
$$
Now consider the Bar-Natan difference element
$$\delta_{\bn}:= \bn(D)(1) - \bn(D')(1) \ \ \in \ \bn^{0,1}(K).$$
Commutativity of the diagram shows that $\pi_*(\delta_{\bn})=\delta_{\kh}$, so we have $\delta_{\bn}\neq 0$. It follows that $\bn(D) \neq \bn(D')$.

Finally, we wish to argue that $H \cdot \delta_{\bn}\neq 0$. To see this, one can use the program KnotJob \cite{knotjob} to compute the Bar-Natan--Lee--Turner spectral sequence (over $\ff_2$) and then determine the $\ff_2[H]$-module structure of $\bn(K)$. This spectral sequence collapses on its third page, and the relevant portions of the first two pages are shown in Table~\ref{table:946_SS}.  In particular, all elements in bigrading $(0,1)$ survive to the second page, which implies that $H \cdot \delta_{\bn} \neq 0$, as desired.\end{proof}

\setlength\extrarowheight{2pt}

\newcommand{\tabwidth}{0.05\textwidth}

\begin{table}[tb]\small
\vspace{0.45in}

\medskip
\centering
\setlength\extrarowheight{2pt}
\begin{tabular}{|>{\centering}m{.06\textwidth}||cc|>{\centering}m{\tabwidth}|>{\centering}m{\tabwidth}|>{\centering}m{\tabwidth}|>{\centering}m{\tabwidth}|>{\centering}m{\tabwidth}|>{\centering}m{\tabwidth}|c|}
\multicolumn{10}{c}{}\\
\multicolumn{10}{c}{Page 1} \\
\cline{1-10}  
\color{black}\backslashbox{\!$q$\!}{\!$h$\!} &  \ $\hdots$   &\hspace{-1.8pt} {\color{black}{\vrule}} \hspace{1pt} \color{black}\raisebox{-3pt}{$-7$} \hspace{1.25pt}  & \color{black}\raisebox{-3pt}{$-6$} & \color{black}\raisebox{-3pt}{$-5$} & \color{black}\raisebox{-3pt}{$-4$} & \color{black}\raisebox{-3pt}{$-3$} & \color{black}\raisebox{-3pt}{$-2$} & \color{black}\raisebox{-3pt}{$-1$} & \color{black}\raisebox{-3pt}{$0$} \\
\hhline{=||=========}
$1$     &   &   &   &   &   &   &   &   & \cellcolor{cellgray} \hspace{4.25pt}2\hspace{5.75pt}  \\
\hhline{-||~--------}
$-1$     &   &   &   &   &   &   &   &   & 2   \\
\hhline{-||~--------}
$-3$     &   &   &   &   &   &   &  2 & 1 &   \\
\hhline{-||~--------}
$-5$     &   &   &   &   &   & 4  &  2 &  1 &   \\
\hhline{-||~--------}
$-7$     &   &   &   &   &  9 & 5  &   &   &   \\
\hhline{-||~--------}
$-9$     &   &   &   & 17 & 10  & 1  &   &   &   \\
\hhline{-||~--------}
$-11$     &   &   &  21 &  17  &   1 &   &   &   &   \\
\hhline{-||~--------}
$-13$     &   & \ \ \  27   &   22 &   &   &   &   &   &   \\
\hhline{-||}
$\vdots$ & \ \ \reflectbox{$\ddots$} \\
\multicolumn{10}{c}{}\\
\end{tabular}

\medskip

\vspace{0.25in}

\begin{tabular}{|>{\centering}m{.06\textwidth}||cc|>{\centering}m{\tabwidth}|>{\centering}m{\tabwidth}|>{\centering}m{\tabwidth}|>{\centering}m{\tabwidth}|>{\centering}m{\tabwidth}|>{\centering}m{\tabwidth}|c|}
\multicolumn{10}{c}{}\\
\multicolumn{10}{c}{Page 2} \\
\cline{1-10}  
\color{black}\backslashbox{\!$q$\!}{\!$h$\!} &  \ $\hdots$   &\hspace{-1.8pt} {\color{black}{\vrule}} \hspace{1pt} \color{black}\raisebox{-3pt}{$-7$} \hspace{1.25pt}  & \color{black}\raisebox{-3pt}{$-6$} & \color{black}\raisebox{-3pt}{$-5$} & \color{black}\raisebox{-3pt}{$-4$} & \color{black}\raisebox{-3pt}{$-3$} & \color{black}\raisebox{-3pt}{$-2$} & \color{black}\raisebox{-3pt}{$-1$} & \color{black}\raisebox{-3pt}{$0$} \\
\hhline{=||=========}
$1$     &   &   &   &   &   &   &   &   & \cellcolor{cellgray} \hspace{4.25pt}2\hspace{5.75pt}  \\
\hhline{-||~--------}
$-1$     &   &   &   &   &   &   &   &   & 2   \\
\hhline{-||~--------}
$-3$     &   &   &   &   &   &   &   & 1 &   \\
\hhline{-||~--------}
$-5$     &   &   &   &   &   &   &   &  1 &   \\
\hhline{-||~--------}
$-7$     &   &   &   &   &   &   &   &   &   \\
\hhline{-||~--------}
$-9$     &   &   &   & 1 &    &   &   &   &   \\
\hhline{-||~--------}
$-11$     &   &   &   &1    &   &   &   &   &   \\
\hhline{-||~--------}
$-13$     &   &   &1  &   &   &   &   &   &   \\
\hhline{-||}
$\vdots$ & \ \ \reflectbox{$\ddots$} \\
\multicolumn{10}{c}{}\\
\end{tabular}

\medskip
\bigskip

\caption{The first two pages of the Bar-Natan--Lee--Turner spectral sequence for the knot $K$ from Figure~\ref{fig:from946}, shown for $h \geq -7$ and $q \geq -13$.}
\label{table:946_SS}

\vspace{-10pt}

\end{table}


%
%

\clearpage

\subsection{The TQFT approach}\label{subsec:tqft}

\subsubsection{The basic setup}

Let $\Cobt$ be the category whose objects are closed 1-manifolds in the plane and whose morphisms are orientable 2-dimensional cobordisms between such 1-manifolds, considered up to boundary-preserving homeomorphism; see Figure~\ref{fig:hopf-cobs} for an example.\footnote{Although we will depict these cobordisms in $\rr^2 \times I$, we do not treat them as embedded submanifolds. See \cite[\S11.3]{barnatan} for a related discussion.} Composition is given by concatenation, and the identity morphisms are product cobordisms.  These morphisms are generated by unions of elementary cobordisms of the form depicted in Figure~\ref{fig:cobs}.

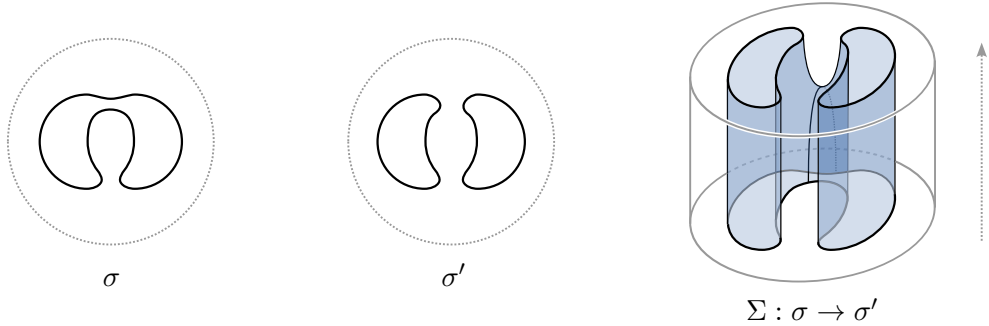
\begin{figure}[h]\center
\def\svgwidth{.9\linewidth}
\begingroup%
  \makeatletter%
  \providecommand\color[2][]{%
    \errmessage{(Inkscape) Color is used for the text in Inkscape, but the package 'color.sty' is not loaded}%
    \renewcommand\color[2][]{}%
  }%
  \providecommand\transparent[1]{%
    \errmessage{(Inkscape) Transparency is used (non-zero) for the text in Inkscape, but the package 'transparent.sty' is not loaded}%
    \renewcommand\transparent[1]{}%
  }%
  \providecommand\rotatebox[2]{#2}%
  \newcommand*\fsize{\dimexpr\f@size pt\relax}%
  \newcommand*\lineheight[1]{\fontsize{\fsize}{#1\fsize}\selectfont}%
  \ifx\svgwidth\undefined%
    \setlength{\unitlength}{869.2163951bp}%
    \ifx\svgscale\undefined%
      \relax%
    \else%
      \setlength{\unitlength}{\unitlength * \real{\svgscale}}%
    \fi%
  \else%
    \setlength{\unitlength}{\svgwidth}%
  \fi%
  \global\let\svgwidth\undefined%
  \global\let\svgscale\undefined%
  \makeatother%
  \begin{picture}(1,0.33049187)%
    \lineheight{1}%
    \setlength\tabcolsep{0pt}%
    \put(0.82176695,0.00441947){\color[rgb]{0,0,0}\makebox(0,0)[t]{\smash{\begin{tabular}[t]{c}$\Sigma: \sigma \to \sigma'$\end{tabular}}}}%
    \put(0,0){\includegraphics[width=\unitlength,page=1]{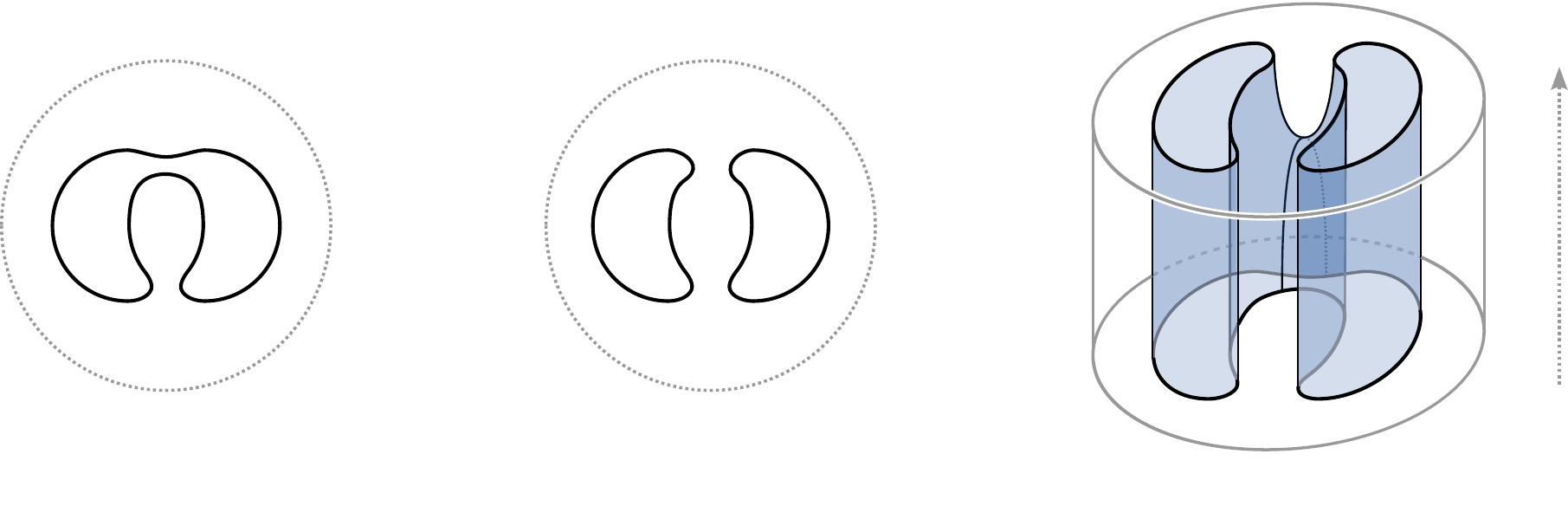}}%
    \put(0.10679404,0.03958851){\color[rgb]{0,0,0}\makebox(0,0)[t]{\smash{\begin{tabular}[t]{c}$\sigma$\end{tabular}}}}%
    \put(0.45597955,0.03958851){\color[rgb]{0,0,0}\makebox(0,0)[t]{\smash{\begin{tabular}[t]{c}$\sigma'$\end{tabular}}}}%
  \end{picture}%
\endgroup%

\caption{A pair of 1-manifolds $\sigma,\sigma'$ and a cobordism $\Sigma$ between them.}\label{fig:hopf-cobs}
\end{figure}

\begin{figure}[h]\center
\includegraphics[width=\linewidth]{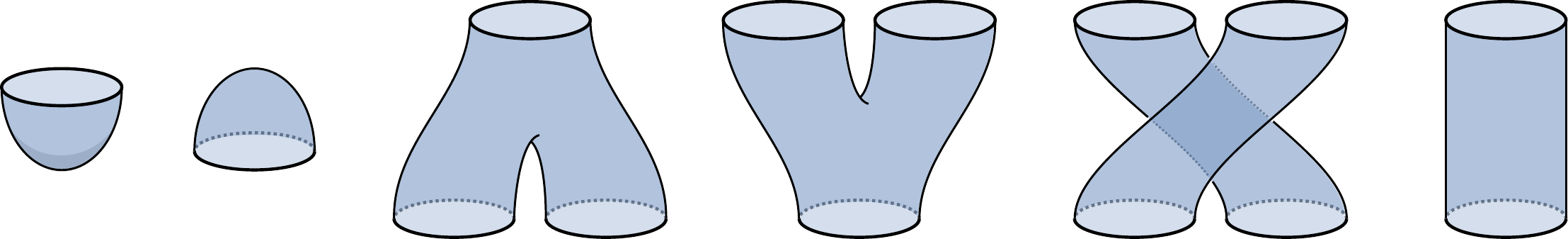}
\caption{Six elementary cobordisms that generate the morphisms in $\Cobt$.}\label{fig:cobs}
\end{figure}

\noindent 

\begin{definition}\label{def:tqft}
Define a functor $\fcal : \Cobt \to \Mod_\rcal$  as follows:

\begin{itemize}
\item Each circle is assigned a free $\rcal$-module $\acal=\rcal \langle \bfo,\bfx\rangle$.

\item Each closed 1-manifold $\sigma$ is assigned the tensor product $\acal_{\sigma}=\acal^{\otimes | \sigma |}$ of the copies of $\acal$ assigned to its component circles, where $|\sigma|$ denotes the number of components of $\sigma$. 

\item Given a cobordism $\Sigma: \sigma \to \sigma'$, we decompose it into concatenations of disjoint unions of the elementary cobordisms discussed above, then apply appropriate tensor products of the following maps.
\end{itemize}
\end{definition}
\begin{align*}
\raisebox{-.28cm}{\includegraphics[height=0.715cm]{small-birth.pdf}} \quad \qquad \qquad 
\iota: \rcal \to \acal \, ; \quad \ 
& \quad \iota(1) = \bfo \tag{birth}
\\
\\
\raisebox{-.28cm}{\includegraphics[height=0.715cm]{small-death.pdf}} \quad \qquad \qquad 
\epsilon: \acal \to \rcal \, ; \quad \ 
&\begin{cases} \epsilon(\bfo) = 0 
\\
\epsilon(\bfx) = 1
\end{cases}
\tag{death}
\\
\\
\raisebox{-.51cm}{\includegraphics[height=1.25cm]{small-merge.pdf}} \qquad 
m: \acal \otimes \acal \to \acal \,; \quad \ 
&\begin{cases} m(\bfo \otimes \bfo) = \bfo  
\\
m(\bfo \otimes \bfx) = \bfx
\\
m(\bfx \otimes \bfo)=\bfx
\\
m(\bfx \otimes \bfx)= 0
\end{cases}
\tag{merge}
\\
\\
\raisebox{-.51cm}{\includegraphics[height=1.25cm]{small-split.pdf}} \qquad 
\Delta: \acal \to \acal \otimes \acal \, ; \quad \ 
&\begin{cases} \Delta(\bfo) = \bfo \otimes \bfx + \bfx \otimes \bfo 
\\
\Delta(\bfx) = \bfx \otimes \bfx
\end{cases}
\tag{split}
\\
\\
\raisebox{-.51cm}{\includegraphics[height=1.05cm]{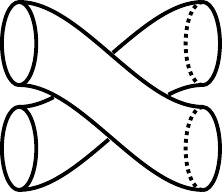}} \ \ \quad 
\rho : \acal \otimes \acal \to \acal \otimes \acal \, & ; \quad 
 \rho(\mathbf{a} \otimes \mathbf{a'})= \mathbf{a'} \otimes \mathbf{a} \tag{permutation}
 \\
\\
\raisebox{-.28cm}{\includegraphics[height=0.65cm]{small-id.pdf}}  \qquad \quad \ \, 
\id: \acal \to \acal \, ; \quad \ 
& \ \id(\mathbf{a})=\mathbf{a}
\tag{identity}
\end{align*}

\bigskip

\noindent To show that $\fcal$ is well-defined, one must verify that topologically equivalent cobordisms with different decompositions induce the same map. 


\begin{example} For the cobordisms in Figure~\ref{fig:birth-death}, we must show $(\id \otimes \, \epsilon) \circ \Delta=\id$:
\begin{align*}
(\id \otimes \, \epsilon) \circ \Delta (\bfo) &= (\id \otimes \, \epsilon)( \bfo \otimes \bfx + \bfx \otimes \bfo) 
\\
 &=(\id \otimes \, \epsilon)( \bfo \otimes \bfx) + (\id \otimes \, \epsilon)(\bfx \otimes \bfo)
\\
 &= \bfo \otimes \epsilon(\bfx) + \bfx \otimes \epsilon(\bfo)
 \\
 & = \bfo \otimes 1 + \bfx \otimes 0
 \\
 &= \bfo
\end{align*}
and
\begin{align*}
(\id \otimes \, \epsilon) \circ \Delta (\bfx) &= (\id \otimes \, \epsilon)( \bfx \otimes \bfx) 
\\
 &= \bfx \otimes \epsilon(\bfx)
 \\
 & = \bfx \otimes 1 
 \\
 &= \bfx
\end{align*}
\end{example}

\begin{exercise}
Show that $\fcal(\sphere)$ is the zero map and $\fcal(\torus)$ is $\times 2$.
\end{exercise}

\begin{exercise}
Show that the six elementary cobordisms in Figure~\ref{fig:cobs} generate $\Cobt$ (under disjoint unions and composition).
\end{exercise}

\smallskip

\smallskip \paragraph{\textbf{Frobenius algebra and gradings}} The operations defined above (namely the \emph{unit map} $\iota$, \emph{multiplication} $m$, \emph{comultiplication} $\Delta$, and \emph{counit map} $\epsilon$) endow $\acal$ with the structure of a Frobenius algebra. When viewing $\acal$ as an algebra, we may write multiplication $m(\mathbf{a} \otimes \mathbf{a'})$ as $\mathbf{a} \cdot \mathbf{a'}$. In particular, we have $\bfx^2=\bfx \cdot \bfx = 0$. As an intermediate perspective, one can view $\acal$ as the ring $\rcal[\bfx]/(\bfx^2)$, after setting $\bfo=1$.

\begin{remark}
These structures make $\acal$ into a \emph{Frobenius algebra}. In particular,  one needs $(\acal,m,\iota)$ to be an associative algebra over $\rcal$, $(\acal,\Delta,\epsilon)$ to be a coassociative coalgebra over $\rcal$, and the following diagrams commute:
\begin{equation}\label{eqn:frobenius}
\begin{tikzcd}
\acal \otimes \acal \arrow[r, "\Delta\otimes \id"] \arrow[d, "m"] & \acal \otimes \acal \otimes \acal \arrow[d, "\id \otimes m"] \\
\acal \arrow[r,"\Delta"]                          &  \acal \otimes \acal                                      
\end{tikzcd}
\qquad \qquad
\begin{tikzcd}
\acal \otimes \acal \arrow[r, "\id \otimes \Delta"] \arrow[d, "m"] & \acal \otimes \acal \otimes \acal \arrow[d, "m \otimes \id"] \\
\acal \arrow[r,"\Delta"]                          &  \acal \otimes \acal                                      
\end{tikzcd}
\end{equation}
\end{remark}

\begin{exercise}
Interpret the diagrams from \eqref{eqn:frobenius} in terms of $\Cobt$ and $\fcal$.
\end{exercise}

\begin{exercise}
Show that if $M$ is a closed, oriented manifold, then Poincar\'e duality endows $H^*(M;\zz)$ with a Frobenius algebra structure (over $\zz$). 
As an example, check that the Frobenius algebra $\acal$ defined above corresponds to $H^*(S^2;\zz)$.\footnote{We will assign gradings on $\acal$ by shifting the gradings on $H^*(S^2;\zz)$ down by $1$.}
\end{exercise}




We can equip $\acal$ with gradings given by setting $\deg(\bfo)=1$ and $\deg(\bfx)=-1$. 
Tensor products $\acal^{\otimes n}$ inherit gradings given by
$$\deg(\mathbf{a}_{\boldsymbol{1}} \otimes \cdots \otimes \mathbf{a}_{\boldsymbol{n}})=\deg(\mathbf{a}_{\boldsymbol{1}})+\cdots+\deg(\mathbf{a}_{\boldsymbol{n}}).$$

\begin{exercise}
\begin{enumerate} [label=\bfseries(\alph*)]
\item 
Show that the maps $\iota$, $m$, $\Delta$, and $\epsilon$ are graded maps of degree
$$\deg(\iota)=1, \qquad \deg(m)=-1, \qquad \deg(\Delta)=-1, \qquad \deg(\epsilon)=1.$$
\item Show that the map $\fcal(\Sigma)$ induced by a cobordism $\Sigma$ in $\Cobt$ is a graded map of degree $\chi(\Sigma)$.
\end{enumerate}
\end{exercise}



\medskip
\subsubsection{Bar-Natan's frame}

While the category $\Cobt$ is mostly sufficient for the informal goals of these notes, we digress here to sketch the richer category within which Bar-Natan works \cite{barnatan}. We strongly suggest reading Bar-Natan's paper itself, the beauty of which is mostly lost in the summary below.




\smallskip
\smallskip

\emph{Tangles}

The initial modification enables one to work with tangles instead of links, providing a rigorous way to work locally: For any finite collection of distinct points $B$ in $\partial D^2$, let $\Cobt(B)$ denote the category whose objects are closed 1-manifolds in $D^2$ with boundary $B$ and whose morphisms are 2-dimensional cobordisms in $D^2\times I$ between such 1-manifolds. (We view the previously-considered category as $\Cobt(\emptyset)$.) When clear from context, we will drop $B$ from the notation and simply write $\Cobt$.

\smallskip
\smallskip

\emph{Formal complexes}

From here, Bar-Natan's strategy is to emulate the construction of Khovanov homology while postponing the application of the functor $\fcal$, preserving the role of smoothings and cobordisms as long as possible. To do so, Bar-Natan constructs a category that has enough structure to support certain ``formal complexes'' that underlie Khovanov homology. If one squints, this looks like the Khovanov chain complex but without $\bfo$- and $\bfx$-labels on smoothings, and where the differentials are formal linear combinations of saddle cobordisms.

Technically speaking, the first step is to enlarge $\Cobt$ to allow formal $\zz$-linear combinations of the original morphisms between any fixed pair of objects, with composition extended in the natural bilinear way. The second step is to pass to the matrix category $\mathbf{Mat}(\Cobt)$ whose objects are formal direct sums of objects from $\Cobt$ and whose morphisms are matrices of morphisms between those formal direct sums. The third step is to consider the category $\mathbf{Kom}(\mathbf{Mat}(\Cobt))$ whose objects are ``formal chain complexes'' $\cdots \to C^k \to C^{k+1} \to \cdots$ of finite length, where the chain groups and differentials are objects and morphisms in $\mathbf{Mat}(\Cobt)$, respectively, such that the composition of consecutive morphisms is zero. Its morphisms are the corresponding analogs of chain maps. 

Given a tangle $T$, the cube-of-resolutions construction (with edges viewed as saddle cobordisms) yields a ``formal'' complex $\lb T \rb$ in $\mathbf{Kom}(\mathbf{Mat}(\Cobt))$. Here the chain groups comprising $\lb T \rb$ are naturally separated by the numbers of 0- and 1-resolutions; in particular, the chain group $\lb T \rb^h$ is generated by smoothings $\sigma$ of $T$ with homological grading $h=|\sigma|-n_-(T)$, where $|\sigma|$ is the number of 1-resolutions in $\sigma$ and $n_-$ denotes the number of negative crossings in $T$. 

\begin{exercise}
Show that the cobordisms associated to Reidemeister I and II induce chain maps $\lb T \rb \to \lb T' \rb$.
\end{exercise}

However, these cobordism maps do \emph{not} generally induce chain homotopy equivalences. This is rectified in the next step.

\begin{remark}
Let us recall the natural notion of chain homotopy in $\mathbf{Kom}$: A pair of morphisms between formal complexes $$f,g:(A^\bullet,\partial_A) \to (B^\bullet,\partial_B)$$ are \emph{homotopic} if there are ``backwards diagonal'' morphisms $h^k: A^k \to B^{k-1}$ such that $$f^k-g^k = h^{k+1} d_A^k + d^{k-1}_B h^k.$$
It is worth recalling here that Khovanov homology (and Bar-Natan's construction) are covariant functors, but the differential \emph{increases} the homological grading.
\end{remark}

\smallskip

\smallskip

\emph{The quotient $\Cobl$}

The complex $\lb T \rb$ is not itself an invariant of the tangle $T$. Instead, Bar-Natan applies the procedure above to a quotient $\smash{\Cobl}$ of $\Cobt$,  chosen so that the corresponding complex $\lb T \rb$ in $\smash{\mathbf{Kom}(\mathbf{Mat}(\Cobl))}$ is indeed an invariant up to chain homotopy equivalence.   For notational convenience, let $\Kobh$ denote $\smash{\mathbf{Kom}(\mathbf{Mat}(\Cobl))}$ modulo homotopy, i.e., where homotopic morphisms are identified.

 We now review the ``local relations'' used to define $\Cobl$:
\begin{itemize}[leftmargin=1.25cm]
\item [$(S)$] \emph{The sphere relation:} If a cobordism $\Sigma$ contains a closed sphere as a connected component, then $\Sigma$ is set equal to zero.

\begin{center} \def\svgwidth{.27\linewidth}
\begingroup%
  \makeatletter%
  \providecommand\color[2][]{%
    \errmessage{(Inkscape) Color is used for the text in Inkscape, but the package 'color.sty' is not loaded}%
    \renewcommand\color[2][]{}%
  }%
  \providecommand\transparent[1]{%
    \errmessage{(Inkscape) Transparency is used (non-zero) for the text in Inkscape, but the package 'transparent.sty' is not loaded}%
    \renewcommand\transparent[1]{}%
  }%
  \providecommand\rotatebox[2]{#2}%
  \newcommand*\fsize{\dimexpr\f@size pt\relax}%
  \newcommand*\lineheight[1]{\fontsize{\fsize}{#1\fsize}\selectfont}%
  \ifx\svgwidth\undefined%
    \setlength{\unitlength}{167.24689514bp}%
    \ifx\svgscale\undefined%
      \relax%
    \else%
      \setlength{\unitlength}{\unitlength * \real{\svgscale}}%
    \fi%
  \else%
    \setlength{\unitlength}{\svgwidth}%
  \fi%
  \global\let\svgwidth\undefined%
  \global\let\svgscale\undefined%
  \makeatother%
  \begin{picture}(1,0.37287511)%
    \lineheight{1}%
    \setlength\tabcolsep{0pt}%
    \put(0,0){\includegraphics[width=\unitlength,page=1]{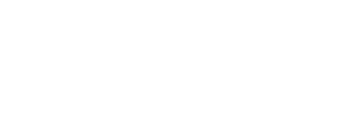}}%
    \put(0.63255133,0.15169833){\color[rgb]{0,0,0}\makebox(0,0)[t]{\smash{\begin{tabular}[t]{c}$=$\end{tabular}}}}%
    \put(0.91606472,0.15169833){\color[rgb]{0,0,0}\makebox(0,0)[t]{\smash{\begin{tabular}[t]{c}$0$\end{tabular}}}}%
    \put(0,0){\includegraphics[width=\unitlength,page=2]{sphere.pdf}}%
  \end{picture}%
\endgroup%

\end{center}

\smallskip

\item [$(T)$] \emph{The torus relation:}  If a cobordism $\Sigma$ contains a closed torus as a connected component, then set $\Sigma = 2 \Sigma'$ where $\Sigma'$ is $\Sigma$ without the torus component.

\begin{center} \def\svgwidth{.3\linewidth}
\begingroup%
  \makeatletter%
  \providecommand\color[2][]{%
    \errmessage{(Inkscape) Color is used for the text in Inkscape, but the package 'color.sty' is not loaded}%
    \renewcommand\color[2][]{}%
  }%
  \providecommand\transparent[1]{%
    \errmessage{(Inkscape) Transparency is used (non-zero) for the text in Inkscape, but the package 'transparent.sty' is not loaded}%
    \renewcommand\transparent[1]{}%
  }%
  \providecommand\rotatebox[2]{#2}%
  \newcommand*\fsize{\dimexpr\f@size pt\relax}%
  \newcommand*\lineheight[1]{\fontsize{\fsize}{#1\fsize}\selectfont}%
  \ifx\svgwidth\undefined%
    \setlength{\unitlength}{205.32783328bp}%
    \ifx\svgscale\undefined%
      \relax%
    \else%
      \setlength{\unitlength}{\unitlength * \real{\svgscale}}%
    \fi%
  \else%
    \setlength{\unitlength}{\svgwidth}%
  \fi%
  \global\let\svgwidth\undefined%
  \global\let\svgscale\undefined%
  \makeatother%
  \begin{picture}(1,0.39250386)%
    \lineheight{1}%
    \setlength\tabcolsep{0pt}%
    \put(0,0){\includegraphics[width=\unitlength,page=1]{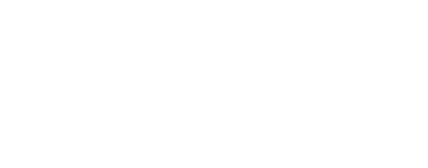}}%
    \put(0.70069987,0.16795555){\color[rgb]{0,0,0}\makebox(0,0)[t]{\smash{\begin{tabular}[t]{c}$=$\end{tabular}}}}%
    \put(0.93163172,0.16795555){\color[rgb]{0,0,0}\makebox(0,0)[t]{\smash{\begin{tabular}[t]{c}$2$\end{tabular}}}}%
    \put(0,0){\includegraphics[width=\unitlength,page=2]{torus.pdf}}%
  \end{picture}%
\endgroup%
 
\end{center}

\medskip

\item [$(\Tu)$] \emph{The 4-tube relation:} Suppose a cobordism $\Sigma$ contains a subsurface consisting of an annulus and two disks as shown on the left below. Then $\Sigma$ is set equal to the linear combination of cobordisms obtained by modifying the configuration of the tube as shown.

\bigskip

\hfill \def\svgwidth{.99\linewidth}
\begingroup%
  \makeatletter%
  \providecommand\color[2][]{%
    \errmessage{(Inkscape) Color is used for the text in Inkscape, but the package 'color.sty' is not loaded}%
    \renewcommand\color[2][]{}%
  }%
  \providecommand\transparent[1]{%
    \errmessage{(Inkscape) Transparency is used (non-zero) for the text in Inkscape, but the package 'transparent.sty' is not loaded}%
    \renewcommand\transparent[1]{}%
  }%
  \providecommand\rotatebox[2]{#2}%
  \newcommand*\fsize{\dimexpr\f@size pt\relax}%
  \newcommand*\lineheight[1]{\fontsize{\fsize}{#1\fsize}\selectfont}%
  \ifx\svgwidth\undefined%
    \setlength{\unitlength}{470.55144062bp}%
    \ifx\svgscale\undefined%
      \relax%
    \else%
      \setlength{\unitlength}{\unitlength * \real{\svgscale}}%
    \fi%
  \else%
    \setlength{\unitlength}{\svgwidth}%
  \fi%
  \global\let\svgwidth\undefined%
  \global\let\svgscale\undefined%
  \makeatother%
  \begin{picture}(1,0.1611432)%
    \lineheight{1}%
    \setlength\tabcolsep{0pt}%
    \put(0,0){\includegraphics[width=\unitlength,page=1]{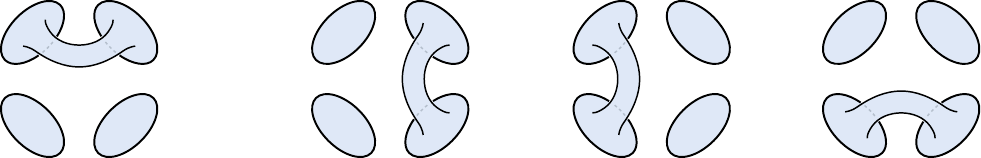}}%
    \put(0.23930639,0.07058563){\color[rgb]{0,0,0}\makebox(0,0)[t]{\smash{\begin{tabular}[t]{c}$=$\end{tabular}}}}%
    \put(0.53134014,0.07058563){\color[rgb]{0,0,0}\makebox(0,0)[t]{\smash{\begin{tabular}[t]{c}$+$\end{tabular}}}}%
    \put(0.79192662,0.07058563){\color[rgb]{0,0,0}\makebox(0,0)[t]{\smash{\begin{tabular}[t]{c}$-$\end{tabular}}}}%
  \end{picture}%
\endgroup%


\end{itemize}

\bigskip

\noindent From this position, Bar-Natan proves the following:

\begin{theorem}[{\cite[Theorem 1]{barnatan}}]
The isomorphism class of the complex $\lb T \rb$, viewed in $\mathbf{Kob_{/h}}$, is an invariant of the tangle $T$.
\end{theorem}

The strategy of the proof is to associate chain homotopy equivalences to each of the three Reidemeister moves --- indeed, this is where the definitions of the cobordism maps for Reidemeister I and II moves in Tables~\ref{table:R1}-\ref{table:R2}. We illustrate this by sketching the proof of invariance for one of the Reidemeister I moves. The idea of the proof is captured in Figure~\ref{fig:R1-chain}, whose horizontal rows depict the two local complexes $\lb T \rb$ and $\lb T' \rb$ that we claim are isomorphic (where $T$ has one crossing and $T'$ has none); here superscripts on $\lb \,\cdot \, \rb$ denote homological gradings. The vertical arrows describe morphisms $f$ and $g$, which can be understood as being induced by cobordisms (that induce the zero map on smoothings that disagree with the incoming boundary of the cobordism).

\begin{figure}
\center
\hfill \def\svgwidth{.925\linewidth}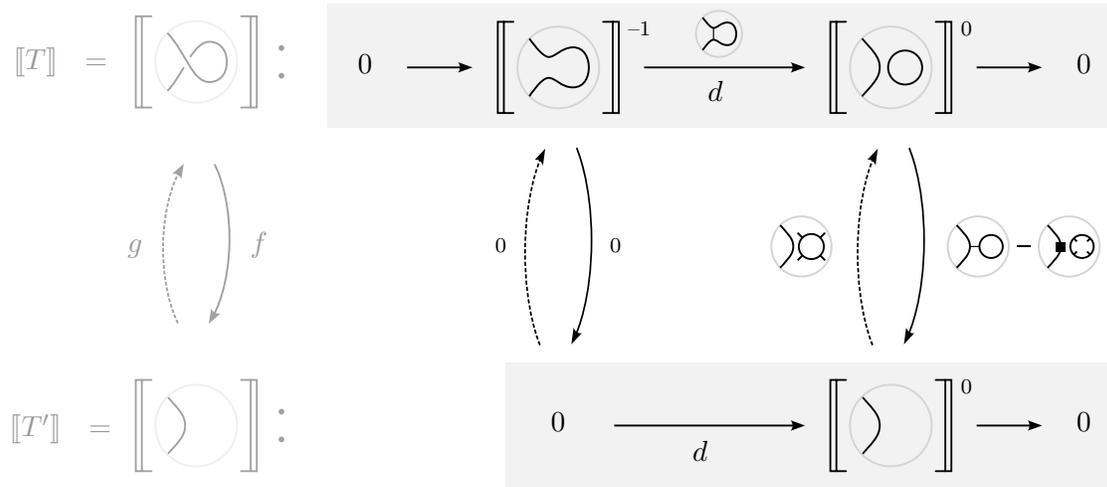
\caption{Invariance under one of the  Reidemeister I moves.}\label{fig:R1-chain}
\end{figure}

\begin{wrapfigure}[10]{r}{.44\linewidth}
\vspace{-0.21in}

\center

\includegraphics[width=.84\linewidth]{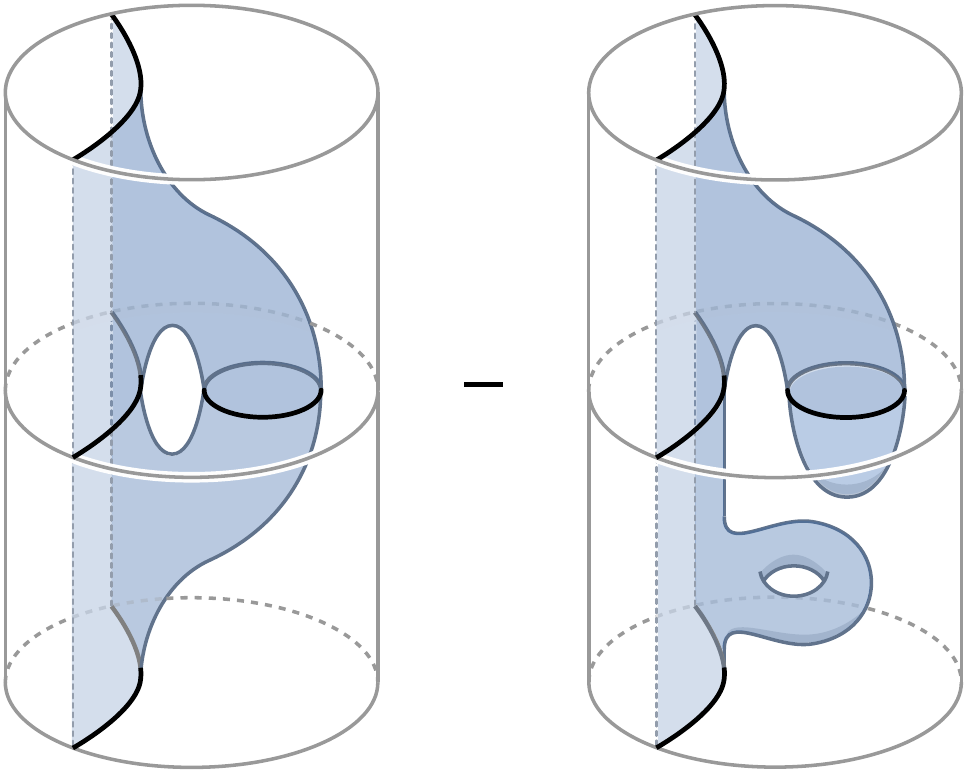}
\caption{}\label{fig:R1-chain-check1}
\end{wrapfigure}

First we  check that $f$ and $g$ are chain maps, which amounts to checking that the central square of Figure~\ref{fig:R1-chain} commutes (for each fixed choice of $f$ and $g$). The compositions $g \circ d$ and $d \circ g$ are both clearly zero. Similarly, the composition $f \circ d$ is clearly zero, so we must also show that  the composition $d \circ f$ is zero. This composition corresponds to the morphism shown in Figure~\ref{fig:R1-chain-check1}, which is clearly zero because it is the difference of two isotopic cobordisms.

Next we must show that $f$ and $g$ are homotopy inverses. To that end, we first consider the composition $f \circ g$. The only nontrivial component of this map corresponds to the morphism shown in Figure~\ref{fig:R1-invariance}(a); it is a difference of two cobordisms where the first is equivalent to a product (hence induces the identity) and the second contains a 2-sphere (hence vanishes by the sphere relation). This implies $f \circ g = \id$.

\begin{figure}
\center
\hfill \def\svgwidth{.95\linewidth}
\begingroup%
  \makeatletter%
  \providecommand\color[2][]{%
    \errmessage{(Inkscape) Color is used for the text in Inkscape, but the package 'color.sty' is not loaded}%
    \renewcommand\color[2][]{}%
  }%
  \providecommand\transparent[1]{%
    \errmessage{(Inkscape) Transparency is used (non-zero) for the text in Inkscape, but the package 'transparent.sty' is not loaded}%
    \renewcommand\transparent[1]{}%
  }%
  \providecommand\rotatebox[2]{#2}%
  \newcommand*\fsize{\dimexpr\f@size pt\relax}%
  \newcommand*\lineheight[1]{\fontsize{\fsize}{#1\fsize}\selectfont}%
  \ifx\svgwidth\undefined%
    \setlength{\unitlength}{1047.24525752bp}%
    \ifx\svgscale\undefined%
      \relax%
    \else%
      \setlength{\unitlength}{\unitlength * \real{\svgscale}}%
    \fi%
  \else%
    \setlength{\unitlength}{\svgwidth}%
  \fi%
  \global\let\svgwidth\undefined%
  \global\let\svgscale\undefined%
  \makeatother%
  \begin{picture}(1,0.3909609)%
    \lineheight{1}%
    \setlength\tabcolsep{0pt}%
    \put(0,0){\includegraphics[width=\unitlength,page=1]{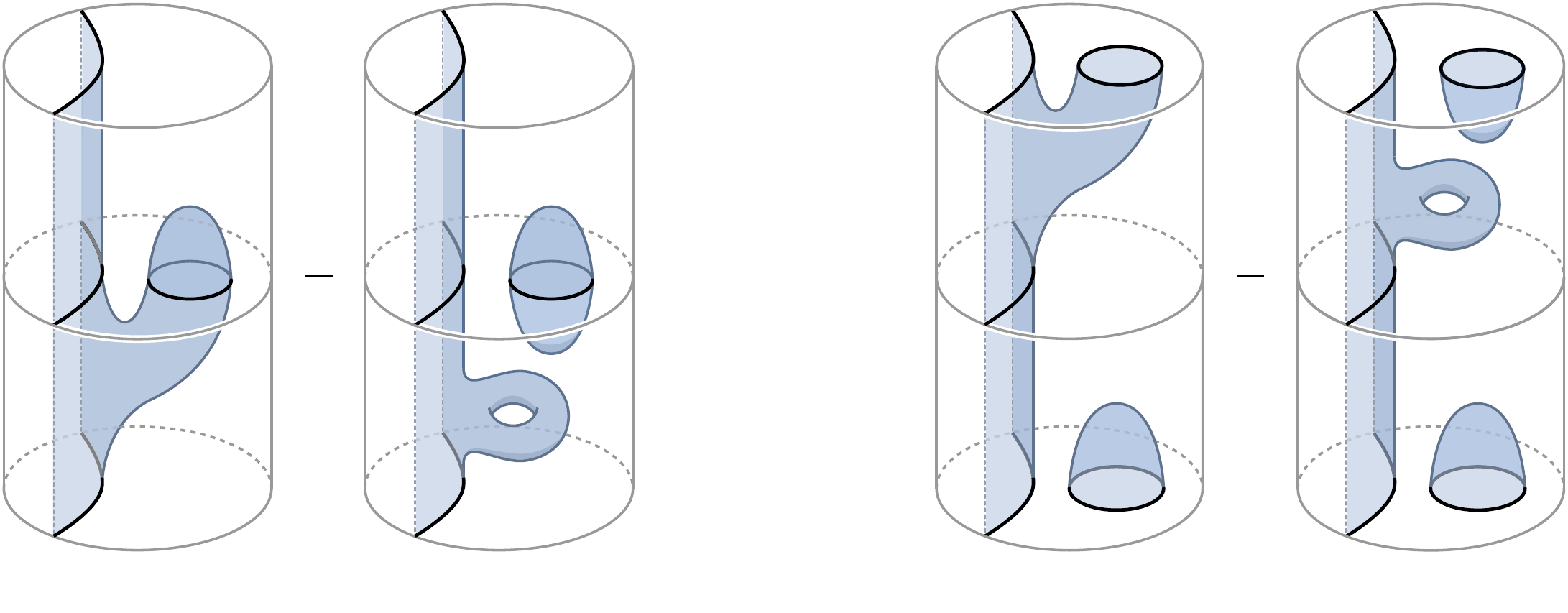}}%
    \put(0.20306524,0.00384192){\color[rgb]{0,0,0}\makebox(0,0)[t]{\smash{\begin{tabular}[t]{c}(a)\end{tabular}}}}%
    \put(0.79744567,0.00384192){\color[rgb]{0,0,0}\makebox(0,0)[t]{\smash{\begin{tabular}[t]{c}(b)\end{tabular}}}}%
  \end{picture}%
\endgroup%

\caption{}\label{fig:R1-invariance}
\end{figure}

Next we consider the composition $g \circ f$.   The only nontrivial component of the composition $g \circ f$ is represented by the difference of cobordisms $\lb T \rb^0 \to \lb T \rb^0$  in Figure~\ref{fig:R1-invariance}(b); in particular, the composition $g \circ f$ is zero on $\lb T \rb^{-1} \to \lb T \rb^{-1}$.  Since neither of these are obviously equivalent to the identity, we will need to use a nontrivial chain homotopy. To that end, we note that the chain homotopy map $h$ will have at most one possible nontrivial component,  as depicted on the left side of Figure~\ref{fig:chain-homotopy}. A natural candidate for the morphism $h$ is shown on the right side of Figure~\ref{fig:chain-homotopy}. It remains to verify the condition $g \circ f - \id = h d + d h$.

\begin{figure}[b]
\center
\hfill \def\svgwidth{\linewidth}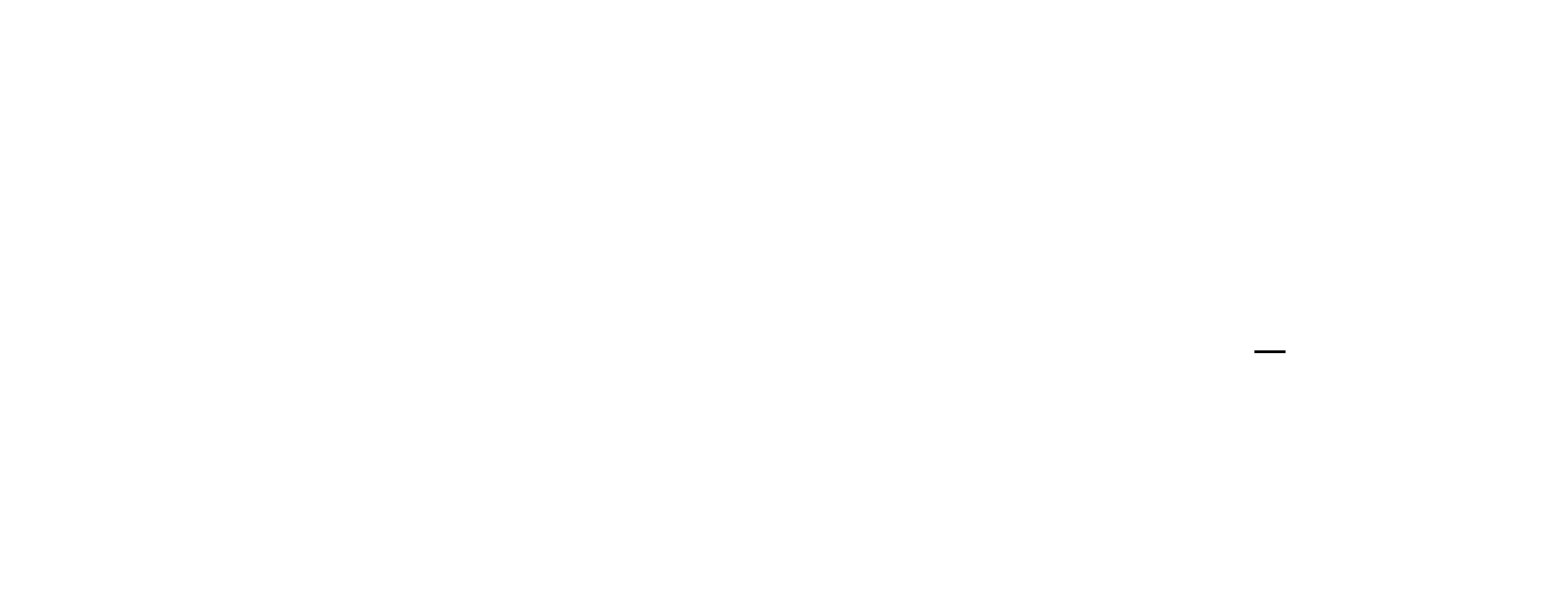
\caption{}\label{fig:chain-homotopy}
\end{figure}

 In homological grading $-1$, we have $g \circ f = 0$, so we must show that $-\id = h^0 d^{-1} + d^{-2} h^{-1}$; the morphism $d^{-2} h^{-1}$ is zero and $h^0 d^{-1}$ is easily seen to be $-\id$. 
 
 In homological grading $0$, we must establish the equality
$$(g \circ f)^{0} - \id = h^{1} d^{0}+d^{-1}h^{0}.$$
Since $h^{1} d^0$ is zero, this becomes equivalent to the relationship in Figure~\ref{fig:R1-invariance-hard}. After rearranging, this follows immediately from the 4-tube relation $(\Tu)$. 
Thus, we have shown that the local complexes $\lb T \rb$ and $\lb T' \rb$ are isomorphic.

\begin{figure}
\center
 \def\svgwidth{.925\linewidth}
\begingroup%
  \makeatletter%
  \providecommand\color[2][]{%
    \errmessage{(Inkscape) Color is used for the text in Inkscape, but the package 'color.sty' is not loaded}%
    \renewcommand\color[2][]{}%
  }%
  \providecommand\transparent[1]{%
    \errmessage{(Inkscape) Transparency is used (non-zero) for the text in Inkscape, but the package 'transparent.sty' is not loaded}%
    \renewcommand\transparent[1]{}%
  }%
  \providecommand\rotatebox[2]{#2}%
  \newcommand*\fsize{\dimexpr\f@size pt\relax}%
  \newcommand*\lineheight[1]{\fontsize{\fsize}{#1\fsize}\selectfont}%
  \ifx\svgwidth\undefined%
    \setlength{\unitlength}{1221.04697881bp}%
    \ifx\svgscale\undefined%
      \relax%
    \else%
      \setlength{\unitlength}{\unitlength * \real{\svgscale}}%
    \fi%
  \else%
    \setlength{\unitlength}{\svgwidth}%
  \fi%
  \global\let\svgwidth\undefined%
  \global\let\svgscale\undefined%
  \makeatother%
  \begin{picture}(1,0.38106077)%
    \lineheight{1}%
    \setlength\tabcolsep{0pt}%
    \put(0,0){\includegraphics[width=\unitlength,page=1]{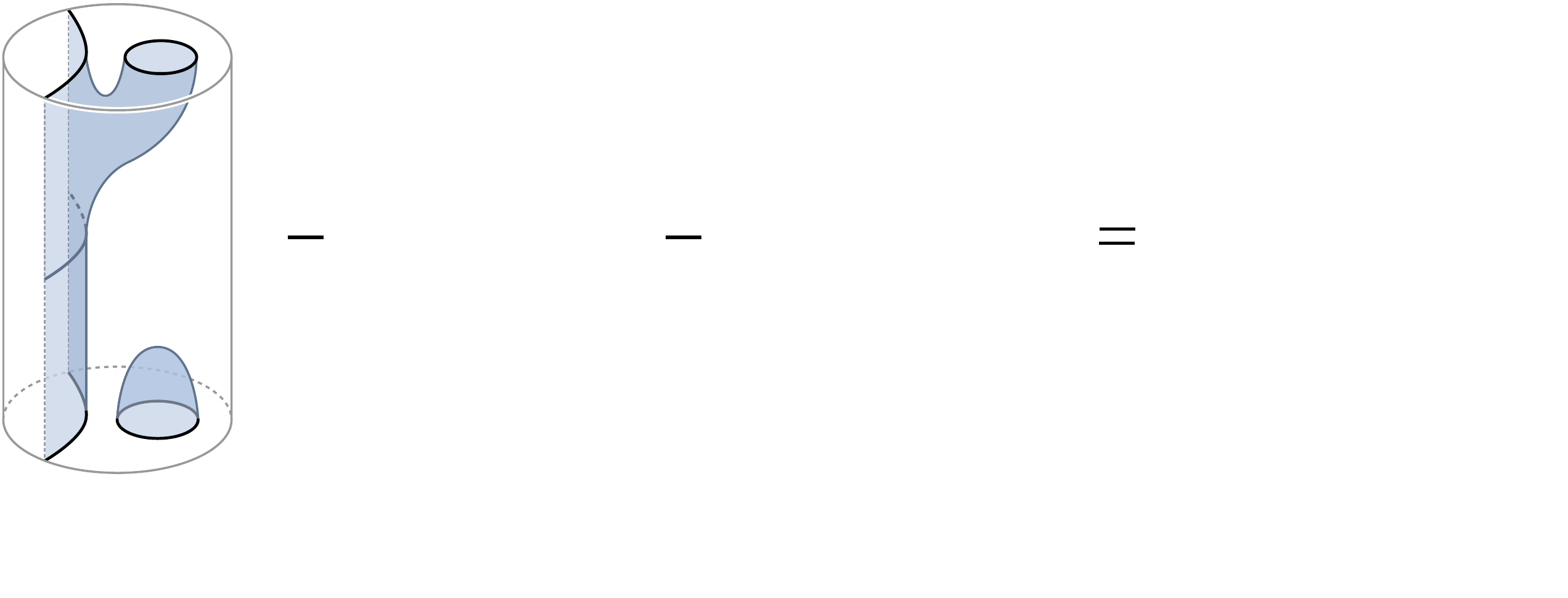}}%
    \put(0.19452239,0.00338306){\color[rgb]{0.6,0.6,0.6}\makebox(0,0)[t]{\smash{\begin{tabular}[t]{c}$(g \circ f)^0$\end{tabular}}}}%
    \put(0.89578825,0.00325908){\color[rgb]{0.6,0.6,0.6}\makebox(0,0)[t]{\smash{\begin{tabular}[t]{c}$d^{-1} h^0$\end{tabular}}}}%
    \put(0,0){\includegraphics[width=\unitlength,page=2]{invariance-R1-hard.pdf}}%
    \put(0.55640343,0.0031351){\color[rgb]{0.6,0.6,0.6}\makebox(0,0)[t]{\smash{\begin{tabular}[t]{c}$\id$\end{tabular}}}}%
  \end{picture}%
\endgroup%

\caption{}\label{fig:R1-invariance-hard}
\end{figure}

The local proofs for the remaining Reidemeister moves use some similar ingredients, as well the adaptation of more homological algebra to the setting of Bar-Natan's formal complexes (especially for Reidemeister III moves, which involve more complicated local complexes). Once this is done, Bar-Natan uses the framework of \emph{planar algebras} to show that the local arguments indeed establish invariance at the global level \cite[\S5]{barnatan}.


%

\smallskip
\smallskip 

\emph{...and Khovanov homology?}

The functor $\fcal: \Cobt \to \Mod_\rcal$ from Definition~\ref{def:tqft} naturally induces a functor from $\mathbf{Kob_{/h}}$ to  $\mathbf{Kom}_\rcal$, the category of chain complexes of $\rcal$-modules; we will also denote this functor by $\fcal$. For any link $L$, applying $\fcal$ to $\lb L \rb \in \Kobh$ yields a complex of $\rcal$-modules $\fcal(\lb L \rb )$ that coincides with  Khovanov's complex.

\medskip 
\subsubsection{Remarks on the proof of invariance of the cobordism maps}

Bar-Natan's framework is also a natural environment for establishing invariance of the maps induced by link cobordisms. Just as in \cite{jacobsson}, Bar-Natan uses the fact that any pair of movies of link diagrams describing link cobordisms that are isotopic rel boundary are related by a finite collection of \emph{movie moves} \cite{carter-saito}. An example of a movie movie is shown in Figure~\ref{fig:movie-move}.

\begin{figure}[b]
\center
\includegraphics[width=.8\linewidth]{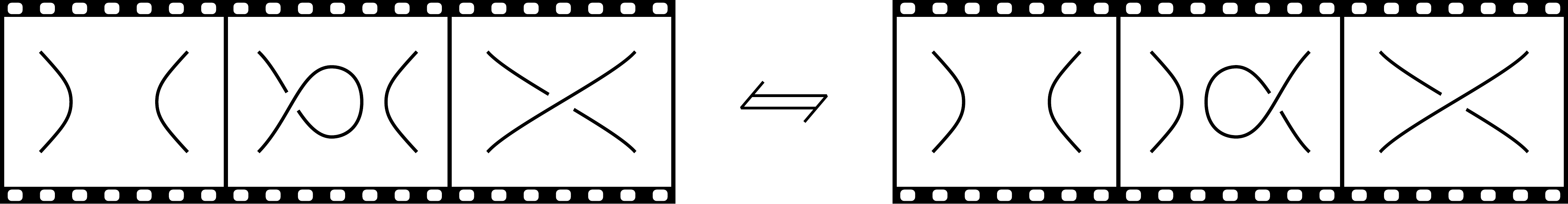}
\caption{}\label{fig:movie-move}
\end{figure}


Bar-Natan then shows that any two local movies that are related by movie moves induce the same morphisms in $\mathbf{Kob_{/h}}$ (up to multiplication by $\pm 1$). In some cases, this is achieved by directly calculating the induced morphisms; see Exercise~\ref{exer:movie-move} below for an example. For many cases, Bar-Natan takes a less direct approach, instead arguing that the local complexes associated to the tangles involved in a given movie-move have no automorphisms except $\pm \id$. This ensures that the two local movies induce the same automorphism (up to multiplication by $\pm 1$).

\begin{exercise}\label{exer:movie-move}
Show that the two local cobordisms depicted in Figure~\ref{fig:movie-move} induce the same morphism in $\mathbf{Kob_{/h}}$. (From left-to-right, as well as right-to-left.)


\hfill \emph{Hint: Verify (and apply) the so-called \emph{neck-cutting relation} shown in Figure~\ref{fig:neck}.}

\begin{figure}\center
  \def\svgwidth{.9\linewidth}
\begingroup%
  \makeatletter%
  \providecommand\color[2][]{%
    \errmessage{(Inkscape) Color is used for the text in Inkscape, but the package 'color.sty' is not loaded}%
    \renewcommand\color[2][]{}%
  }%
  \providecommand\transparent[1]{%
    \errmessage{(Inkscape) Transparency is used (non-zero) for the text in Inkscape, but the package 'transparent.sty' is not loaded}%
    \renewcommand\transparent[1]{}%
  }%
  \providecommand\rotatebox[2]{#2}%
  \newcommand*\fsize{\dimexpr\f@size pt\relax}%
  \newcommand*\lineheight[1]{\fontsize{\fsize}{#1\fsize}\selectfont}%
  \ifx\svgwidth\undefined%
    \setlength{\unitlength}{1913.27423336bp}%
    \ifx\svgscale\undefined%
      \relax%
    \else%
      \setlength{\unitlength}{\unitlength * \real{\svgscale}}%
    \fi%
  \else%
    \setlength{\unitlength}{\svgwidth}%
  \fi%
  \global\let\svgwidth\undefined%
  \global\let\svgscale\undefined%
  \makeatother%
  \begin{picture}(1,0.09778025)%
    \lineheight{1}%
    \setlength\tabcolsep{0pt}%
    \put(0,0){\includegraphics[width=\unitlength,page=1]{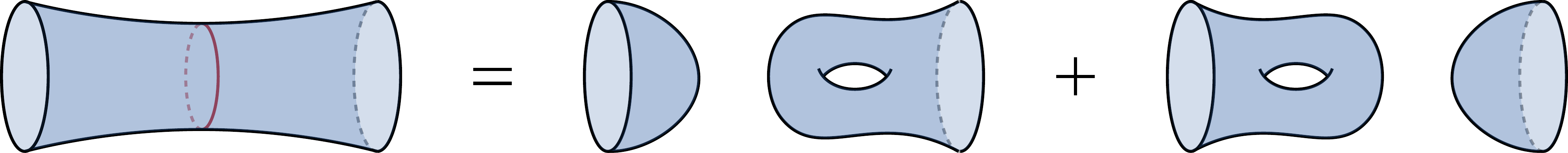}}%
    \put(-0.02366867,0.04129837){\color[rgb]{0,0,0}\makebox(0,0)[t]{\smash{\begin{tabular}[t]{c}{\normalsize$2$}\end{tabular}}}}%
  \end{picture}%
\endgroup%

\caption{The neck-cutting relation in the Bar-Natan category.}\label{fig:neck}
\end{figure}
\end{exercise}


Bar-Natan's strategy is carried out at the level of the category $\Kobh$ and therefore applies to any variant of Khovanov homology obtained by applying a compatible TQFT to $\Kobh$. In particular, by changing the Frobenius algebra as in \S\ref{subsec:bar-natan}, one  obtains invariance for Bar-Natan homology.


\smallskip
\smallskip

\emph{Formal quantum gradings and delooping}

When defining the formal complexes $\lb \hspace{1pt} \cdot \hspace{1pt} \rb$ above, we saw that a natural homological grading could be assigned to smoothings (i.e., planar tangles) in the Bar-Natan category, consistent with the homological grading previously defined \emph{after} applying the TQFT. A formal quantum grading can also be incorporated at this stage. We briefly recall this here, following the discussion in \cite[\S6]{barnatan} and \cite{morrison:local-slides}; see also  \cite{zhang:notes}.

\begin{remark}\label{rem:graded}
In a graded category, the collection of morphisms $\hom(A,B)$ between objects $A$ and $B$ must form a \emph{graded} abelian group, and the set of objects must admit a $\zz$-action that shifts gradings, written $(m,A)\mapsto A\{m\}$ for $m \in \zz$. Shifting gradings does not affect the morphisms themselves, so we have
$$\hom(A\{m\},B\{n\})=\hom(A,B)$$
as abelian groups. However, gradings are affected: if $f \in \hom(A,B)$ has degree $d$, then $f \in \hom(A\{m\},B\{n\})$ has degree $d+n-m$. If the objects have no intrinsic grading (as is the case with $\Cobt$), we may view each underlying object $A$ as having grading 0 and let $A\{m\}$ denote a copy with its grading shifted by $m \in \zz$. \hfill $\diamond$
\end{remark}

\smallskip

Given a fixed set $B$ consisting of an even number of distinct points in $\partial D^2$, we wish to upgrade $\Cobt(B)$ to a graded category.  As in Remark~\ref{rem:graded}, we may extend the set of objects (i.e., 1-manifolds $\sigma$ in $D^2$ with $\partial \sigma=B$) to include grading-shifted copies of the objects (i.e., $\sigma\{m\}$ for $m \in \zz$). Given a cobordism $\Sigma$ between objects $\sigma,\sigma' \in \Cobt(B)$,  there is a natural notion of degree given by
$$\deg(\Sigma) := \chi(\Sigma)-\tfrac{1}{2}|B|$$
when $\Sigma$ is viewed as an element of $\hom(\sigma,\sigma')$. (Here $|B|$ denotes the number of points in $B$.) Thus, when viewed as a morphism in $\hom(\sigma\{m\},\sigma'\{n\})$, we have 
$$\deg(\Sigma)= \chi(\Sigma)-\tfrac{1}{2}|B|+n-m.$$



\begin{exercise}
    Show that the degree is additive under composition of cobordisms in $\Cobt(B)$ (whenever the composition is defined).
\end{exercise}


These choices can be extended to Bar-Natan's category of formal complexes $\Kobh(B)$. In particular, given a tangle diagram $T$ in $D^2$ with boundary $B$, we equip $\lb T\rb$ with formal quantum gradings by replacing each chain group $\lb T \rb^h$ with a copy $\lb T\rb^h\{h +(n_+ - n_-)\}$, where the latter has its quantum grading shifted as indicated. That is, we have
$$
\lb T \rb : \quad \quad \lb T \rb^{-n_-}\{n_+-2n_-\} \ \longrightarrow \ \cdots  \ \longrightarrow  \ \lb T \rb^{n_+}\{2n_+-n_-\}.
$$

\begin{exercise}
Show that all differentials in $\lb T \rb$ have degree zero with respect to the quantum grading.
\end{exercise}

The introduction of this formal quantum grading makes possible an elegant statement\footnote{This fact can be stated  without quantum gradings, albeit less informatively.} about removing closed components of smoothings, 
via a process known as \emph{delooping} \cite{bar-natan:fast} (c.f., \cite{morrison:delooping-slides}):

\smallskip

\begin{exercise}\label{exer:delooping}
    In $\Cobl$, prove there is an isomorphism $\hspace{1pt} \boldsymbol{\bigcirc} \hspace{1pt} \cong \hspace{1pt}  \emptyset \{-1\} \oplus \emptyset\{+1\}$. 


\hfill  \emph{Hint: The morphism out of $\hspace{1pt} \boldsymbol{\bigcirc} \hspace{1pt}$ should be a direct sum of two cobordisms to $\emptyset$.}
\end{exercise}

\medskip
\smallskip

\emph{Dotted cobordisms}

The final part of our discussion concerns another modification of the Bar-Natan category: \emph{dotted cobordisms} \cite[\S11]{barnatan}. The category $\Cobtd$ is defined to have the same objects as $\Cobt$ and to contain all of the morphisms from $\Cobt$, however we now allow these cobordisms to be decorated with a finite number of dots ($\bullet$). Each dot shifts the degree of the cobordism by $-2$ and can be moved freely within its connected component of the cobordism. A quotient category $\Cobld$ is formed by taking $\Cobtd$ modulo the relations in Figure~\ref{fig:dotted-relations}. As before, we can obtain Khovanov homology by applying an appropriate TQFT to $\Cobld$; at the chain level, a dot on a cobordism has the effect of multiplication by $\bfx$.

\begin{figure}[t]
\center
\def\svgwidth{\linewidth}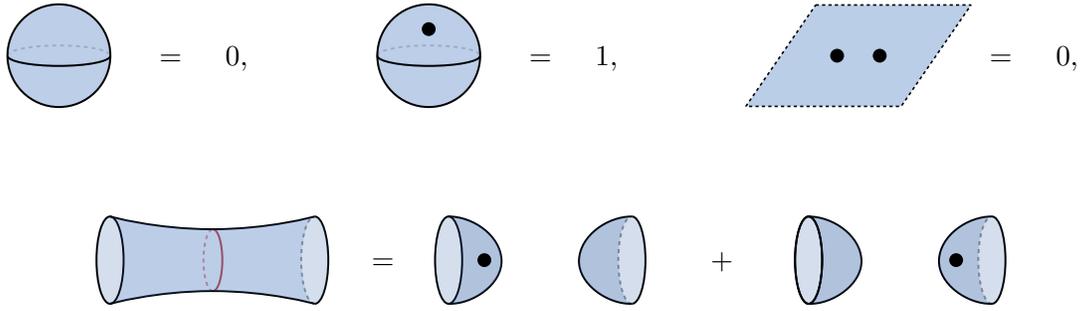
\caption{Relations in Bar-Natan's dotted cobordism category.}\label{fig:dotted-relations}
\end{figure}

\begin{exercise}
    Show that the relations in Figure~\ref{fig:dotted-relations} imply the torus relation $(T)$ and 4-tube relation $(\Tu)$.
\end{exercise}

\begin{exercise}
    Reprove the delooping isomorphism (Exercise~\ref{exer:delooping}) in the dotted cobordism category.
\end{exercise}

We close with an application to ribbon concordance. Given a pair of knots $K$ and $K'$ in $S^3$, recall that a \emph{concordance} from $K$ to $K'$ is a smoothly embedded annulus $C \subset S^3 \times [0,1]$ cobounded by $K \subset S^3 \times 0$ and $K' \subset S^3 \times 1$. We say $C$ is \emph{ribbon} if the projection to $[0,1]$ restricts to a Morse function on $C$ that has no local maxima (i.e., only index 0 and index 1 critical points). 

In \cite{zemke:ribbon}, Zemke showed that any ribbon concordance induces an injective map on knot Floer homology. This was later shown to hold for Khovanov homology by Levine-Zemke \cite{levine-zemke} (see also \cite{kang:ribbon}).   
Zemke's key topological input is the following description of the doubled cobordism $ \bar C \circ C$, where $\bar C$ is obtained by reversing $C$ in time.

\begin{lemma}[Zemke] Let $C$ be a ribbon concordance from $K$ to $K'$ as above. There exist unknotted, unlinked 2-spheres $S_1,\ldots,S_n$ in $(S^3 \setminus K) \times [0,1]$ and disjointly embedded 3-dimensional 1-handles $h_1,\ldots,h_n$ in $S^3 \times [0,1]$ such that 
\begin{itemize}
\item $h_i$ joins $K \times [0,1]$ to $S_i$,
\item  $h_i$ is disjoint from $S_j$ for each $j \neq i$, and
\item $ \bar C \circ C$ is obtained from $(K \times [0,1]) \cup S_1\cup \cdots \cup S_n$ by surgery along $h_1,\ldots,h_n$.
\end{itemize}
\end{lemma}
An example is depicted in Figure~\ref{fig:zemke} where $K$ is the right-handed trefoil, $K'$ is the knot $8_{11}$, and the ribbon concordance $C: K \to K'$ has a single 0-handle and 1-handle.

\begin{figure}[t]
\center
\includegraphics[width=.9\linewidth]{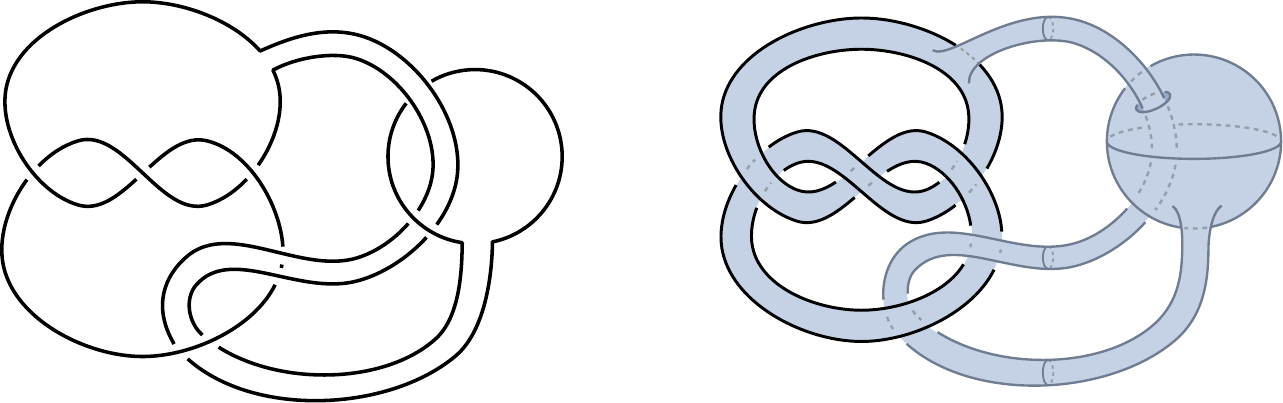}
\caption{(Left) The knot $8_{11}$, drawn to suggest a ribbon concordance $C$ from the right-handed trefoil knot to $8_{11}$. (Right) The double $\bar C \circ C$.}\label{fig:zemke}
\end{figure}

With this setup, the reader can recover the main result of \cite{levine-zemke}:

\begin{exercise}
Use the above to show that any ribbon concordance induces an injective map on Khovanov homology and Bar-Natan homology.
\end{exercise}

Observe that if $\Sigma$ and $\Sigma'$ are  cobordisms in $S^3 \times [0,1]$ from $K$ to $K'$ that are distinguished by their induced maps on Khovanov homology, and if $C$ is a ribbon concordance from $K'$ to $K''$, then it immediately follows that $C \circ \Sigma$ and $C \circ \Sigma'$ are distinguished by their induced maps on Khovanov homology.

{\small \footnotesize \bibliographystyle{alphamod} \bibliography{hayden}}

\newcommand{\etalchar}[1]{$^{#1}$}
\begin{thebibliography}{HKM{\etalchar{+}}22}

\bibitem[Ali19]{alishahi:unknotting}
{\bfseries Alishahi}, \href {https://doi.org/10.2140/pjm.2019.301.15} {{\em
  Unknotting number and {K}hovanov homology}}, Pacific J. Math., 301(1):15--29,
  2019.

\bibitem[BN05]{barnatan}
{\bfseries Bar-Natan}, \href {https://doi.org/10.2140/gt.2005.9.1443} {{\em
  Khovanov's homology for tangles and cobordisms}}, Geom. Topol., 9:1443--1499,
  2005.

\bibitem[BN07]{bar-natan:fast}
{\bfseries Bar-Natan}, {\em Fast {K}hovanov homology computations}, Journal of
  Knot Theory and Its Ramifications, 16(03):243--255, 2007.

\bibitem[BS16]{baykur-sunukjian}
{\bfseries Baykur, Sunukjian}, {\em Knotted surfaces in 4-manifolds and
  stabilizations}, Journal of Topology, 9(1):215--231, 2016.

\bibitem[CDGW]{snappy}
{\bfseries Culler, Dunfield, Goerner, Weeks}, {\em Snap{P}y, a computer program
  for studying the geometry and topology of $3$-manifolds}, Available at
  \url{http://snappy.computop.org}.

\bibitem[CP21]{conway-powell}
{\bfseries Conway, Powell}, {\em Characterisation of homotopy ribbon discs},
  Advances in Mathematics, 391:107960, 2021.

\bibitem[CP23]{conway-powell:cyclic}
{\bfseries Conway, Powell}, \href {https://doi.org/10.2140/gt.2023.27.739}
  {{\em Embedded surfaces with infinite cyclic knot group}}, Geom. Topol.,
  27(2):739--821, 2023.

\bibitem[CS98]{carter-saito}
{\bfseries Carter, Saito}, \href {https://doi.org/10.1090/surv/055} {{\em
  Knotted surfaces and their diagrams}}, volume~55 of {\em Mathematical Surveys
  and Monographs}, American Mathematical Society, Providence, RI, 1998.

\bibitem[DMS23]{dms:equivariant}
{\bfseries Dai, Mallick, Stoffregen}, {\em Equivariant knots and knot {F}loer
  homology}, Journal of Topology, 16(3):1167--1236, 2023.

\bibitem[Ell10]{elliott}
{\bfseries Elliott}, {\em State cycles, quasipositive modification, and
  constructing {H}-thick knots in {K}hovanov homology}, PhD thesis, Rice
  University, 2010, \textit{Proquest LLC, Ann Arbor, MI}, 2010.

\bibitem[Etn05]{etnyre:knot-intro}
{\bfseries Etnyre}, {\em Legendrian and transversal knots}, In {\em Handbook of
  knot theory}, pages 105--185, Elsevier B. V., Amsterdam, 2005.

\bibitem[Flo88]{floer:lagr}
{\bfseries Floer}, {\em Morse theory for {L}agrangian intersections}, J. Diff.
  Geom., 28:513--547, 1988.

\bibitem[Fox66]{fox:rolling}
{\bfseries Fox}, \href {https://doi.org/10.1090/S0002-9904-1966-11467-2} {{\em
  Rolling}}, Bull. Amer. Math. Soc., 72:162--164, 1966.

\bibitem[FS97]{fs:surfaces}
{\bfseries Fintushel, Stern}, \href
  {https://doi.org/10.4310/MRL.1997.v4.n6.a10} {{\em Surfaces in
  {$4$}-manifolds}}, Math. Res. Lett., 4(6):907--914, 1997.

\bibitem[GHKP23]{guth-hayden-kang-park}
{\bfseries Guth, Hayden, Kang, Park}, \href {http://arxiv.org/abs/2310.19713}
  {{\em Doubled disks and satellite surfaces}}, Preprint, arxiv:2310.19713,
  2023.

\bibitem[Gre21]{greene:floer}
{\bfseries Greene}, \href {https://doi.org/10.1090/noti2194} {{\em Heegaard
  {F}loer homology}}, Notices Amer. Math. Soc., 68(1):19--33, 2021.

\bibitem[Gut24]{guth}
{\bfseries Guth}, {\em For exotic surfaces with boundary, one stabilization is
  not enough}, J.~Eur.~Math.~Soc. (DOI 10.4171/JEMS/1541), arXiv:2207.11847,
  2024.

\bibitem[Hay21]{hayden:corks}
{\bfseries Hayden}, {\em Corks, covers, and complex curves}, arXiv:2107.06856,
  2021.

\bibitem[Hay23]{hayden:atomic}
{\bfseries Hayden}, \href {https://doi.org/10.48550/ARXIV.2302.10127} {{\em An
  atomic approach to {W}all-type stabilization problems}}, arXiv: 2302.10127,
  2023.

\bibitem[HKM{\etalchar{+}}22]{hkmps:seifert}
{\bfseries Hayden, Kim, Miller, Park, Sundberg}, {\em Seifert surfaces in the
  4-ball}, J.~Eur.~Math.~Soc. (to appear), arxiv:2205.15283, 2022.

\bibitem[HS24]{hayden-sundberg}
{\bfseries Hayden, Sundberg}, {\em Khovanov homology and exotic surfaces in the
  4-ball}, Journal f{\"u}r die reine und angewandte Mathematik (Crelles
  Journal), 2024(809):217--246, 2024.

\bibitem[Jac04]{jacobsson}
{\bfseries Jacobsson}, \href {https://doi.org/10.2140/agt.2004.4.1211} {{\em An
  invariant of link cobordisms from {K}hovanov homology}}, Algebr. Geom.
  Topol., 4:1211--1251, 2004.

\bibitem[JMZ21]{jmz:exotic}
{\bfseries Juh\'{a}sz, Miller, Zemke}, {\em Transverse invariants and exotic
  surfaces in the 4-ball}, Geom. Topol., 25(6):2963--3012, 2021.

\bibitem[JZ20]{juhasz-zemke:disks}
{\bfseries Juh{\'a}sz, Zemke}, {\em Distinguishing slice disks using knot
  {F}loer homology}, Selecta Mathematica, 26:1--18, 2020.

\bibitem[JZ24]{juhasz-zemke:stabilization}
{\bfseries Juh\'asz, Zemke}, \href {https://doi.org/10.1112/topo.12338} {{\em
  Stabilization distance bounds from link {F}loer homology}}, J. Topol.,
  17(2):Paper No. e12338, 112, 2024.

\bibitem[Kan22]{kang:ribbon}
{\bfseries Kang}, {\em Link homology theories and ribbon concordances}, Quantum
  Topology, 13(1):183--205, 2022.

\bibitem[Kau87]{Kauffman}
{\bfseries Kauffman}, \href {https://doi.org/10.1016/0040-9383(87)90009-7}
  {{\em State models and the {J}ones polynomial}}, Topology, 26(3):395--407,
  1987.

\bibitem[Kho00]{Khovanov}
{\bfseries Khovanov}, \href {https://doi.org/10.1215/S0012-7094-00-10131-7}
  {{\em A categorification of the {J}ones polynomial}}, Duke Math. J.,
  101(3):359--426, 2000.

\bibitem[Kho06]{khovanov:invariant}
{\bfseries Khovanov}, \href {http://www.jstor.org/stable/3845459} {{\em An
  invariant of tangle cobordisms}}, Trans. Amer. Math. Soc.,
  \textbf{358}(1):315--327, 2006.

\bibitem[KWZ19]{kwz:immersed}
{\bfseries Kotelskiy, Watson, Zibrowius}, \href
  {https://doi.org/10.48550/ARXIV.1910.14584} {{\em Immersed curves in
  {K}hovanov homology}}, Preprint, arXiv:1910.14584, 2019.

\bibitem[Lee05]{lee}
{\bfseries Lee}, {\em An endomorphism of the {K}hovanov invariant}, Adv. Math.,
  197(2):554--586, 2005.

\bibitem[Lit79]{litherland:deform}
{\bfseries Litherland}, \href {https://doi.org/10.2307/1998993} {{\em Deforming
  twist-spun knots}}, Trans. Amer. Math. Soc., 250:311--331, 1979.

\bibitem[Liv82]{livingston}
{\bfseries Livingston}, \href {http://projecteuclid.org/euclid.mmj/1029002727}
  {{\em Surfaces bounding the unlink}}, Michigan Math. J., 29(3):289--298,
  1982.

\bibitem[LS22]{lipshitz-sarkar:mixed}
{\bfseries Lipshitz, Sarkar}, \href {https://doi.org/10.1090/tran/8736} {{\em A
  mixed invariant of nonorientable surfaces in equivariant {K}hovanov
  homology}}, Trans. Amer. Math. Soc., 375(12):8807--8849, 2022.

\bibitem[LZ19]{levine-zemke}
{\bfseries Levine, Zemke}, {\em Khovanov homology and ribbon concordances},
  Bull. Lond. Math. Soc., 51(6):1099--1103, 2019.

\bibitem[Mor05a]{morrison:local-slides}
{\bfseries Morrison}, {\em Bar-natan's local {K}hovanov homology},
  \url{https://tqft.net/math/Bar-NatanLocalHomology_Slides.pdf}, June 2005.

\bibitem[Mor05b]{morrison:delooping-slides}
{\bfseries Morrison}, {\em Delooping},
  \url{https://tqft.net/math/Delooping_Slides.pdf}, September 2005.

\bibitem[MWW22]{morrison-walker-wedrich}
{\bfseries Morrison, Walker, Wedrich}, {\em Invariants of 4-manifolds from
  {K}hovanov-{R}ozansky link homology}, Geom. Topol., 26(8):3367--3420, 2022.

\bibitem[OS04a]{oz-sz:hfk}
{\bfseries Ozsv\'{a}th, Szab\'{o}}, \href
  {https://doi.org/10.1016/j.aim.2003.05.001} {{\em Holomorphic disks and knot
  invariants}}, Adv. Math., 186(1):58--116, 2004.

\bibitem[OS04b]{oz-sz:three}
{\bfseries Ozsv{\'a}th, Szab{\'o}}, \href
  {https://doi.org/10.4007/annals.2004.159.1027} {{\em Holomorphic disks and
  topological invariants for closed three-manifolds}}, Ann. of Math. (2),
  159(3):1027--1158, 2004.

\bibitem[{Ozs}06]{oz-sz:four}
{\bfseries {Ozsv{\'a}th, Peter and Szab{\'o}, Zolt{\'a}n}}, {\em {Holomorphic
  triangles and invariants for smooth four-manifolds}}, {Adv. Math.},
  {202}({2}):{326--400}, {2006}.

\bibitem[Per08]{perutz}
{\bfseries Perutz}, {\em Hamiltonian handleslides for {H}eegaard {F}loer
  homology}, In {\em Proceedings of {G}\"okova {G}eometry-{T}opology
  {C}onference 2007}, pages 15--35, G\"okova Geometry/Topology Conference
  (GGT), G\"okova, 2008.

\bibitem[Pla06]{plamenevskaya:transverse-Kh}
{\bfseries Plamenevskaya}, {\em Transverse knots and {K}hovanov homology},
  Math. Res. Lett., 13(4):571--586, 2006.

\bibitem[Ras03]{rasmussen:thesis}
{\bfseries Rasmussen}, {\em Floer homology and knot complements}, PhD thesis,
  Harvard University, arXiv preprint math/0306378, 2003.

\bibitem[Ras05]{rasmussen:closed}
{\bfseries Rasmussen}, {\em Khovanov's invariant for closed surfaces},
  Unpublished note; arXiv:0502527, 2005.

\bibitem[Ras10]{rasmussen:s}
{\bfseries Rasmussen}, {\em Khovanov homology and the slice genus}, Inventiones
  mathematicae, 182(2):419--447, 2010.

\bibitem[Sch23]{knotjob}
{\bfseries Sch\"{u}tz}, {\em {K}not{J}ob}, Available at
  \url{https://www.maths.dur.ac.uk/users/dirk.schuetz/knotjob.html}, 2023.

\bibitem[SS23]{sundberg-swann}
{\bfseries Sundberg, Swann}, {\em Relative {K}hovanov--{J}acobsson classes},
  Algebraic \& Geometric Topology, 22(8):3983--4008, 2023.

\bibitem[Sun23]{sundberg:github}
{\bfseries Sundberg}, {\em Khcob},
  https://github.com/imsundberg/imsundberg.github.io, 2023.

\bibitem[Sza19]{hfk-calc}
{\bfseries Szab{\'o}}, {\em {K}not {F}loer homology calculator}, Available at
  \url{https://web.math.princeton.edu/~szabo/HFKcalc.html}, 2019.

\bibitem[Tan06]{tanaka:closed}
{\bfseries Tanaka}, {\em Khovanov-{J}acobsson numbers and invariants of
  surface-knots derived from {B}ar-{N}atan's theory}, Proceedings of the
  American Mathematical Society, 134(12):3685--3689, 2006.

\bibitem[Zem19a]{zemke:ribbon}
{\bfseries Zemke}, \href {https://doi.org/10.4007/annals.2019.190.3.5} {{\em
  Knot {F}loer homology obstructs ribbon concordance}}, Ann. of Math. (2),
  190(3):931--947, 2019.

\bibitem[Zem19b]{zemke:gradings}
{\bfseries Zemke}, \href {https://doi.org/10.4171/QT/124} {{\em Link cobordisms
  and absolute gradings on link {F}loer homology}}, Quantum Topol.,
  10(2):207--323, 2019.

\bibitem[Zha25]{zhang:notes}
{\bfseries Zhang}, \href {http://arxiv.org/abs/2501.03115} {{\em Notes on
  {K}hovanov homology}}, Preprint, arXiv:2501.03115, 2025.

\end{thebibliography}

\end{document}